\documentclass{amsart}

\usepackage{amssymb}
\usepackage{graphicx,color}

\usepackage{esint}

\makeatletter
\@addtoreset{equation}{section}

\makeatother

\newtheorem{theorem}{Theorem}[section]
\newtheorem{proposition}[theorem]{Proposition}

\newtheorem{lemma}[theorem]{Lemma}
\theoremstyle{definition}
\newtheorem{definition}[theorem]{Definition}
\newtheorem{remark}[theorem]{Remark}

\newcommand{\R}{\mathbb{R}}

\begin{document}

\title[Sharp boundary $\varepsilon$-regularity of optimal transport maps]{Sharp boundary $\varepsilon$-regularity of optimal transport maps}
\author{Tatsuya Miura}
\address{Department of Mathematics, Tokyo Institute of Technology, 2-12-1 Ookayama, Meguro-ku, Tokyo 152-8551, Japan}
\email{miura@math.titech.ac.jp}
\author{Felix Otto}
\address{Max Planck Institute for Mathematics in the Sciences, Inselstra{\ss}e 22, 04103 Leipzig, Germany}
\email{felix.otto@mis.mpg.de}
%\subjclass[2010]{XXXXX, and XXXXX}

\date{\today}

\begin{abstract}
	In this paper we develop a boundary $\varepsilon$-regularity theory for optimal transport maps between bounded open sets with $C^{1,\alpha}$-boundary.
	Our main result asserts sharp $C^{1,\alpha}$-regularity of transport maps at the boundary in form of a linear estimate under certain assumptions: The main quantitative assumptions are that the local nondimensionalized transport cost is small and that the boundaries are locally almost flat in $C^{1,\alpha}$.
	Our method is completely variational and builds on the recently developed interior regularity theory.
\end{abstract}

\maketitle
\tableofcontents

\section{Introduction}

Let $d\geq2$, $\alpha\in(0,1)$, and $\Omega_0,\Omega_1\subset\mathbb{R}^d$ be bounded open sets with $C^{1,\alpha}$-boundary.
For $\lambda_0,\lambda_1\in[1/2,2]$, let $T$ be a unique solution to the optimal transport problem between constant densities $\lambda_0\chi_{\Omega_0}$ and $\lambda_1\chi_{\Omega_1}$ of same mass:
\begin{align}\label{eqn:minproblem}
	\min_{T\sharp(\lambda_0\chi_{\Omega_0})=\lambda_1\chi_{\Omega_1}}\int_{\mathbb{R}^d}|T(x)-x|^2\lambda_0\chi_{\Omega_0}dx,
\end{align}
where $\chi_{\Omega_i}$ denotes the characteristic function of $\Omega_i$ and, with an abuse of notation, $T\sharp(\lambda_0\chi_{\Omega_0})$ denotes the push-forward by $T$ of the measure $\lambda_0\chi_{\Omega_0}dx$.

In this paper we aim at clarifying how the regularity of the boundary $\partial\Omega_i$ influences boundary regularity of the map $T:\Omega_0\to\Omega_1$, focusing on the simplest case of constant densities.
To this end we develop a boundary $\varepsilon$-regularity theory on the intrinsic $L^2$-level, building on the interior $\varepsilon$-regularity theory by the last author and his co-authors \cite{GoldmanOtto,GoldmanHuesmannOtto}.

Our main result asserts sharp $C^{1,\alpha}$-regularity of $T$ near the boundary under certain assumptions, which roughly mean that $T$ is quantitatively ``close to being the identity'' as in usual $\varepsilon$-regularity theory.
Besides the main quantitative assumption, we assume two qualitative properties:
One is the {\em tangency condition} in an open ball $B_R(p)\subset\mathbb{R}^d$ that
\begin{align}\label{eqn:tangency}
	p\in\partial\Omega_0\cap\partial\Omega_1, \quad \nu_0(p)=\nu_1(p), \quad \partial\Omega_i\cap B_R(p)\ \textrm{is connected}\ (i=0,1),
\end{align}
where $\nu_i$ denotes the outer unit normal of $\partial\Omega_i$:
The other is the {\em topological condition} in $B_R(p)$ that
\begin{align}\label{eqn:topological}
	T(B_{R/2}(p)\cap\Omega_0)\subset B_{R}(p) \quad \textrm{and} \quad T^{-1}(B_{R/2}(p)\cap\Omega_1)\subset B_{R}(p),
\end{align}
where by abuse of notation we let $T^{-1}$ stand for the optimal transport map from $\lambda_1\chi_{\Omega_1}$ to $\lambda_0\chi_{\Omega_0}$ (since $T^{-1}(T(x))=x$ and $T(T^{-1}(y))=y$ hold a.e.\ \cite[Theorem 2.12 (iv)]{Villani}) and understand that the inclusions hold up to null sets.

Here is our main theorem, which asserts $C^{1,\alpha}$-regularity with a linear estimate:

\begin{theorem}[Boundary $\varepsilon$-regularity]\label{thm:epsilonregularity}
	There exist constants $\varepsilon,C>0$ depending only on $d,\alpha$ with the following property:
	Let $p\in\mathbb{R}^d$, $R>0$, $\lambda_0,\lambda_1\in[1/2,2]$ and $T$ be the solution of \eqref{eqn:minproblem}.
	If the tangency condition \eqref{eqn:tangency} and the topological condition \eqref{eqn:topological} hold in $B_R(p)$, and if\footnote{
	For a function $f:A\subset\mathbb{R}^n\to\mathbb{R}^m$ and a subset $B\subset\mathbb{R}^n$, we let $[f]_{\alpha,B}$ denote the $\alpha$-H\"older semi-norm $[f]_{\alpha,B}:=\sup\left.\left\{\frac{|f(x_1)-f(x_2)|}{|x_1-x_2|^\alpha} \right| x_1\neq x_2\in A \cap B \right\}$.
	}
	%Concerning $[\nu_i]_{\alpha,B_{R}(p)}$ and $[\nabla T]_{\alpha,B_{R/16}(p)}$, we take $X=\mathbb{R}^d$ for both, but $A=\partial\Omega_i$ for the former and $A=\Omega_0$ for the latter.}
	\begin{align}\label{eqn:smallness}
		\varepsilon':={1\over R^{d+2}}\int_{B_R(p)}|T-x|^2\lambda_0\chi_{\Omega_0}dx + R^{2\alpha}[\nu_0]_{\alpha,B_{R}(p)}^2+R^{2\alpha}[\nu_1]_{\alpha,B_{R}(p)}^2\leq\varepsilon,
	\end{align}
	then $T$ is of class $C^{1,\alpha}$ in $\overline{B_{R/16}(p)\cap\Omega_0}$, and we have the estimate
	$$R^{2\alpha}[\nabla T]_{\alpha,B_{R/16}(p)}^2 \leq C\varepsilon'.$$
\end{theorem}

Recall that the corresponding convex potential $u$ (i.e., $\nabla u=T$ a.e.) solves the Monge-Amp\`{e}re equation in Brenier's sense (cf.\ \cite[Section 4.1.4]{Villani}):
\begin{equation}\label{eq:MA}
	\det(\nabla^2u) =\lambda_0/\lambda_1 \quad \text{in }\Omega_0, \quad \nabla u(\Omega_0) =\Omega_1.
\end{equation}
Our proof is crucially based on the observation that in a quantitative sense the Monge-Amp\`{e}re equation may be linearized by a Poisson equation (cf.\ \cite[Exercise 4.1]{Villani}) with a certain (linearized) Neumann boundary condition on $\partial\Omega_0$.
%Indeed, we observe by letting $u=\frac{1}{2}|x|^2+\varepsilon \phi$ that $\det(\nabla^2u) \approx 1 + \varepsilon\Delta \phi$ (cf.\ \cite[Exercise 4.1]{Villani}), and the ``no-flux'' Neumann condition may be regarded as a linearization of the nonlinear boundary condition in \eqref{eq:MA}, since in the regular case the original boundary condition implies that $\nabla u(\partial\Omega_0)=\partial\Omega_1$ (``boundary goes to boundary'').

Theorem \ref{thm:epsilonregularity}, however, highlights that in view of regularity theory there is a significant difference between linear and nonlinear boundary conditions.
We first notice that in terms of the potential $u$, Theorem \ref{thm:epsilonregularity} asserts $u\in C^{2,\alpha}$ under $\partial\Omega_i\in C^{1,\alpha}$.
This should be compared with the fact that $\partial\Omega_i\in C^{2,\alpha}$ is necessary for $u\in C^{2,\alpha}$ in case of a Neumann boundary condition.
We briefly argue how this difference occurs, focusing on $d=2$.
Suppose locally that $\partial\Omega_0$ and $\partial\Omega_1$ coincide (but are different globally), and also the generic behavior that $\nabla u(\partial\Omega_0)=\partial\Omega_1$ (``boundary goes to boundary'').
Then the linearized boundary condition should be $\partial u/\partial\nu=0$, that is, $\gamma'=(\nabla u\circ\gamma)/|\nabla u\circ\gamma|$ (if $\nabla u\neq0$), where $\gamma$ denotes the unit speed parametrization of the boundary curve of $\Omega_0$; roughly speaking, this is a ``1st order = 1st order'' relation, so necessarily $\gamma\in C^{2,\alpha}$ if $u\in C^{2,\alpha}$.
However, the nonlinear boundary condition $\nabla u(\partial\Omega_0)=\partial\Omega_1$ means that $g(\partial_1 u(z,g(z)))=\partial_2u(z,g(z))$ if $\partial\Omega_0$ and $\partial\Omega_1$ are locally represented by the graph of $g$;
this is a ``0th order = 1st order'' relation, so just $g\in C^{1,\alpha}$ is required for $u\in C^{2,\alpha}$.

In Theorem \ref{thm:epsilonregularity}, the topological condition is unremovable in the sense that, without this condition, our $L^2$-smallness assumption does not rule out a ``separation of boundary layer'' as demonstrated in Remark \ref{rem:exseparation}.
This phenomenon represents one essential difference from the interior $\varepsilon$-regularity theory on $L^2$-level.
The tangency condition on the other hand is not restrictive thanks to the affine invariance of the optimal transportation problem (see Remark \ref{rem:tangency}).

Theorem \ref{thm:epsilonregularity} generalizes Chen and Figalli's boundary $\varepsilon$-regularity theory on $L^\infty$-level \cite[Theorem 2.2]{ChenFigalli} in a certain direction.
Their result in particular asserts $C^{1,\alpha'}$-regularity of $T$ for some (non-explicit) $\alpha'\in(0,1)$ when $\Omega_i$, $i=0,1$, are $C^2$-domains.
Compared to their result, on one hand, our assertion is stronger by giving the (explicit) sharp exponent $\alpha$, and on the other hand, our assumption is not only weaker but also described in a completely local way (see Remark \ref{rem:chenfigalli} for details).
We should however clarify that Chen and Figalli's method can deal with non-quadratic costs (resp.\ non-constant densities) as small perturbations from the quadratic cost (resp.\ the constant densities), which are not covered by our result.
We are confident that our approach is also applicable to non-constant densities as in \cite{GoldmanOtto}, and hope that it can be extended to non-quadratic costs, as in \cite{OttoProdhommeRied}.

Theorem \ref{thm:epsilonregularity} may be also regarded as an extension of Jhaveri's recent boundary $\varepsilon$-regularity result \cite[Theorem 2.6]{Jhaveri} \`{a} la Chen and Figalli, which asserts (lower order) $C^{0,\beta}$-regularity of $T$ for any $\beta\in(0,1)$, only assuming $C^{1,\alpha}$-regularity for domains.
In the same paper, Jhaveri gives several sharp counterexamples, in particular showing that for any small $\varepsilon>0$ there are smooth domain $\Omega_0$ and $\Omega_1$ that are $\varepsilon$-close in $C^1$-sense for which the optimal transport map is even discontinuous (in fact, they may be close in $C^{1,\alpha'}$-sense for small $\alpha'(\varepsilon)$).
Our assumption of $C^{1,\alpha}$-boundary regularity is thus sharp for establishing an $\varepsilon$-regularity theory.

The regularity of optimal transport maps is one of the most active area in the study of optimal transportation.
Although Brenier's theorem ensures existence and uniqueness of transport maps between general probability densities $\rho_0$ and $\rho_1$, their regularity sensitively depends on not only the regularity of $\rho_i$ but also the global geometry of ${\rm supp}\rho_i$ ($=\overline{\Omega_i}$).
Concerning the global regularity of $T$ for the quadratic cost, it is shown by Caffarelli \cite{Caffarelli2} that $T\in C^{0,\delta}(\overline{\Omega_0})$ for small $\delta>0$ when $\rho_i\in L^\infty(\Omega_i)$ is uniformly positive and $\Omega_i$ is bounded and convex for $i=0,1$; recently, Savin and Yu \cite{SavinYu} strengthen this regularity up to $C^{0,\delta}(\overline{\Omega_0})$ for any $\delta\in(0,1)$ in case of $d=2$ and constant densities.
Beyond the above lower order regularity, after Delano\"{e} \cite{Delanoe} and Urbas \cite{Urbas}, Caffarelli \cite{Caffarelli3} also shows that $T\in C^{1,\alpha}(\overline{\Omega_0})$ if in addition $0<\rho_i\in C^{0,\alpha}(\overline{\Omega_i})$ and $\Omega_i$ is $C^2$ uniformly convex for $i=0,1$.
In a recent study, Chen, Liu, and Wang \cite{ChenLiuWang1} ensure the same assertion even for $C^{1,1}$ convex domains (and small perturbations of them \cite{ChenLiuWang3}); when $d=2$, they also verify the same result for $C^{1,\alpha}$ convex domains \cite{ChenLiuWang2}.
The convexity assumption is known to be essential for global regularity (see e.g.\ \cite{Caffarelli1,MaTrudingerWang}); in this paper, in view of the local character of an $\varepsilon$-regularity result, we do not restrict our attention to convex domains.
Without convexity (or near convexity), the best known approach seems to be obtaining partial regularity of $T$ as developed in \cite{Figalli,FigalliKim,DePhilippisFigalli2,GoldmanOtto}.
In particular, the latter two studies by de~Philippis-Figalli \cite{DePhilippisFigalli2} and by Goldman and the last author \cite{GoldmanOtto} are straightforward consequences of an interior $\varepsilon$-regularity theory.
The aforementioned theory by Chen and Figalli \cite{ChenFigalli} is a boundary counterpart of de~Philippis and Figalli's interior theory, and seems to be a first breakthrough on boundary regularity without convexity (see also \cite{Jhaveri}).
From this point of view, our theory is a boundary counterpart of Goldman and the last author's interior theory \cite{GoldmanOtto}.
It is challenging and still open to establish boundary partial regularity, as already mentioned in \cite{ChenFigalli}.
We also refer the reader to a general exposition \cite{DePhilippisFigalli1}.

In summary, all the above results assume at least $C^{1,1}$-regularity of boundaries in order to obtain $C^{1,\alpha}$-type boundary regularity of $T$, except for $2$-dimensional convex domains, and some of the arguments crucially rely on convexity of domains.
Our theory requires only the optimal $C^{1,\alpha}$-regularity of boundaries in every dimension, and - as is natural for an $\varepsilon$-regularity theory - no condition on the global geometry of domains.
%(On a lower order level, Jhaveri shows $C^{0,\delta}$-regularity of $T$ for any $\delta\in(0,1)$ under $C^{1,\alpha}$-regularity of boundaries, \`{a} la Chen and Figalli's boundary $\varepsilon$-regularity theory.)

Our proof of Theorem \ref{thm:epsilonregularity} is purely variational, as opposed to most of the aforementioned results that view the regularity of optimal transportation as a specific instance of regularity of the Monge-Amp\`{e}re (type) equation and heavily use the maximum principle for the latter.
The general flow of our proof follows Goldman and the last author's interior theory \cite{GoldmanOtto}, which mimics de Giorgi's strategy for minimal surfaces (see e.g.\ \cite{Maggi}); namely, we compare the transport map with a harmonic gradient vector field, and obtain a key one-step improvement estimate, which is then iteratively used to conclude a Campanato-type estimate.
However, in all the steps we encounter several delicate issues coming from the boundaries (besides the aforementioned topological condition issue).
One of the main difficulties arises in the main harmonic approximation estimate (Proposition \ref{prop:harmoniclagrangian}), for which we need a new construction of a variational competitor near the boundaries; in the construction, we combine our new ideas with several techniques developed by Goldman, Huesmann, and the last author \cite{GoldmanHuesmannOtto} to deal with rough densities precisely.
Another delicate issue appears in the one-step improvement estimate (Proposition \ref{prop:onestep}), in which our construction of an affinely transformed (``tilted'') map has to allow for differing values of the two constant densities (see Remark \ref{rem:nonunitvalue}); this is one reason why we state our main result for constant densities of {\em different} values $\lambda_0$, $\lambda_1$ as opposed to the interior theory \cite{GoldmanOtto}.

We finally discuss the similarities and differences between the comparison-based
and the variational approaches to $\varepsilon$-regularity. We do this on the
level of the interior regularity by comparing \cite{DePhilippisFigalli2} and \cite{GoldmanOtto},
not addressing the additional difficulties in \cite{DePhilippisFigalli2} due to
a non-Euclidean cost functional; here, for the sake of this discussion, we deal with non-constant densities $\rho_0$ and $\rho_1$.
Both approaches to $C^{2,\alpha}$-regularity of the convex potential $u$ employ a Campanato
iteration, meaning that on every (dyadic) length scale, the corresponding regularity of $u$
is inferred from the inner regularity of a suitable solution $v$ of a reference equation.
In passing from one scale to the next smaller one, both approaches
crucially rely on the affine invariance of the Monge-Amp\`ere equation and optimal
transportation, respectively. Both approaches use Campanato's characterization of H\"older
semi-norms in terms of much weaker norms.
The most fundamental difference consists in the choice of the reference equation:
In case of \cite{DePhilippisFigalli2}, it is the Monge-Amp\`ere equation with constant
r.h.s.\ $\det\nabla^2v=1$; for \cite{GoldmanOtto} it is the Poisson equation
$\Delta\phi=\rho_1-\rho_0$ with r.h.s.~given by the difference of the (rescaled) densities $\rho_0$ and $\rho_1$
via $v=\phi+\frac{1}{2}|x|^2$.
The main advantage of this quantitative linearization of Monge-Amp\`ere by Poisson
in \cite{GoldmanOtto} is that the smallness assumption linearly appears in the upper bound:
\begin{align*}
	\varepsilon':= \Big(\int_{B_2}|\nabla u-x|^2\Big)^\frac{1}{2} + [\rho_0]_{\alpha,B_2}+[\rho_1]_{\alpha,B_2}\ll 1 \ \Longrightarrow \ [\nabla^2 u]_{\alpha,B_1}\lesssim
	\varepsilon',
\end{align*}
whereas in \cite{DePhilippisFigalli2}, under the equivalent smallness assumption one only obtains the less specific estimate:
\begin{align*}
	\sup_{B_2}\Big|u-\frac{1}{2}|x|^2\Big|+[\rho_0]_{\alpha,B_2}+[\rho_1]_{\alpha,B_2}\ll 1 \ \Longrightarrow \ [\nabla^2 u]_{\alpha,B_1}\lesssim 1.
\end{align*}
A further substantial difference is in how the closeness of $u$ and $v$ is captured:
In \cite{DePhilippisFigalli2}, closeness follows from the comparison principle
for the Monge-Amp\`ere equation and
thus is formulated in terms of $\sup|u-v|$. In \cite{GoldmanOtto}, the closeness
of $u$ and $v$ follows variationally by constructing a local competitor for the
transportation cost functional. In fact, one uses the harmonic gradient $\nabla\phi$ of
$\Delta\phi=\mathrm{const.}$ as a competitor for the displacement $\nabla u-x$.
As a consequence, the closeness is monitored in terms of
$(\int|\nabla u-\nabla v|^2)^\frac{1}{2}$.
In this respect, \cite{GoldmanOtto} is close in spirit to
de Giorgi's approach to $\varepsilon$-regularity for minimal surfaces, which also is based
on constructing a competitor of for the (nonlinear) area functional with help of
harmonic graphs.
%%%%%%%%%%%%%%%%
Another difference is that \cite{DePhilippisFigalli2} obtains $u\in C^{2,\alpha}$ in
three bootstrapping steps, as opposed to the single step in \cite{GoldmanOtto}:
In a first step, $u\in C^{1,\beta}$ for any $\beta<1$ is obtained
\cite[Theorem 4.3]{DePhilippisFigalli2}; in this qualitative
step, a $v$ close to $u$ is obtained by compactness \cite[Lemma 4.1]{DePhilippisFigalli2},
and the inner estimates for $v$ rely on the $\varepsilon$-regularity theory
for the Monge-Amp\`ere equation with constant densities \cite{FigalliKim}.
In a second more quantitative step, $u\in C^{1,1}$ is obtained
\cite[Theorem 5.3 up to last step]{DePhilippisFigalli2}; in this step,
$v$ is constructed by solving a boundary problem on a section of $u$
and quantitatively compared to $u$ by the comparison principle
\cite[Proposition 5.2]{DePhilippisFigalli2}. The previously established $C^{1,\beta}$-regularity
is used in order to leverage the H\"older continuity of the target density $g$
on the (possibly quite eccentric) image of the section under $\nabla u$.
In a third and short quantitative step, one obtains $C^{2,\alpha}$-regularity
of $u$ from the $C^{\alpha}$-regularity of $\rho_0$ and $\rho_1$,
since the image of the section is now known to be not too eccentric.

This paper is organized as follows.
In Section \ref{sect:preliminaries} we present key steps of the proof of Theorem \ref{thm:epsilonregularity}.
Section \ref{sect:Linftytheory} is devoted to $L^2$-$L^\infty$ theory.
In Sections \ref{sec:reduction} and \ref{sect:eulerian} we prove the main harmonic approximation estimate.
We finally conclude the H\"{o}lder regularity in Theorem \ref{thm:epsilonregularity} in Section \ref{sect:lagrangian}.

\subsection*{Acknowledgments}
The authors would like to thank the anonymous referees for their careful reading and constructive comments.
TM is deeply grateful to the Max Planck Institute for Mathematics in the Sciences for its hospitality: This research is mostly done when he was a postdoctoral fellow at the MPI MIS.
He is in part supported by JSPS KAKENHI Grant Numbers 18H03670 and 20K14341, and by Grant for Basic Science Research Projects from The Sumitomo Foundation.

\section{Preliminaries and main steps}\label{sect:preliminaries}

\subsection{Notation}

Throughout this paper we use the following notations.
The symbol $\lesssim$ (resp.\ $\gtrsim$) means that $\leq$ (resp.\ $\geq$) holds up to a universal constant $C$, where in this paper we call $C$ universal if it only depends on $d$ and $\alpha$ (if applicable); for example, $f\lesssim g$ means that $f\leq Cg$.
The symbol $\sim$ means that both $\lesssim$ and $\gtrsim$ hold.
An assumption of the form $f\ll1$ means that there is a universal constant $\varepsilon>0$ depending only on $d$ and $\alpha$ (if applicable) such that if $f\leq\varepsilon$, then the conclusion holds.
We also use notations like $\ll_{\beta}$, $\lesssim_{\beta}$, which mean that the universal constants in $\ll$, $\lesssim$ also depend on not only $d$ and $\alpha$ but also an additionally given parameter $\beta$.
We denote by $|E|$ the Lebesgue measure of $E\subset\mathbb{R}^d$, and by $\chi_{E}$ the characteristic function.
As in the introduction, $B_R(p)\subset\mathbb{R}^d$ stands for the open ball of radius $R>0$ centered at $p\in\mathbb{R}^d$, and abbreviate as $B_R$ if $p=0$.
We often drop the integral measures in integrals.
The following symbol means the average of a function $f$:
$$\fint_A f := \frac{1}{|A|}\int_{A} f.$$

\subsection{Well-preparedness}

We introduce the term {\em well-prepared} to shorten forthcoming statements.
In what follows, without loss of generality, we may focus on the case that $p=0$ by translation.
In addition, by re-defining as $\lambda_0':=1$ and $\lambda_1':=\lambda_1/\lambda_0$, we may consider only the case of $\lambda_0=1$ and $\lambda_1\in[1/4,4]$.

\begin{definition}[Well-prepared transport map]
	Let $R>0$.
	A map $T:\Omega_0\to\Omega_1$ is {\em well prepared in $B_R$} if $T$ is an optimal transport map from $\chi_{\Omega_0}$ to $\lambda\chi_{\Omega_1}$ for some value $\lambda\in[1/4,4]$, where $\Omega_0,\Omega_1\subset\mathbb{R}^d$ are open sets with $C^{1,\alpha}$-boundary, such that the tangency condition (\ref{eqn:tangency}) and the topological condition (\ref{eqn:topological}) hold in $B_R$.
\end{definition}

For an optimal transport map $T$ that is well prepared in $B_{R}$, we set
\begin{align}\label{wg12}
	E(\Omega_0,\Omega_1,T,R) &:=\frac{1}{R^2}\fint_{B_R}|T-x|^2\chi_{\Omega_0}dx,\\
	D(\Omega_0,\Omega_1,R) &:=R^{2\alpha}\left([\nu_0]_{\alpha,B_{R}(p)}^2+[\nu_1]_{\alpha,B_{R}(p)}^2\right).
\end{align}
The former is the localized and nondimensionalized $L^2$-cost of $T$, and the latter is determined by given data of densities, measuring deviation from being flat boundaries in $C^{1,\alpha}$-sense.
We often abbreviate them as $E_R$ or $D_R$: Note that the smallness assumption in Theorem \ref{thm:epsilonregularity} is equivalent to $E_R+D_R\ll1$.
We also often drop $R$ and use $E$ or $D$ when $R=1$: Assuming $R=1$ is not restrictive thanks to the scale invariance that
$$E(\Omega_0,\Omega_1,T,R)=E(\widetilde{\Omega}_0,\widetilde{\Omega}_1,\widetilde{T},1),\quad D(\Omega_0,\Omega_1,R)=D(\widetilde{\Omega}_0,\widetilde{\Omega}_1,1),$$
where $\widetilde{\Omega}_i:=R^{-1}\Omega_i$ ($i=0,1$) and $\widetilde{T}(\widetilde{x}):=R^{-1}T(R\widetilde{x})$ for $\widetilde{x}\in\widetilde{\Omega}_0$.

\subsection{Key ingredients}

We now explain the general flow of the proof of Theorem \ref{thm:epsilonregularity}, exhibiting several key steps.
%The big picture follows the strategy of Goldman-Otto \cite{GoldmanOtto}, which originates from De Giorgi's regularity theory for minimal surfaces, although we need several new ideas in boundary arguments.
%We first indicate that closeness of $\lambda$ to $1$ can be controlled by $E$; this fact is a simple consequence of the marginal condition but plays a key role.

The first key ingredient, Proposition \ref{prop:Linfty}, consists of several $L^2$-$L^\infty$ type estimates, which are frequently used throughout this paper.
Proposition \ref{prop:Linfty} is proved in Section \ref{sect:Linftytheory}: The importance of the topological condition is clarified in this part.

\begin{proposition}[$L^\infty$-bounds]\label{prop:Linfty}
	Let $T$ be an optimal transport map well prepared in $B_1$.
	If $E+D\ll1$, then
	\begin{align}\label{eqn:Linfty1}
		\sup_{x\in\Omega_0\cap B_{1/2}}|T(x)-x|+\sup_{y\in\Omega_1\cap B_{1/2}}|T^{-1}(y)-y|\lesssim E^{1\over d+2}+D^{1\over 2},
	\end{align}
	and for any $t\in[0,1]$, the map $T_t(x):=tT(x)+(1-t)x$ satisfies that
	\begin{align}\label{eqn:Linfty2}
		T_t(\Omega_0\cap B_{1/4})\subset B_{1/2}, \quad T_t^{-1}(\Omega_1\cap B_{1/4})\subset B_{1/2},
	\end{align}
	where $T_t^{-1}$ is interpreted in the sense of preimage.
\end{proposition}

Employing the $L^\infty$-bounds, we then establish the main estimate, which roughly means that the optimal transport map $T$ is well approximated by a harmonic gradient, that is, the gradient of a function of constant Laplacian; the proof is given in Sections \ref{sec:reduction} and \ref{sect:eulerian}.

\begin{proposition}[Harmonic approximation]\label{prop:harmoniclagrangian}
	For any $\varepsilon\in(0,1]$ there exists $\eta=\eta(d,\varepsilon)\in(0,1]$ with the following property:
	Let $T$ be an optimal transport map well prepared in $B_1$.
	If
	$$E+D\le\eta,$$
	then there exists a harmonic gradient $\nabla\phi$ on $\overline{B_{r}}$ such that
	\begin{align}
		\int_{B_{r}}|T-x-\nabla\phi|^2\chi_{\Omega_0}dx &\le \varepsilon E+\frac{C}{\varepsilon}D, \label{eqn:harmonicmainestimate}\\
		\int_{B_{r}}|\nabla\phi|^2 \leq CE, \label{eqn:harmonicDirichlet}
	\end{align}
	where $C=C(d)>0$ and $r=r(d)\in(0,1)$ depend only on $d$, and in addition, $\phi$ is symmetric with respect to the plane $\{x\cdot\nu=0\}$, i.e.,
	\begin{align}\label{eqn:harmonicsymmetry}
		\phi(x)=\phi(x+2(x\cdot\nu)) \quad \text{for all }x\in\overline{B}_r,
	\end{align}
	where $\nu:=\nu_0(0)\ (=\nu_1(0))$ denotes the outer normal at the origin.
\end{proposition}

The above harmonic approximation estimate is used for obtaining the so-called one-step improvement result, i.e., the quantitative closeness of $T$ at a scale $R$ is improved at a smaller scale $\theta R$ after an affine change of coordinates.
The main ingredient is the following affine invariance: For an optimal transport map $T:\Omega_0\to\Omega_1$ from $\chi_{\Omega_0}$ to $\lambda\chi_{\Omega_1}$, where $\lambda>0$, and for a matrix $B\in\R^{d\times d}$ and a vector $b\in\mathbb{R}^d$, if we let
\begin{align}\label{eqn:onestep0}
	\hat{T}(\hat{x}) &:=B(T(B^*\hat{x})-b),\\
	\hat{\Omega}_0 &:= B^{-*}\Omega_0,\quad \hat{\Omega}_1:=B(\Omega_1-b),\quad \hat{\lambda}:=\lambda|\det B|^{-2},\nonumber
\end{align}
where $B^*$ denotes the adjoint (transposed) matrix of $B$, and $B^{-*}:=(B^*)^{-1}$, then $\hat{T}$ is also the optimal transport map from $\chi_{\hat{\Omega}_0}$ to $\hat{\lambda}\chi_{\hat{\Omega}_1}$.
This invariance is a consequence of the fact that the optimality of $T$ is characterized by being the gradient of a convex potential and pushing the initial density forward to the target density, cf.\ \cite[Theorem 2.12]{Villani}.

\begin{proposition}[One-step improvement]\label{prop:onestep}
	For any $\beta\in(0,1)$ there exist constants $\varepsilon,\theta\in(0,1]$ and $C_\beta\geq1$ depending only on $d,\alpha,\beta$ with the following property:
	Let $R>0$ and $T$ be an optimal transport map well prepared in $B_R$, and suppose that
	\begin{equation}\label{onestep11}
		E_R+D_R\leq\varepsilon.
	\end{equation}
	Then there are $B\in \R^{d\times d}$ and $b\in\mathbb{R}^d$ satisfying
	\begin{align}\label{eqn:onestep1}
		|B-Id|^2+\frac{|b|^2}{R^2}\lesssim E_R+D_R,
	\end{align}
	such that the optimal transport map $\hat{T}(\hat{x}):=B(T(B^*\hat{x})-b)$ from $\chi_{\hat{\Omega}_0}$ to $\hat{\lambda}\chi_{\hat{\Omega}_1}$, where
	$$\hat{\Omega}_0:=B^{-*}\Omega_0,\quad \hat{\Omega}_1:=B(\Omega_1-b),\quad \hat{\lambda}:=|\det B|^{-2}\lambda,$$
	is well prepared in $B_{\theta R}$ and satisfies that
	\begin{align}
		\hat{E}_{\theta R} &:= E(\hat{\Omega}_0,\hat{\Omega}_1,\hat{T},\theta R) \leq \theta^{2\beta}E_R + C_\beta D_R, \label{eqn:onestep2}\\
		\hat{D}_{\theta R} &:= D(\hat{\Omega}_0,\hat{\Omega}_1,\theta R) \leq \theta^{2\alpha}(1+C\sqrt{E_R+D_R})D_R, \label{eqn:onestep8}
	\end{align}
	where $C=C(d,\alpha)>0$.
\end{proposition}

Using the one-step improvement result iteratively, we obtain the following estimate of Campanato type, which is a reformulation of $C^{1,\alpha}$-regularity on integral level.

\begin{proposition}[Campanato-type estimate]\label{prop:iteration}
	Let $R>0$ and $T$ be an optimal transport map well prepared in $B_R$, and suppose that
	$$E_R+D_R\ll 1.$$
	Then for any $r\in(0,R)$, there are $\bar{A}_r\in\R^{d\times d}$ and $\bar{a}_r\in\R^d$ such that
	\begin{align}\label{eqn:iteration1}
		\frac{1}{r^{2+2\alpha}}\fint_{B_r}|T-(\bar{A}_rx+\bar{a}_r)|^2\chi_{\Omega_0}\lesssim R^{-2\alpha}(E_R+D_R),
	\end{align}
	and
	\begin{align}\label{eqn:iteration9}
		|\bar{A}_r-Id|^2+\frac{|\bar{a}_r|^2}{R^2}\lesssim E_R+D_R.
	\end{align}
\end{proposition}

We finally deduce from Proposition \ref{prop:iteration}, the Campanato theory, and standard boundary arguments (as e.g.\ in Chen-Figalli \cite{ChenFigalli}) that the above Campanato-type estimate implies the desired pointwise $C^{1,\alpha}$-regularity in form of Theorem \ref{thm:epsilonregularity}.
All the statements from Proposition \ref{prop:onestep} are proved in Section \ref{sect:lagrangian}.

\section{Boundary $L^2$-$L^\infty$ estimates}\label{sect:Linftytheory}

In this section we prove Proposition \ref{prop:Linfty}, developing boundary $L^2$-$L^\infty$ theory in a slightly more general Lipschitz setting, still under the topological condition.
Let $\Omega_0$ and $\Omega_1$ be open sets, and $T$ be an optimal transport map between constant densities $\chi_{\Omega_0}$ and $\lambda\chi_{\Omega_1}$ with $\lambda\in[1/4,4]$.
Suppose that for $i=0,1$, $\Omega_i\subset\mathbb{R}^d$ is locally a Lipschitz half-space:
\begin{align}\label{eqn:Lip}
\Omega_i\cap B_1=\{(s,x')\in B_1 \mid s>g_i(x')\}\quad\mbox{with}\;|\nabla' g_i|\le 1/4,
\end{align}
where $\nabla'$ denotes the $(d-1)$ dimensional gradient.
In this setting we also denote by $E$ the $L^2$-energy in $B_1$, and introduce the width $\delta$ of the boundaries as
\begin{align*}
E :=\fint_{B_1}|T-x|^2\chi_{\Omega_0}dx, \quad \delta :=\|g_0\|_\infty+\|g_1\|_\infty.
\end{align*}
Note that if $\partial\Omega_i\in C^{1,\alpha}$ ($i=0,1$) and the tangency condition (\ref{eqn:tangency}) holds, then
\begin{align}\label{ao019}
	\delta^2\lesssim D
\end{align}
where $D$ is defined in Section \ref{sect:preliminaries}.
Hence, for the proof of Proposition \ref{prop:Linfty}, we may assume that $E^\frac{1}{d+2}+\delta\ll1$ instead of $E+D\ll1$.

%Proposition \ref{prop:Linfty} states that the $L^2$-smallness $E\ll1$ (together with the flatness $\delta\ll1$) controls $L^\infty$-quantities near boundaries, {\em under} the topological condition; without this condition, the $L^2$-smallness does not rule out large transport of a thin region near a boundary, which we call ``separation of boundary layer'' (see Remark \ref{rem:exseparation}).
%This condition plays a key role and reveals the main difference from the interior results by Goldman and Otto.
%In fact, their interior results ensure general $L^2$-$L^\infty$ estimates without such a condition, although
%In fact, unless an additional condition, the $L^2$-smallness does not rule out large transport of a thin region near a boundary, which we call ``separation of boundary layer''.
%The main difference is that the interior case the smallness of the $L^2$-energy always implies uniform control of $T(x)-x$ and $T^{-1}(y)-y$, although this is not the case near the boundaries, even for locally flat boundaries.
%roughly speaking, in a boundary layer region the $L^2$-smallness does not rule out large transportation, which corresponds to ``separation of boundary layer'', while except for the boundary layer we still have $L^2$-$L^\infty$ estimates of the same type as in the interior results.

In Section \ref{subsec:Linftygeneral} we first develop general $L^2$-$L^\infty$ theory without the topological condition.
More precisely, we establish $L^2$-$L^\infty$ estimates outside boundary layers of scale $E^{1\over d+2}$ (or $E^{1\over d+2}+\delta$), while we show that even inside the boundary layers, certain $L^2$-$L^\infty$ estimates hold except for the ``outward'' direction $-e_1$; this exception corresponds what we call a ``separation of boundary layer'' (see Remark \ref{rem:exseparation}).

In Section \ref{subsec:Linfty3} we observe that the topological condition (\ref{eqn:topological}) is in fact a simple sufficient condition for preventing separation of boundary layer, and complete the proof of Proposition \ref{prop:Linfty}.
We also show that the topological condition is always valid under a global assumption on the geometry of $\Omega_0$ and $\Omega_1$, which is imposed in the previous study by Chen and Figalli \cite{ChenFigalli}.

Throughout this section we argue in the framework of Kantorovich (see e.g.\ \cite[Theorem 2.12]{Villani}), since it is convenient for pointwise arguments.
Let $\pi$ be an optimal transference plan between $\chi_{\Omega_0}$ and $\lambda\chi_{\Omega_1}$; recall that $\pi$ is related to $T$ by $\pi=(Id\times T)\sharp\chi_{\Omega_0}$ (cf.\ \cite[Theorem 2.12]{Villani}).
The $L^2$-energy $E$ is then expressed as
\begin{align}\label{eqn:pienergy}
E = \frac{1}{|B_1|}\int_{B_1\times\mathbb{R}^d}|y_1-y_0|^2\pi(dy_0dy_1).
\end{align}
Recall the general fact that ${\rm supp}\pi\subset\overline{\Omega_0}\times\overline{\Omega_1}$, which follows from the marginal condition.
We note that for any $x_0\in\overline{\Omega_0}$ there exists $x_1\in\overline{\Omega_1}$ such that $(x_0,x_1)\in{\rm supp}\pi$
(this follows from the boundedness of $\Omega_1$).

\subsection{General theory of boundary $L^2$-$L^\infty$ estimates}\label{subsec:Linftygeneral}

We first prove forward $L^\infty$-bounds.

\begin{lemma}\label{lem:Linftyimage}
	Suppose \eqref{eqn:Lip} holds.
	If $E\ll1$, then the following estimates hold.
	\begin{itemize}
		\item[(1)] (Interior estimate.) For any $(x_0,x_1)\in{\rm supp}\pi$ such that $x_0\in B_{3/4}$ and $\mathrm{dist}(x_0,\partial\Omega_0)\geq E^{1\over d+2}$,
		\begin{align}\label{eqn:interiorbound}
			 |x_1-x_0| \lesssim E^{1\over d+2}.
		\end{align}
		\item[(2)] (Boundary estimate.) For any $(x_0,x_1)\in{\rm supp}\pi$ such that $x_0\in B_{3/4}$ and any $e\in \partial B_1$ such that $e\cdot e_1\geq1/2$,
		\begin{align}\label{eqn:globalbound}
			(x_1-x_0)\cdot e \lesssim E^{1\over d+2}.
		\end{align}
	\end{itemize}
\end{lemma}

\begin{proof}
	We start with a bit of elementary geometry:
	By the Lipschitz condition (\ref{eqn:Lip}), the open cone $C:=\{x\in\mathbb{R}^d\mid x\cdot e_1>|x|/3 \}$ satisfies that $(x_0+C)\cap B_1\subset\Omega_0$ for any $x_0\in\overline{\Omega_0}\cap B_{3/4}$.
	Clearly, for the strictly smaller closed cone $\tilde{C}:=\{x\in C\mid x\cdot e_1\geq|x|/2 \}$,
	there is a small number $0<\theta\le 1$, only depending on $d$,
	such that for any $e\in \widetilde C\cap\partial B_1$ we have $B_\theta(e)\subset C$.
	Hence for any point $x_0\in\mathbb{R}^d$, any direction $e\in\partial B$,
	and any radius $r>0$ we have
	\begin{align}\label{eqn:conecondition}
	\big(x_0\in\overline{\Omega_0}\cap B_{3/4},\ e\cdot e_1\geq1/2,\ r\le 1/8\big) \quad \Longrightarrow \quad B_{\theta r}(x_0+re)\subset\Omega_0\cap B_1.
	\end{align}

	Now, aiming at showing (\ref{eqn:globalbound}), we give ourselves a pair of points
	$(x_0,x_1)\in{\rm supp}\pi\cap(B_{3/4}\times\mathbb{R}^d)$; recall that $x_0\in\overline{\Omega_0}\cap B_{3/4}$.
	Then we obtain from the monotonicity of $\pi$ in form of $(y_1-x_1)\cdot(y_0-x_0)\ge 0$
	for all pair of points $(y_0,y_1)\in{\rm supp}\pi$ that
	\begin{equation}\label{eqn:001}
		\begin{split}
			(x_1-x_0)\cdot(y_0-x_0) & \le (y_1-y_0)\cdot(y_0-x_0) + |y_0-x_0|^2\\
			& \le \frac{1}{2}|y_1-y_0|^2+\frac{3}{2}|y_0-x_0|^2.
		\end{split}
  \end{equation}
	Given $r$ and $e$ as in (\ref{eqn:conecondition}),
	we integrate (\ref{eqn:001}) against $\pi$ over
	$(y_0,y_1)\in B_{\theta r}(x_0+re)\times\mathbb{R}^d$,
	obtaining
	by the marginal condition on $\pi$ and by (\ref{eqn:pienergy})
	\begin{align*}
	\lefteqn{(x_1-x_0)\cdot\int_{B_{\theta r}(x_0+re)}(y_0-x_0)\chi_{\Omega_0}dy_0}\nonumber\\
	&\leq\frac{1}{2}|B_1|E+\frac{3}{2}\int_{B_{\theta r}(x_0+re)}|y_0-x_0|^2\chi_{\Omega_0}dy_0,
	\end{align*}
	which by (\ref{eqn:conecondition}), by the elementary relations
	\begin{equation}\label{io11}
			\int_{B_{\theta r}(x_0+re)}|y_0-x_0|^2dy_0 \le |B_{\theta}|r^{d+2}, \quad \int_{B_{\theta r}(x_0+re)}(y_0-x_0)dy_0 = |B_{\theta}|r^{d+1}e,
	\end{equation}
	and by $\theta\sim1$ turns into
	\begin{align*}
	r^{d+1}(x_1-x_0)\cdot e
	\lesssim E+r^{d+2}.
  \end{align*}
	Taking $r=E^{1\over d+2}\ll1$, we obtain for all directions $e$ with $e\cdot e_1\geq1/2$,
	\begin{align}\label{eqn:002}
		(x_1-x_0)\cdot e \lesssim E^{1\over d+2}.
	\end{align}

	We finally prove (\ref{eqn:interiorbound}).
	If we additionally assume that $\mathrm{dist}(x_0,\partial\Omega_0)\geq E^{1\over d+2}$, then the inclusion in (\ref{eqn:conecondition}) holds even for any $e\in\partial B_1$ and $r={1\over 2}E^{1\over d+2}$.
	Hence, arguing as the above proof, we find that the resulting estimate (\ref{eqn:002}) holds for all directions $e$, thus obtaining (\ref{eqn:interiorbound}).
\end{proof}

We next prove similar estimates of backward type.

\begin{lemma}\label{lem:Linftypreimage}
	Suppose \eqref{eqn:Lip} holds.
	If $E^\frac{1}{d+2}+\delta\ll1$, then the following estimates hold.
	\begin{itemize}
		\item[(1)] (Interior estimate.) For any $(x_0,x_1)\in{\rm supp}\pi$ such that $x_1\in B_{1/2}$ and $x_1\cdot e_1 \gg E^{1\over d+2}+\delta$,
		\begin{align}\label{eqn:interiorbound2}
			 |x_0-x_1| \lesssim E^{1\over d+2},
		\end{align}
		where $c_0\geq1$ is the universal constant in Lemma \ref{lem:Linftyimage}.
		\item[(2)] (Boundary estimate.) For any $(x_0,x_1)\in{\rm supp}\pi$ such that $x_1\in B_{1/2}$, and any $e\in\partial B_1$ such that $e\cdot e_1\geq1/2$,
		\begin{align}\label{eqn:globalbound2}
			(x_0-x_1)\cdot e \lesssim E^{1\over d+2}+\delta.
		\end{align}
	\end{itemize}
\end{lemma}

At first glance, one would expect the estimates of the backward type to be identical to those
of the forward type, just with the indices $0$ and $1$ exchanged.
However, our assumption $E\ll1$, cf. \eqref{eqn:pienergy}, breaks that symmetry, so that the argument for the estimates of the backward type are slightly more subtle.

Below we first prove the main boundary estimate \eqref{eqn:globalbound2} by taking several steps, and in the last paragraph turn to the interior estimate \eqref{eqn:interiorbound2}.

\begin{proof}[Proof of Lemma \ref{lem:Linftypreimage}]
	Fix an arbitrary $(x_0,x_1)\in{\rm supp}\pi\cap(\mathbb{R}^d\times B_{1/2})$
	and an $e\in\partial B_1$ satisfying the weaker condition
	\begin{align}\label{io04}
	e\cdot e_1\ge\frac{1}{4}.
	\end{align}
	In line with the proof of Lemma \ref{lem:Linftyimage},
	we first argue that provided a radius $r$ satisfies
	\begin{align}\label{io01}
	\delta\ll r\ll 1,
	\end{align}
	we have the inclusion property
	\begin{align}\label{io05}
	B_{r/8}(x_1+re)\subset\Omega_0\cap B_1.
	\end{align}
	Notice that $x_1\in\overline{\Omega_1}\cap B_{1/2}$ by the marginal condition,
	so that by (\ref{eqn:Lip}) we have $x_1\cdot e_1\ge-\delta$. By (\ref{io04})
	this implies $(x_1+re)\cdot e_1\ge-\delta+\frac{r}{4}$, that is,
	$(x_1+re)\cdot e_1\ge\delta+\frac{r}{8}$ for $r\ge 16\delta$. This in turn yields
	$B_{r/8}(x_1+re)\subset\{x \mid x\cdot e_1>\delta\}$ for $r\ge 16\delta$.
	On the other hand, we have $B_{r/8}(x_1+re)\subset B_1$ for $r\le\frac{4}{9}$.
	Appealing once more to (\ref{eqn:Lip}) we obtain (\ref{io05})
	provided (\ref{io01}) holds in the specific form of $16\delta\le r\le\frac{4}{9}$.

	We now come to the central part and argue that for given $\varepsilon>0$,
	provided in addition to (\ref{io01}) we have
	\begin{align}\label{io09}
	E^\frac{1}{d+2}\ll_\varepsilon r,
	\end{align}
	there exists $\tilde e\in\partial B_1$ such that
	\begin{align}\label{io02}
	|\tilde e-e|<\varepsilon\quad\mbox{and}\quad (x_0-x_1)\cdot\tilde e\lesssim_\varepsilon r.
	\end{align}
	As in the proof of Lemma \ref{lem:Linftyimage}, we note that by monotonicity we have
	for any $(y_0,y_1)\in{\rm supp}\pi$:
	\begin{align}\label{io07}
	(x_0-x_1)\cdot(y_1-x_1)\le\frac{1}{2}|y_1-y_0|^2+\frac{3}{2}|y_0-x_1|^2.
	\end{align}
	To control the second l.h.s.~factor we record the elementary inequality
	\begin{align}\label{io08}
	|(y_1-x_1)-(y_0-x_1)|\le\frac{1}{2\tilde r}|y_1-y_0|^2+\frac{\tilde r}{2}
	\end{align}
	for some radius $\tilde r$ to be optimized.
	We integrate both inequalities (\ref{io07}) and (\ref{io08})
	with respect to $\pi(dy_0dy_1)$ restricted to
	$B_{r/8}(x_1+re)\times\mathbb{R}^d$.
	Then the integral of $|y_1-y_0|^2$ is bounded by $E$ by the second part of the inclusion (\ref{io05}) (``$\subset B_1$'') and by (\ref{eqn:pienergy}).
	On the terms involving $y_0-x_1$, we use the marginal condition
	followed by the first part of the inclusion (\ref{io05}) (``$\subset\Omega_0$''), and again appeal to (\ref{io11}) by replacing $\theta$ with $1/8$ and $x_0$ with $x_1$.
	Writing
	\begin{align*}
	\tilde e:=\frac{1}{|B_{1/8}|r^{d+1}}\int_{B_{r/8}(x_1+re)\times\mathbb{R}^d}(y_1-x_1)\pi(dy_0dy_1),
	\end{align*}
	we so obtain from (\ref{io07}) and (\ref{io08})
	\begin{align*}
	|B_{1/8}|r^{d+1}(x_0-x_1)\cdot\tilde e
	\le\frac{1}{2}E+\frac{3}{2}|B_{1/8}|r^{d+2},\\
	\big||B_{1/8}|r^{d+1}\tilde e-|B_{1/8}|r^{d+1}e\big|\le\frac{1}{2\tilde r}E
	+\frac{1}{2}|B_{1/8}|r^{d}\tilde r.
	\end{align*}
	Appealing to the assumption (\ref{io09}),
	this turns into (\ref{io02}) by choosing
	$\tilde r$ such that $E^\frac{1}{d+2}\ll_\varepsilon \tilde r\ll_\varepsilon r$.

	We now may conclude \eqref{eqn:globalbound2} by an elementary consideration:
	As usual, we call $\{\tilde e_1,\cdots,\tilde e_N\}\subset\mathbb{R}^d$
	an $\varepsilon$-net of the spherical cap
	$\{e\in\partial B_1 \mid e\cdot e_1\ge\frac{1}{4}\}$ if for any $e$ in the latter set there
	exists an $\tilde e$ in the former such that $|\tilde e-e|<\varepsilon$. Now (\ref{io02})
	can be rephrased in the following way: For a given $\varepsilon>0$,
	there exists an $\varepsilon$-net $\{\tilde e_1,\cdots,\tilde e_N\}$ such that
	\begin{align}\label{io03}
	(x_0-x_1)\cdot\tilde e_n\lesssim_{\varepsilon} E^\frac{1}{d+2}+\delta\quad\mbox{for}\;n=1,\cdots,N.
	\end{align}
	By elementary geometry, there exists $\varepsilon=\varepsilon(d)$
	such that the smaller spherical cap
	$\{e\in\partial B_1 \mid e\cdot e_1\ge\frac{1}{2}\}$ is contained in the
	convex hull of $\{\tilde e_1,\cdots,\tilde e_N\}\cup\{2\tilde e_1,\cdots,2\tilde e_N\}$,
	where $\{\tilde e_1,\cdots,\tilde e_N\}$ is an $\varepsilon$-net (for the larger spherical cap).
	This allows to
	pass from (\ref{io03}) to (\ref{eqn:globalbound2}) at the expense of a factor of two.

	Finally, we turn to the interior estimate (\ref{eqn:interiorbound2});
	we are given an $x_1\in B_{1/2}$ and momentarily fix an $e\in\partial B_1$.
	We first note that for a constant $1\le M<\infty$ to be fixed later, provided
	\begin{align}\label{io10}
	r\le M E^\frac{1}{d+2},\quad x_1\cdot e_1-\delta\gg_M E^\frac{1}{d+2},\quad E\ll_M 1,
	\end{align}
	we (easily) obtain the inclusion (\ref{io05}) from (\ref{eqn:Lip}).
	Based on this inclusion we showed above
	that for given $\varepsilon>0$, provided that in addition to (\ref{io10}) we have (\ref{io09}),
	we obtain (\ref{io02}). By the same $\varepsilon$-net argument as above,
	now applied to the simpler situation of the entire sphere instead of a spherical cap,
	eventually fixing an $\varepsilon$, this may be upgraded to
	\begin{align*}
	(x_0-x_1)\cdot e\lesssim r\quad\mbox{provided that}\ E^\frac{1}{d+2}\ll r \ \mbox{next to (\ref{io10})}.
	\end{align*}
	Fixing $M$ sufficiently large so that $r\le M E^\frac{1}{d+2}$ is not in conflict
	with $E^\frac{1}{d+2}\ll r$, we obtain (\ref{eqn:interiorbound2}) by the arbitrariness
	of $e\in\partial B$.
\end{proof}

\begin{remark}[Examples of separation of boundary layer]\label{rem:exseparation}
	We first consider a simple but important example in one dimension.
	Let $\Omega_0:=(0,2)$ and $\Omega_1:=(-1-\varepsilon,-1)\cup(0,2-\varepsilon)$, where $0<\varepsilon\ll1$.
	Notice that $\Omega_i\cap B_1=(0,1)$, and hence $\delta=0$.
	In this case, recalling that $T:\Omega_0\to\Omega_1$ is optimal if (and only if) $T\sharp\chi_{\Omega_0}=\chi_{\Omega_1}$ and $T=\nabla u$ a.e.\ for some convex function $u$, we can explicitly calculate the corresponding optimal transport map as
	\begin{align*}
		T(x)=
		\begin{cases}
			x-1-\varepsilon & {\rm for}\ x\in(0,\varepsilon),\\
			x-\varepsilon & {\rm for}\ x\in(\varepsilon,2).
		\end{cases}
	\end{align*}
	In particular, $E\lesssim \varepsilon$.
	We now find that in this example, $E\ll1$ and $\delta=0$ hold but the map $T$ has large transport near the origin.
	The topological condition prevents this kind of separation due to (one dimensional) disconnectedness, and thus we use the term ``topological''.

	The same idea also works in higher dimensions.
	Using the above notations, we define $\widetilde{\Omega}_i:=\Omega_i\times(-1,1)^{d-1}\subset\mathbb{R}^d$ for $i=0,1$; then the optimal transport map $\widetilde{T}:\widetilde{\Omega}_0\to\widetilde{\Omega}_1$ is expressed by $\widetilde{T}(x_1,\dots,x_d)=(T(x_1),x_2,\dots,x_d)$, and hence $\widetilde{T}$ has large transport near the flat boundary even though $E\ll1$ and $\delta=0$.
	%In addition, the inverse transport $\widetilde{T}^{-1}:\widetilde{\Omega}_1\to\widetilde{\Omega}_0$ is an example such that the transport map itself satisfies $L^\infty$-smallness in $B_1$ but its inverse does not.
\end{remark}

\subsection{Topological condition}\label{subsec:Linfty3}

In this subsection we observe as a corollary of the above lemmas that the topological condition (\ref{eqn:topological}) is a simple sufficient condition for full $L^\infty$-bounds.
In terms of $\pi$, the topological condition supposes that $\pi$ transports any $x_0\in B_{1/2}$ into $B_1$, and likewise,
any $x_1\in B_{1/2}$ into $B_1$, which means

\begin{equation}\label{eqn:003}
	\begin{split}
		&(x_0,x_1)\in B_1\times B_1 \\
		&\mbox{for all}\;
		(x_0,x_1)\in {\rm supp}\pi\cap((B_{1/2}\times\mathbb{R}^d)\cup(\mathbb{R}^d\times B_{1/2})).
	\end{split}
\end{equation}

\begin{proposition}\label{prop:topLinfty}
	Suppose \eqref{eqn:Lip} holds.
	Under the topological condition \eqref{eqn:003}, if $E^\frac{1}{d+2}+\delta\ll1$, then for any $(x_0,x_1)\in{\rm supp}\pi\cap((B_{1/2}\times\mathbb{R}^d)\cup(\mathbb{R}^d\times B_{1/2}))$,
	\begin{align*}
		|x_0-x_1|\lesssim E^{1\over d+2}+\delta.
	\end{align*}
\end{proposition}

\begin{proof}
	We only consider the case that $(x_0,x_1)\in{\rm supp}\pi\cap(\mathbb{R}^d\times B_{1/2})$ since the other case is similar.
	If $x_1\cdot e_1 \geq 4c_0E^{1\over d+2}+\delta$, then the interior estimate (\ref{eqn:interiorbound2}) in Lemma \ref{lem:Linftypreimage} implies the assertion, so we may assume that $x_1\cdot e_1 < 4c_0E^{1\over d+2}+\delta$.
	Now we notice that $x_0\cdot e_1\geq-\delta$ follows from (\ref{eqn:003}); indeed, (\ref{eqn:003}) implies that $x_0$ belongs to $\overline{\Omega_0}\cap B_1$, which by \eqref{eqn:Lip} is contained in $\{x \mid x\cdot e_1>-\delta \}$.
	We thus obtain
	\begin{align}\label{eqn:topLinfty2}
		(x_0-x_1)\cdot(-e_1)\lesssim E^{1\over d+2}+\delta.
	\end{align}
	On the other hand, the boundary estimate (\ref{eqn:globalbound2}) in Lemma \ref{lem:Linftypreimage} implies that
	\begin{align}\label{eqn:topLinfty1}
		(x_0-x_1)\cdot e\lesssim E^{1\over d+2}+\delta \quad \textrm{for any}\ e\in\partial B_1 \ \textrm{with} \ e\cdot e_1\geq1/2.
	\end{align}
	Combining (\ref{eqn:topLinfty1}) with (\ref{eqn:topLinfty2}), we complete the proof.
\end{proof}

We are now in a position to prove Proposition \ref{prop:Linfty}.

\begin{proof}[Proof of Proposition \ref{prop:Linfty}]
	Without loss of generality we may assume that $\nu_0(0)=\nu_1(0)=-e_1$.
	Note that since $\delta^2\lesssim D$, cf.\ (\ref{ao019}), the assumption $E+D\ll1$ implies that $E^\frac{1}{d+2}+\delta\ll1$, and also, together with (\ref{eqn:tangency}), the Lipschitz condition in (\ref{eqn:Lip}).
	Hence, estimate (\ref{eqn:Linfty1}) immediately follows from Proposition \ref{prop:topLinfty} and $\delta^2\lesssim D$.
	In addition, the first part of (\ref{eqn:Linfty2}) is a direct consequence of (\ref{eqn:Linfty1}).
	Therefore, we only need to confirm the last part of (\ref{eqn:Linfty2}), i.e., that under the topological condition (\ref{eqn:003}), if $E^{1\over d+2}+\delta\ll1$, then for any $t\in[0,1]$, all pairs $(x_0,x_1)\in{\rm supp}\pi$ such that $tx_1+(1-t)x_0\in B_{1/4}$ are contained in $B_{1/2}\times\mathbb{R}^d$.

	Consider $x'_0:=\delta e_1$ so that $x'_0\in\overline{\Omega_0}\cap B_{1/2}$ for $\delta\ll1$.
	Hence there is some $x_1'\in\overline{\Omega_1}$ such that $(x_0',x_1')\in{\rm supp}\pi$ and, by Proposition \ref{prop:topLinfty}, $|x_1'-x_0'|\lesssim E^{1\over d+2}+\delta$.
	In particular, since $|x_0'|=\delta$,
	\begin{equation}\label{eq:Linfty01}
		|x_1'|+|x_0'| \lesssim E^{1\over d+2}+\delta.
	\end{equation}
	We now fix an arbitrary $t\in[0,1]$ and a pair $(x_0,x_1)\in{\rm supp}\pi$ such that $|tx_1+(1-t)x_0|<1/4$.
	Then, using the monotonicity $(x_1-x_1')\cdot(x_0-x_0')\geq0$ and $E^{\frac{1}{d+2}}+\delta\ll1$, we find that there is a universal constant $C\geq1$ such that
	\begin{align*}
		\frac{1}{16}+C(E^{1\over d+2}+\delta) &\geq (|tx_1+(1-t)x_0|+|tx_1'+(1-t)x_0'|)^2\\
		&\geq |t(x_1-x_1')+(1-t)(x_0-x_0')|^2\\
		&\geq t^2|x_1-x_1'|^2+(1-t)^2|x_0-x_0'|^2 \quad \text{(by monotonicity)}\\
		&\geq \frac{1}{2}\min\{|x_1-x_1'|^2,|x_0-x_0'|^2\}.
	\end{align*}
	Thus we find that
	$$|x_i-x_i'|^2\leq \frac{1}{8}+2C(E^{1\over d+2}+\delta) \quad \textrm{for}\ i=0 \ \textrm{or}\ i=1,$$
	and hence, since $E^\frac{1}{d+2}+\delta\ll1$ and $|x_i'|\ll1$ by \eqref{eq:Linfty01}, we have either $|x_0|<2/5$ or $|x_1|<2/5$.
	In case that $|x_0|<2/5$ the assertion directly follows, while if $|x_1|<2/5$, then $|x_0|<1/2$ by Proposition \ref{prop:topLinfty} and hence the assertion follows.
\end{proof}

We conclude this section by indicating that the topological condition always holds if we additionally assume the global condition required in Chen and Figalli's study.

\begin{remark}[Comparison with Chen and Figalli's assumption]\label{rem:chenfigalli}
	Consider the condition that for $i=0,1$, the set $\Omega_i$ is globally contained in a Lipschitz half-space, i.e., there is an extension $\bar{g}_i$ of $g_i$ to $\mathbb{R}^{d-1}$ such that $\|\bar{g}_i\|_\infty\leq \delta$ and
	\begin{align}\label{eqn:globalcondition}
		\Omega_i \subset H_i:=\{(s,x')\in \mathbb{R}^d \mid s>\bar{g}_i(x')\}.
	\end{align}
	This global condition clearly prevents large transportation in the direction of $-e_1$.
	We then find in the almost same way as proving \eqref{eqn:globalbound} and \eqref{eqn:globalbound2} that, under the condition (\ref{eqn:globalcondition}) for $i=0$ (resp.\ $i=1$), if $E+\delta\ll1$, then for any $(x_0,x_1)\in{\rm supp}\pi\cap(B_{1/2}\times\mathbb{R}^d)$ (resp.\ $(x_0,x_1)\in{\rm supp}\pi\cap(\mathbb{R}^d\times B_{1/2})$),
	\begin{align}
		|x_0-x_1|\lesssim E^{1\over d+2}+\delta,
	\end{align}
	and in particular ${\rm supp}\pi\cap(B_{1/2}\times\mathbb{R}^d)$ (resp.\ ${\rm supp}\pi\cap(\mathbb{R}^d\times B_{1/2})$) is contained in $B_1\times B_1$, so that the forward (resp.\ backward) topological condition is satisfied.

	An important point is that the assumption of Theorem \ref{thm:epsilonregularity} is always satisfied under Chen and Figalli's assumption (within the framework of the quadratic cost and constant densities).
	More precisely, they assume, normalizing $R$ to $1$,
	\begin{itemize}
		\item[(i)] the tangency condition (\ref{eqn:tangency}),
		\item[(ii)] $(H_i\cap B_1(p))\subset\Omega_i\subset (H_i\cap B_{4}(p))$ with $\bar{g}_i\in C^2$,
		\item[(iii)] global $L^\infty$ and $C^2$ smallness: $\|T-x\|_{L^\infty(\Omega_0)}+\|\bar{g}_0\|_{C^2}+\|\bar{g}_1\|_{C^2}\ll1$.
	\end{itemize}
	It is obvious that (iii) implies our smallness assumption (\ref{eqn:smallness}), and from the above we see that (ii) (with (iii)) implies the topological condition (\ref{eqn:topological}).
	(Incidentally, it seems that Chen and Figalli's argument also works for $C^{1,\alpha}$-domains.)

	We also notice that the forward topological condition independently follows from the global $L^\infty$-smallness in (iii), although the backward one does not; in this sense, it now turns out that the assumption (ii) for $i=1$ is unremovable in Chen and Figalli's result.
	%If we assume the bidirectional $L^\infty$-smallness $$R^{-1}(\|T-x\|_{L^\infty(\Omega_0\cap B_R)}+\|T^{-1}-y\|_{L^\infty(\Omega_1\cap B_R)})\ll1$$ instead of the $L^2$-smallness, then Theorem \ref{thm:epsilonregularity} holds without the topological condition.
\end{remark}

\section{Reduction from Lagrangian to Eulerian}\label{sec:reduction}

For the harmonic approximation result (Proposition \ref{prop:harmoniclagrangian}) we mainly argue in the Eulerian formulation, following the interior theory \cite{GoldmanOtto}.
This section is devoted to demonstrating how to reduce Proposition \ref{prop:harmoniclagrangian} described in the Lagrangian coordinate into a statement in the Eulerian coordinate.
Section \ref{subsec:hypotheses} exhibits main hypotheses which we assume throughout the proof of harmonic approximation.
In Section \ref{subsec:Eulerianformulation} we define the Eulerian formulation of the optimal transport problem.
Finally, in Section \ref{subsec:reductionLE}, we give an Eulerian statement and prove that the given statement indeed implies \ref{prop:harmoniclagrangian}.

\subsection{Hypotheses}\label{subsec:hypotheses}

Since we are going to prove Proposition \ref{prop:harmoniclagrangian}, we give ourselves an arbitrary $\varepsilon\in(0,1]$.

For later purpose it is convenient to pass from balls $B_R$ of radius $R$ to cubes $Q_R=(-R,R)^d$ of half-side length $R$, both centered at the origin.
We first adapt our abbreviations for the transportation cost $E'$
and for the (squared) deviation of the boundaries from being flat $D'$ accordingly,
both of which we assume to be bounded by original quantities $E$ and $D$ in the larger ball $B_{32\sqrt{d}}$ (dropping $32\sqrt{d}$ by $E_{32\sqrt{d}}$ and $D_{32\sqrt{d}}$ to lighten notation):
\begin{align}
E'&:=\int_{(Q_{8}\times\mathbb{R}^d)\cup(\mathbb{R}^d\times Q_{8})}
|x_1-x_0|^2\pi(dx_0dx_1) \lesssim E \ll_\varepsilon 1,\label{ao73}\\
D'&:=\max_{i=0,1}[\nabla' g_i]^2_{\alpha,Q_{8}'} \lesssim D \ll_\varepsilon 1,\label{ao83}
\end{align}
where $Q_{R}':=(-R,R)^{d-1}$ and $g_i$ denotes the local graph representation of the $C^{1,\alpha}$-boundary $\partial\Omega_i$:
\begin{align}\label{ao17}
\Omega_i\cap Q_{8}=\{(s,x')\in Q_{8} \mid s>g_i(x')\}\quad\mbox{for}\;i=0,1,
\end{align}
satisfying the tangency condition, meaning
\begin{align}\label{ao82}
g_i(0)=\nabla' g_i(0)=0\quad\mbox{for}\;i=0,1.
\end{align}
Note that the estimate $E'\lesssim E$ in \eqref{ao73} is not a trivial hypothesis since $E'$ measures not only forward but also backward transports as opposed to $E$, but not restrictive thanks to the $L^\infty$-bounds (see Step 1 in the proof of Proposition \ref{prop:harmoniclagrangian} below).

In addition, we demote the Eulerian statement in the sense that all the necessary $L^\infty$-bounds are also supposed as hypotheses.
Namely, we repeatedly use the following $L^\infty$-bounds:
\begin{align}\label{ao53}
\begin{array}{c}
|x_1-x_0|\le M\ \mbox{with}\;M\lesssim E^\frac{1}{d+2}+D^\frac{1}{2}\\[1ex]
\mbox{for all}\ (x_0,x_1)\in
{\rm supp}\pi\cap((Q_{4}\times\mathbb{R}^d)\cup(\mathbb{R}^d\times Q_{4})).
\end{array}
\end{align}
Moreover, we will use the following consequence of monotonicity (and nondegeneracy of the densities):
\begin{align}\label{ao72}
\begin{array}{c}
\mbox{For any}\;(x_0,x_1)\in{\rm supp}\pi,\\[1ex]
\mbox{if}\; tx_1+(1-t)x_0\in Q_2\;\mbox{holds for some}\;t\in[0,1],\\[1ex]
\mbox{then}
\;(x_0,x_1)\in(Q_{4}\times\mathbb{R}^d)\cup(\mathbb{R}^d\times Q_{4}).
\end{array}
\end{align}

We finally suppose the unit-value hypothesis that
\begin{equation}\label{ao98}
	\lambda = 1,
\end{equation}
since the case of general $\lambda$ can be reduced to the unit value case with the help of Lemma \ref{lem:value} below (see Section \ref{subsec:reductionLE} for details).

\subsection{Eulerian formulation}\label{subsec:Eulerianformulation}

We now introduce the Eulerian formulation of the optimal transportation problem (see also \cite{GoldmanOtto,GoldmanHuesmannOtto}, or \cite[Section 5.4]{Villani}).
As is supposed in \eqref{ao98} we argue only for $\lambda=1$.

To clarify the meaning it is convenient to first introduce the ensemble of all trajectories. %(cf.\ \cite[Section 5.1]{Villani})
We define $\mathbb{P}$ as a non-negative measure on straight trajectories; namely, $\mathbb{P}$ is the push forward of the optimal transference plan $\pi$ under the map $F:\mathbb{R}^d\times\mathbb{R}^d \to \mathbb{X}$ that sends $(x_0,x_1)$ to the trajectory $X$ in form of
$X(t)=tx_1 + (1-t)x_0$ ($t\in[0,1]$), where $\mathbb{X}$ denotes the set of all straight trajectories equipped with e.g.\ the $L^\infty$-distance.
% $F$ is measurable (in fact continuous) so that the push forward is well defined, and also
Then, for all Borel function $\zeta$ on $\mathbb{R}^d$ and for all $t\in[0,1]$, we have
\begin{align}\label{ao40}
\int\zeta(tx_1+(1-t)x_0)\pi(dx_0dx_1) = \int\zeta(X(t))\mathbb{P}(dX)
\end{align}
by using the change of variables $\int G\circ F d\pi=\int G d\mathbb{P}$ for the measurable function $G:X \mapsto \zeta(X(t))$ on $\mathbb{X}$; in particular,
\begin{align}\label{ao41}
\int_{\Omega_0}\zeta=\int\zeta(X(0))\mathbb{P}(dX)\quad\mbox{and}\quad
\int_{\Omega_1}\zeta=\int\zeta(X(1))\mathbb{P}(dX).
\end{align}
This ``trajectory point of view'' will be repeatedly used in our proof.

We now introduce an ``Eulerian point of view'' by defining the pair $(\rho,j)$ as $(d\rho,dj):=(dt\rho_t(dx),dtj_t(dx))$ with a non-negative measure $\rho_t$ and an $\mathbb{R}^d$-valued measure $j_t$ on $\mathbb{R}^d$ for $t\in[0,1]$ so that
\begin{equation}\label{ao39}
	\begin{split}
		\int\zeta(x)\rho_t(dx) & =  \int\zeta(X(t))\mathbb{P}(dX),\\
		\int\xi(x)\cdot j_t(dx) & =  \int\xi(X(t))\cdot\dot X(t)\mathbb{P}(dX),
	\end{split}
\end{equation}
for all $\zeta\in C_c(\mathbb{R}^d)$ and $\xi\in C_c(\mathbb{R}^d;\mathbb{R}^d)$.
By definition \eqref{ao39} the pair $(\rho,j)$ solves the continuity equation $\partial_t\rho+\nabla\cdot j=0$ subject to the boundary condition $\rho_0=\chi_{\Omega_0}$ and $\rho_1=\chi_{\Omega_1}$ in the distributional sense, i.e.,
\begin{equation}\label{ao97}
	\begin{split}
		\int_{(0,1)\times\mathbb{R}^d} \partial_t\zeta d\rho + \nabla\zeta \cdot dj = \int_{\mathbb{R}^d}\zeta(1,\cdot)\chi_{\Omega_1} - \int_{\mathbb{R}^d}\zeta(0,\cdot)\chi_{\Omega_0}\\
		\mbox{for all}\ \zeta\in C^1_c([0,1]\times\mathbb{R}^d).
	\end{split}
\end{equation}
Note that in view of the advection equation, $\frac{j}{\rho}$ corresponds to the velocity field.

The pair $(\rho,j)$ has the following minimizing property, which plays a crucial role in our proof.
By the Benamou-Brenier formula (cf.\ \cite[Theorem 8.1]{Villani} and \cite[Chapter 8]{AmbrosioGigliSavare}) the pair $(\rho,j)$ defined through \eqref{ao39} solves the Eulerian formulation of the optimal transportation problem
\begin{equation}\label{ao36}
\min_{(\tilde{\rho},\tilde{j})}\left.\left\{ \int\frac{1}{\tilde{\rho}}|\tilde{j}|^2 \ \right| \
	(\tilde{\rho},\tilde{j})\ \mbox{satisfies}\ (\ref{ao97})
\right\},
\end{equation}
where we define the integrand by a dual formulation: For every finite measure $\tilde{\rho}$ and $\mathbb{R}^d$-valued measure $\tilde{j}$ on $[0,1]\times\mathbb{R}^d$,
\begin{equation}\label{ao96}
	\int\frac{1}{\tilde{\rho}}|\tilde{j}|^2 := \sup_{\xi\in C_c([0,1]\times\mathbb{R}^d;\mathbb{R}^d)}\left(\int 2\xi\cdot d\tilde{j} - \int |\xi|^2d\tilde{\rho}\right).
\end{equation}
An advantage of this Eulerian formulation is admitting singular measures (with respect to the Lebesgue measure) as competitors; in fact, in the proof of Lemma \ref{Lecomp} below, we will construct a measure $(\rho^\mathrm{sing},j^\mathrm{sing})$ that contributes to the competitor and is concentrated on the boundary, thus being singular.
Recall (cf.\ \cite[Proposition 5.18]{Santambrogio}) that if the energy \eqref{ao96} is finite, or equivalently if $\tilde{\rho}\geq0$ and $\tilde{j}\ll\tilde{\rho}$, then by using the Radon-Nikodym derivative (velocity) we have
$$\int\frac{1}{\tilde{\rho}}|\tilde{j}|^2 = \int_{(0,1)\times\mathbb{R}^d} \left|\frac{d\tilde{j}}{d\tilde{\rho}}\right|^2d\tilde{\rho},$$
and if in addition $\tilde{\rho}$ (and thus also $\tilde{j}$) is absolutely continuous with respect to the Lebesgue measure on $[0,1]\times\mathbb{R}^d$, then we have the pointwise understanding that
$$\int_{(0,1)\times\mathbb{R}^d}\frac{1}{\tilde{\rho}}|\tilde{j}|^2 = \int_{(0,1)\times\mathbb{R}^d} \frac{1}{\tilde{\rho}(t,x)}|\tilde{j}(t,x)|^2dtdx,$$
where if $\tilde{\rho}(t,x)=0$ and $\tilde{j}(t,x)=0$ (resp.\ $\tilde{j}(t,x)\neq0$), then we interpret the value of the integrand as $0$ (resp.\ $\infty$).
Note that if $(\rho,j)$ is a minimizer of \eqref{ao36}, then ($j_t\ll$) $\rho_t\ll\mathcal{L}^d$, where $\rho_t:=\rho(t,\cdot)$, $j_t:=j(t,\cdot)$, and $\mathcal{L}^d$ is the Lebesgue measure on $\mathbb{R}^d$ (cf.\ \cite[Section 3.1]{GoldmanOtto}), and hence the integrand in (\ref{ao36}) can be interpreted pointwise; this fact is a qualitative consequence of McCann's displacement convexity \cite{McCann} (see also \cite[Lemma 3.2]{GoldmanOtto}), which shows
that the trivial bounds on the initial and terminal data $\chi_{\Omega_0},\chi_{\Omega_1}\le 1$ are preserved:
\begin{align}\label{ao54}
\rho\le 1.
\end{align}
This quantitative result will also greatly simplify our proof in Section \ref{sect:eulerian}.

We finally remark that in our argument we do not appeal to the global minimality \eqref{ao36} but the local one (see Lemma \ref{lem:localoptimality} below for details).

\subsection{Reduction from Lagrangian to Eulerian}\label{subsec:reductionLE}

We are now in a position to state the main harmonic approximation result in terms of the Eulerian formulation.

\begin{proposition}[Harmonic approximation in a cube]\label{Prharm}
Given $\varepsilon\in(0,1]$, and under all the hypotheses from \eqref{ao73} to \eqref{ao98}, there exists a harmonic gradient $\nabla\phi$ in $\overline{Q_1}$ close to the velocity $\frac{j}{\rho}$
in the sense of
\begin{align}
\int_{(0,1)\times Q_1}\frac{1}{\rho}|j-\rho\nabla\phi|^2&\le
\varepsilon E+\frac{C}{\varepsilon}D,\label{ao88}\\
\int_{Q_1}|\nabla\phi|^2&\le CE,\label{ao81}
\end{align}
where $C>0$ only depends on $d$, and in addition symmetric in the sense of
\begin{align}\label{ao80}
\phi(s,x')=\phi(-s,x') \quad \textrm{for all}\ (s,x')\in Q_1.
\end{align}
\end{proposition}

The proof of Proposition \ref{Prharm} is given in Section \ref{sect:eulerian}.
Note that the l.h.s.~of \eqref{ao88} is well defined by the pointwise understanding as in Section \ref{subsec:Eulerianformulation} since if $\rho=0$ then $|j-\rho\nabla\phi|=|j|$.

In the remainder we prove that Proposition \ref{prop:harmoniclagrangian} indeed follows from Proposition \ref{Prharm}.
To this end we first verify a control of the value of the target density so that we will be able to remove the assumption $\lambda=1$, cf.\ \eqref{ao98}, by a simple scaling argument.

\begin{lemma}[Control of values]\label{lem:value}
	Let $T$ be an optimal transport map well prepared in $B_1$.
	If $D\ll1$, then $|\lambda-1|^2\lesssim E$.
\end{lemma}

\begin{proof}
	Without loss of generality we may assume that $\nu_i(0)=-e_1$.
	Thanks to the well-preparedness, the assumption $D\ll1$ implies that $B_{1/5}(\frac{1}{2}e_1)$ is compactly contained in $\Omega_0\cap\Omega_1\cap B_1$.
	Fix any $\eta=\eta(d)\in C^1_0(\mathbb{R}^d)$ such that
	\begin{align}
		\int_{\mathbb{R}^d}\eta &= 1, \label{ao004}\\
		{\rm supp}\eta &= \overline{B_{1/5}(\tfrac{1}{2}e_1)}\subset\Omega_0\cap\Omega_1\cap B_1.\label{ao005}
	\end{align}
	Then we have
	\begin{align*}
		\int_{(B_1\times\mathbb{R}^d)\cup(\mathbb{R}^d\times B_1)}(\eta(x_1)-\eta(x_0))\pi(dx_0dx_1)
		\stackrel{\eqref{ao005}}{=} \int_{\mathbb{R}^d\times\mathbb{R}^d}(\eta(x_1)-\eta(x_0))\pi(dx_0dx_1),
	\end{align*}
	and by the marginal condition the r.h.s.\ turns into
	\begin{align*}
		\int_{\mathbb{R}^d}\eta \lambda\chi_{\Omega_1}-\int_{\mathbb{R}^d}\eta\chi_{\Omega_0} \stackrel{\eqref{ao005}}{=} \lambda\int_{\mathbb{R}^d}\eta - \int_{\mathbb{R}^d}\eta \stackrel{\eqref{ao004}}{=} \lambda-1.
	\end{align*}
	Hence, noting that $|\eta(x_1)-\eta(x_0)|\leq\sup|\nabla\eta||x_1-x_0|\lesssim|x_1-x_0|$, we find that
	\begin{align*}
		|\lambda-1| &\lesssim \int_{(B_1\times\mathbb{R}^d)\cup(\mathbb{R}^d\times B_1)}|x_1-x_0|\pi(dx_0dx_1)\\
		&\lesssim E^{1/2}\left(\int_{(B_1\times\mathbb{R}^d)\cup(\mathbb{R}^d\times B_1)}\pi(dx_0dx_1)\right)^{1/2}\\
		%&\leq E^{1/2}\left(\pi(B_1\times\mathbb{R}^d)+\pi(\mathbb{R}^d\times B_1)\right)^{1/2}\\
		&= E^{1/2}\left(|B_1\cap\Omega_0|+\lambda|B_1\cap\Omega_1|\right)^{1/2} \lesssim E^{1/2},
	\end{align*}
	where the condition $\lambda\leq 4$ is used in the last estimate.
\end{proof}

We now deduce Proposition \ref{prop:harmoniclagrangian} from Proposition \ref{Prharm}.

\begin{proof}[Proof of Proposition \ref{prop:harmoniclagrangian}]
	Up to rotation we may assume that $\nu=-e_1$.
	We divide our proof into two steps.
	In Step 1 we first prove the unit-value case $\lambda=1$ by using Proposition \ref{Prharm}; a part of the proof is parallel to \cite{GoldmanOtto}.
	We then reduce the case of general $\lambda\in[1/4,4]$ to the unit-value case in Step 2.

	{\em Step 1: $\lambda=1$.}
	By rescaling we may construct a harmonic gradient in $B_{1/4}$ under the assumptions of Proposition \ref{prop:harmoniclagrangian} in the larger ball $B_{32\sqrt{d}}$, namely, the well-preparedness in $B_{32\sqrt{d}}$ with $\lambda=1$ and the smallness $E+D \ (=E_{32\sqrt{d}}+D_{32\sqrt{d}})\ll_\varepsilon1$.

	Given any $\varepsilon\in(0,1]$, we first check all the hypotheses in Proposition \ref{Prharm} (except for \eqref{ao98} since it is already assumed).
	Since all the $L^\infty$-bounds in Proposition \ref{prop:Linfty} hold in $B_{8\sqrt{d}}$ and hence in $Q_{8}$, we in particular find that $((Q_8\times\mathbb{R}^d)\cup(\mathbb{R}^d\times Q_8))\cap\mathrm{supp}\pi \subset B_{32\sqrt{d}}\times\mathbb{R}^d$ and thus \eqref{ao73} holds.
	It is now straightforward to check all the remaining hypotheses from (\ref{ao83}) to (\ref{ao98}).
	Therefore, we may apply Proposition \ref{Prharm} and deduce that there is a harmonic gradient $\nabla\phi$ on $Q_1$ with symmetry \eqref{ao80} satisfying \eqref{ao88} and \eqref{ao81}.

	We now prove that the restriction of $\nabla\phi$ to $B_{1/4}$ has the desired properties.
	Obviously, \eqref{eqn:harmonicDirichlet} follows by \eqref{ao81}, and \eqref{eqn:harmonicsymmetry} (with $\nu=-e_1$) by \eqref{ao80}.
	In the remainder we prove that $\nabla\phi$ satisfies \eqref{eqn:harmonicmainestimate} by translating \eqref{ao88} back into the Lagrangian coordinate.
	We first note that the $L^\infty$-bounds \eqref{ao53} and \eqref{ao72} imply that
	\begin{align}\label{ao009}
		T_t(B_{1/4}\cap\Omega_0)\subset B_{1/2},
	\end{align}
	where we recall $T_t:=tT+(1-t)Id$.
	By the triangle inequality and by $\phi=\phi|_{B_{1/4}}$,
	\begin{align}
		&\int_{B_{1/4}}|T-(x+\nabla\phi)|^2\chi_{\Omega_0} \nonumber\\
		&\qquad \lesssim \int_0^1\int_{B_{1/4}}|T-(x+\nabla\phi\circ T_t)|^2\chi_{\Omega_0} + \int_0^1\int_{B_{1/4}}|\nabla\phi-\nabla\phi\circ T_t|^2\chi_{\Omega_0}.\label{ao007}
	\end{align}
	We first estimate the former term in (\ref{ao007}).
	We infer from \eqref{ao39} that $\rho_t=T_t\sharp\chi_{\Omega_0}$ and $j_t=T_t\sharp[(T-Id)\chi_{\Omega_0}]$.
	The velocity field $v:=dj/d\rho$ satisfies that $v(T_t(x),t)=T(x)-x$ so that $T(x)-(x+\nabla\phi(T_t(x)))=(v(t,\cdot)-\nabla\phi)\circ T_t(x)$ holds for a.e.\ $x\in\Omega_0$.
	Hence, by definition of $\rho$ and interpretation of $\frac{1}{\rho}|j-\rho\nabla\phi|^2$ when $\rho=0$, we have
	\begin{align*}
		\int_0^1\int_{B_{1/4}}|T-(x+\nabla\phi\circ T_t)|^2\chi_{\Omega_0} &= \int_0^1\int_{T_t(B_{1/4}\cap\Omega_0)}|v-\nabla\phi|^2d\rho\\
		&= \int_0^1\int_{T_t(B_{1/4}\cap\Omega_0)}\frac{1}{\rho}|j-\rho\nabla\phi|^2\\
		&\stackrel{\eqref{ao009}}{\leq}\int_0^1\int_{B_{1/2}}\frac{1}{\rho}|j-\rho\nabla\phi|^2\\
		&\stackrel{\eqref{ao88}}{\leq}\varepsilon E+\frac{C}{\varepsilon}D.
	\end{align*}
	For the last term in (\ref{ao007}), we have
	\begin{align*}
		\int_0^1\int_{B_{1/4}}|\nabla\phi-\nabla\phi\circ T_t|^2\chi_{\Omega_0} \stackrel{\eqref{ao009}}{\lesssim} \sup_{B_{1/2}}|\nabla^2\phi|^2\int_0^1\int_{B_{1/4}}|T_t-x|^2\chi_{\Omega_0}.
	\end{align*}
	Recalling that $\sup_{B_{1/2}}|\nabla^2\phi|^2\lesssim E$ holds due to the mean-value property of harmonic functions and (\ref{ao81}), and using that $|T_t(x)-x|\leq|T(x)-x|$ so that
	\begin{equation*}
		\int_0^1\int_{B_{1/4}}|T_t-x|^2\chi_{\Omega_0}\leq\int_0^1\int_{B_{1/4}}|T-x|^2\chi_{\Omega_0}\lesssim E,
	\end{equation*}
	we find that the last term in (\ref{ao007}) is bounded of the form $\lesssim E^2$, thus being of higher-order for $E\ll_\varepsilon1$.
	Summarizing the above estimates, we have
	\begin{equation*}
		\int_{B_{1/4}}|T-(x+\nabla\phi)|^2\chi_{\Omega_0}\lesssim \varepsilon E+\frac{C}{\varepsilon}D.
	\end{equation*}
	Since $\varepsilon$ is arbitrary, we may replace $\lesssim$ by $\leq$ and thus obtain the desired bound.

	{\em Step 2: $\lambda\in[1/4,4]$.}
	Given a well-prepared map $T$ with an arbitrary $\lambda\in[1/4,4]$, we define a map $\widetilde{T}$ by $\widetilde{T}(x):=\lambda^{1/d}T(x)$ for $x\in\Omega_0$.
	Let $\widetilde{\Omega}_1:=\lambda^{1/d}\Omega_1$.
	Since $\widetilde{T}$ is still the gradient of a convex potential, and since $|\widetilde{\Omega}_1|=(\lambda^{1/d})^d|\Omega_1|=|\Omega_0|$, the map $\widetilde{T}$ is a well-prepared optimal transport map from $\chi_{\Omega_0}$ to $\chi_{\widetilde{\Omega}_1}$.
	In addition, the assumption $E+D\ll_\varepsilon 1$ for $T$ implies the same kind of smallness for $\widetilde{T}$; indeed, since $|\lambda-1|\ll1$, it is straightforward to check that (after a dyadic loss in the radius)
	\begin{equation}
		\widetilde{D}:=D(\Omega_0,\lambda^{1/d}\Omega_1,1/2)\lesssim D\ll_\varepsilon 1, \label{ao013}
	\end{equation}
	and also we have
	\begin{equation*}
		\widetilde{E}:=\fint_{B_{1/2}}|\widetilde{T}-x|^2\chi_{\Omega_0} \lesssim\int_{B_1\cap\Omega_0}(\lambda^{2/d}|T-x|^2+|\lambda^{1/d}-1|^2|x|^2) \lesssim  E+|\lambda-1|^2,
	\end{equation*}
	so that by Lemma \ref{lem:value},
	\begin{equation}
		\widetilde{E}\lesssim E\ll_\varepsilon 1. \label{ao012}
	\end{equation}
	Hence we deduce from Step 1 that for small $r\in(0,1/2)$ there is a harmonic gradient $\nabla\widetilde{\phi}$ on $\overline{B_r}$ that is symmetric, cf.\ (\ref{eqn:harmonicsymmetry}), and satisfies
	\begin{align}
		\int_{B_r}|\widetilde{T}-x-\nabla\widetilde{\phi}|^2\chi_{\Omega_0}dx &\le \varepsilon \widetilde{E}+\frac{C}{\varepsilon}\widetilde{D}, \label{ao010}\\
		\int_{B_r}|\nabla\widetilde{\phi}|^2 \leq C\widetilde{E}. \label{ao011}
	\end{align}
	Now we define $\phi(x):=\widetilde{\phi}(x)+(\lambda^{-1/d}-1)|x|^2/2$, the gradient of which $\nabla\phi(x)=\nabla\widetilde{\phi}(x)+(\lambda^{-1/d}-1)x$ is still harmonic and symmetric on $\overline{B_r}$.
	Then
	\begin{align*}
		\int_{B_r}|T-x-\nabla\phi|^2\chi_{\Omega_0} &= \int_{B_r\cap\Omega_0}|\lambda^{-1/d}\widetilde{T}-\lambda^{-1/d}x-\nabla\widetilde{\phi}|^2\\
		& \lesssim \lambda^{-2/d}\int_{B_r\cap\Omega_0}|\widetilde{T}-x-\nabla\widetilde{\phi}|^2+|\lambda^{-1/d}-1|^2\int_{B_r\cap\Omega_0}|\nabla\widetilde{\phi}|^2.
	\end{align*}
	By $\lambda\sim1$, (\ref{ao010}), (\ref{ao012}) and (\ref{ao013}), the first term is bounded as
	\begin{equation*}
		\lambda^{-2/d}\int_{B_r\cap\Omega_0}|\widetilde{T}-x-\nabla\widetilde{\phi}|^2\lesssim \varepsilon \widetilde{E}+\frac{C}{\varepsilon}\widetilde{D}\lesssim\varepsilon E+\frac{C}{\varepsilon}D,
	\end{equation*}
	while by Lemma \ref{lem:value} and (\ref{ao011}) the latter term is of higher order and in particular
	\begin{equation*}
		|\lambda^{-1/d}-1|^2\int_{B_r\cap\Omega_0}|\nabla\widetilde{\phi}|^2\lesssim E\widetilde{E} \stackrel{\eqref{ao012}}{\lesssim} \varepsilon E.
	\end{equation*}
	Therefore, by the arbitrariness of $\varepsilon$ we obtain (\ref{eqn:harmonicmainestimate}).
	Since (\ref{eqn:harmonicDirichlet}) follows from (\ref{ao011}) and (\ref{ao012}), and since the symmetry (\ref{eqn:harmonicsymmetry}) is already confirmed, the proof is now complete.
\end{proof}

\section{Proof of harmonic approximation}\label{sect:eulerian}

This section is devoted to the proof of Proposition \ref{Prharm}, i.e., the harmonic approximation on Eulerian level.
Throughout this section we give ourselves an arbitrary $\varepsilon\in(0,1]$, and assume all the hypotheses from (\ref{ao73}) to (\ref{ao98}).
In addition, we remark that in this section we will frequently use the notation $x=(x_1,x')=(x_1,x_2,\dots,x_d)\in\mathbb{R}^d$, in which $x_1$ does {\em not} mean the target point of $(x_0,x_1)\in{\rm supp}\pi$.

\subsection{Outline of the proof}\label{subsec:outline}

We mainly argue in a local region $(0,1)\times Q_R$, where the half-side length $R\in(1,2)$ is well chosen (in Lemma \ref{Legood}) so that all quantities on $\partial Q_R$ that we want to control behave in a generic way; below we drop $R$ for notational simplicity.

For Proposition \ref{Prharm} we will approximate the velocity $j/\rho$ by the gradient $\nabla\phi$ of the (symmetric) solution to a certain Poisson equation with a Neumann boundary condition, cf.\ \eqref{ao05}.
The first main step is Lemma \ref{Leorth}, which ensures an ``approximate orthogonality'' of the form
$$\int_{(0,1)\times Q}\frac{1}{\rho}|j-\rho\nabla\phi|^2\lesssim \int_{(0,1)\times Q}\frac{1}{\rho}|j|^2- \int|\nabla\phi|^2+ \text{small error},$$
so that our problem is reduced to estimating the (local) cost of $(\rho,j)$ directly.
Thanks to this orthogonality and also the local optimality of $(\rho,j)$ (Lemma \ref{lem:localoptimality}), it suffices to construct a suitable competitor the cost of which is comparable with the Dirichlet energy of $\phi$ up to small error.
This will be done in the other main step, Lemma \ref{Lecomp}, in which we construct a (local) variational competitor $(\tilde{\rho},\tilde{j})$ (concentrated on $(0,1)\times\overline{Q}$) based on the solution $\tilde{\phi}$ to a slightly modified Poisson equation, cf.\ \eqref{ao10}, such that
$$\int_{(0,1)\times Q}\frac{1}{\rho}|j|^2 \leq \int\frac{1}{\tilde{\rho}}|\tilde{j}|^2\lesssim \int|\nabla\tilde{\phi}|^2+ \text{small error},$$
and also ensure that the Dirichlet energies of the two solutions are comparable:
$$\int|\nabla\tilde{\phi}|^2 \lesssim \int|\nabla\phi|^2 + \text{small error}.$$
As all the above errors are of the desired form $\varepsilon E+\frac{1}{\varepsilon}D$, we reach the assertion.

We now sketch the idea to construct a competitor in Lemma \ref{Lecomp}.
To this end it is convenient to introduce the width, as in Section \ref{sect:Linftytheory},
\begin{align}\label{ao02}
\delta:=\max_{i=0,1}\|g_i\|_{L^\infty(Q_8')}\ll_\varepsilon 1,
\end{align}
where the smallness follows since
\begin{align}\label{ao87}
\delta \stackrel{\eqref{ao82}}{\le} \max_{i=0,1}4^\alpha[\nabla' g_i]_{\alpha,Q_{8}}\stackrel{\eqref{ao83}}{\lesssim}
D^\frac{1}{2} \ll_\varepsilon 1.
\end{align}
In particular, it follows from (\ref{ao17}) that $Q_{8}\cap\Omega_i$ $\subset\{x_1>-\delta\}$ for $i=0,1$ so that by (\ref{ao53}), we obtain the support property
\begin{align}\label{ao16}
((0,1)\times Q_2)\cap {\rm supp}(\rho,j)\;\subset\;\{x_1>-\delta\}.
\end{align}
When it comes to the construction of a competitor, next to the basic construction obtained in the interior theory \cite{GoldmanOtto} (``main construction'' $(\rho^\mathrm{main},j^\mathrm{main})$) that takes care of the flux through $(0,1)\times\partial Q$,
we need a new construction near the boundary $\partial\Omega_i$ (``boundary construction'' $(\rho^\mathrm{bdry},j^\mathrm{bdry})$), which turns out to be explicit.
However, when it comes to the (adaptation of the) interior case, the flux across $(0,1)\times(\partial Q\cap\{|x_1|<\delta\})$ has to be treated separately.
More precisely, we keep the trajectories that cross $(0,1)\times(\partial Q\cap\{|x_1|<\delta\})$ (``kept trajectories'' $(\rho^\mathrm{kept},j^\mathrm{kept})$), following a strategy from
\cite{GoldmanHuesmannOtto}.
%Here, ``keeping'' means that the competitor is the sum of the contribution of these trajectories (restricted to $(0,1)\times Q$), and of a construction that takes care of the remaining initial, terminal, and boundary data; we then appeal to sub-additivity of the transportation cost $\int\frac{1}{\rho}|j|^2$ in $(\rho,j)$.
Let $\rho_0\le \chi_{\Omega_0}$ denote the
density of the initial position of these kept trajectories
and $\rho_1\le \chi_{\Omega_1}$ the one of the terminal position.
These modifications in the initial and terminal conditions from $\chi_{\Omega_0}$ and $\chi_{\Omega_1}$ to $\chi_{\Omega_0}-\rho_0$ and
$\chi_{\Omega_1}-\rho_1$ (all restricted to $Q$) require a construction that will be accommodated by initial and terminal layers, that is, in $(0,\tau)\times Q$ and $(1-\tau,1)\times Q$ respectively, of a thickness $\tau\ll_\varepsilon 1$ (``initial and terminal construction'' $(\rho^\mathrm{ini},j^\mathrm{ini})$, $(\rho^\mathrm{term},j^\mathrm{term})$); accordingly, the main and boundary construction will be accommodated by the remaining (main) layer $(\tau,1-\tau)\times Q$.
As a collateral damage from introducing the initial and terminal layer, also the flux through $((0,\tau)\cup(1-\tau,1))\times \partial Q$ has to be treated separately.
For this, we have to distinguish between exiting and entering trajectories (positive and negative flux) through $((0,\tau)\cup(1-\tau,1))\times \partial Q$:
Trajectories ``exiting early'', i.e., through $(0,\tau)\times \partial Q$, and ``entering late'', i.e., through $(1-\tau,1)\times \partial Q$, are treated alongside those crossing $(\tau,1-\tau)\times(\partial Q\cap\{|x_1|<\delta\})$, that is, they are kept and contribute to $\rho_0$ and $\rho_1$ (and hence to $(\rho^\mathrm{kept},j^\mathrm{kept})$).
The flux coming from trajectories ``entering early'' through $(0,\tau)\times \partial Q$, and the flux coming from trajectories ``exiting late'' through $(1-\tau,1)\times \partial Q$ are not kept but will be treated as a singular measure supported in $(0,1)\times\partial Q$ (``singular construction'' $(\rho^\mathrm{sing},j^\mathrm{sing})$) at no further cost.
In fact, the corresponding entering particles will just stay put till time $\tau$, and then be released at uniform rate over $(\tau,1-\tau)$; the corresponding exiting ones will be treated in a parallel way.

In summary, we will eventually construct a local competitor $(\tilde{\rho},\tilde{j})$ of the form
\begin{equation}\label{ao103}
	\begin{split}
		(\tilde{\rho},\tilde{j}) &:= (\rho^\mathrm{kept},j^\mathrm{kept})+(\rho^\mathrm{main},j^\mathrm{main})+(\rho^\mathrm{sing},j^\mathrm{sing}) \\
		& \qquad +(\rho^\mathrm{bdry},j^\mathrm{bdry})+(\rho^\mathrm{ini},j^\mathrm{ini})+(\rho^\mathrm{term},j^\mathrm{term}),
	\end{split}
\end{equation}
where $(\tilde{\rho},\tilde{j})$ satisfies the same boundary condition as $(\rho,j)$ (see Figure \ref{fig:domain}), and then appeal to sub-additivity of the cost functional.
The precise definitions of these constructions are given in Section \ref{subsect:competitor} (see also Figures there).

\begin{figure}[htbp]
	\begin{center}
		\includegraphics[width=125mm]{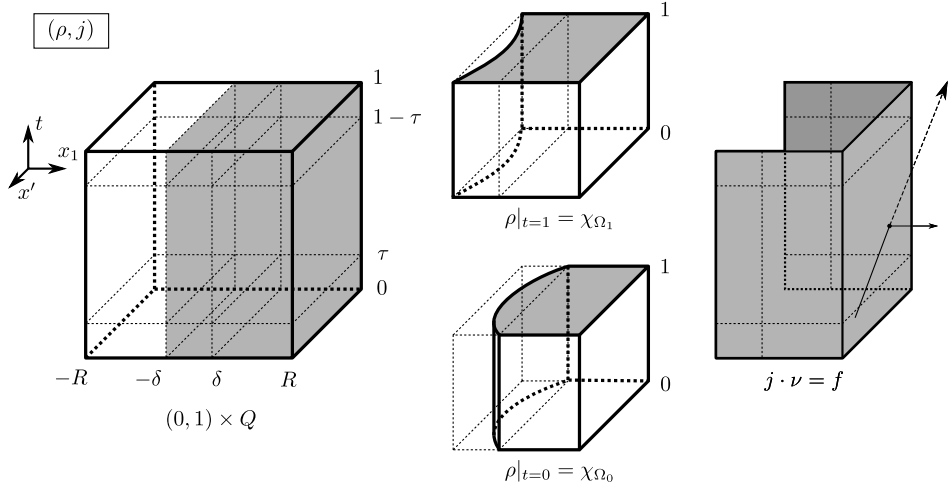}
		\caption{The Eulerian form $(\rho,j)$ of the optimal transport map in $(0,1)\times Q$.
		%The initial condition is $\chi_{\Omega_0}$, the terminal condition $\chi_{\Omega_1}$, and the normal flux $f$ through $(0,1)\times\partial Q$ (concentrated on $(0,1)\times(\partial Q\cap\{x_1>-\delta\})$).
		}
    \label{fig:domain}
	\end{center}
\end{figure}

\subsection{Preliminaries and key lemmas}\label{subsect:keylemmas}

In this section we make the above outline more rigorous by introducing fluxes and formulating the aforementioned key lemmas, and then demonstrate that Proposition \ref{Prharm} indeed follows from these lemmas.

For a given half-side length $R\in(1,2)$, we consider the subset of trajectories that
exit, and the one of those that enter $Q_R$:
\begin{equation}\label{ao67}
	\begin{split}
		{\mathcal T}_{+,R} &:= \{X \mid X(1)\not\in Q_R\ \mbox{and}\ X(t)\in Q_R\ \mbox{for some}\ t\in(0,1)\},\\
		{\mathcal T}_{-,R} &:= \{X \mid X(0)\not\in Q_R\ \mbox{and}\ X(t)\in Q_R\ \mbox{for some}\ t\in(0,1)\},
	\end{split}
\end{equation}
where here and hereafter $X$ denotes a straight trajectory (in $\mathbb{X}$).
Note that these sets are not necessarily disjoint.
For $X\in{\mathcal T}_{+,R}$ and $X\in{\mathcal T}_{-,R}$ we consider the exiting and
entering times, respectively,
\begin{equation}\label{ao71}
	\begin{split}
		t_{+,R}(X) &:= \sup\{t \in (0,1) \mid X(t)\in Q_R\} >0,\\
		t_{-,R}(X) &:= \inf\{t \in (0,1) \mid X(t)\in Q_R\} <1.
	\end{split}
\end{equation}
%
%Notice that if $X\in{\mathcal T}_{+,R}\cap{\mathcal T}_{-,R}$, then we always have
%\begin{equation}\label{ao99}
%	t_{+,R}(X)>t_{-,R}(X).
%\end{equation}
%
%It is an easy consequence of (\ref{ao39}) %and the continuity equation (\ref{ao36})
%that
We now define
the normal flux $f_R$ across $\partial Q_R$, as a (signed) measure on
$(0,1)\times\partial Q_R$ that exists for every (and not just almost every) $R$,
%may be represented as
by
\begin{equation}\label{ao37}
	\begin{split}
		& \int_{(0,1)\times\partial Q_R}\zeta df_R\\
		& :=\int_{{\mathcal T}_{+,R}}\zeta(t_{+,R},X(t_{+,R}))\mathbb{P}(dX)-\int_{{\mathcal T}_{-,R}}\zeta(t_{-,R},X(t_{-,R}))\mathbb{P}(dX),
	\end{split}
\end{equation}
where $\zeta\in C([0,1]\times\partial Q_R)$ is an arbitrary test function and, whenever it is not confusing, we write $t_{\pm,R}=t_{\pm,R}(X)$.
In fact, (\ref{ao37}) holds true for the positive and negative part of the normal flux separately:
\begin{align}\label{ao38}
\int_{(0,1)\times\partial Q_R}\zeta df_{R,\pm}
=\int_{{\mathcal T}_{\pm,R}}\zeta(t_{\pm,R},X(t_{\pm,R}))\mathbb{P}(dX).
\end{align}
In order to pass from (\ref{ao37}) to (\ref{ao38}), we need to show that the two measures on the r.h.s.~of (\ref{ao38}) are orthogonal (mutually singular).
Dropping the index $R$, this means that we have to show $(t_+(X_+),X_+(t_+(X_+)))$ $\not=(t_-(X_-),X_-(t_-(X_-)))$ for
$\mathbb{P}$-a.e.~$X_+\in{\mathcal T}_{+,R}$ and $X_-\in{\mathcal T}_{-,R}$.
In fact this holds for every $X_\pm\in{\mathcal T}_{\pm,R}$; we prove it by contradiction so suppose that $(t_+(X_+),X_+(t_+(X_+)))$ $=(t_-(X_-),X_-(t_-(X_-)))$ for some $X_\pm$; then, writing $X_\pm=tx_{1,\pm}+(1-t)x_{0,\pm}$, we would have $t(x_{1,+}-x_{1,-})$ $=-(1-t)(x_{0,+}-x_{0,-})$ for some $t\in(0,1)$;
by monotonicity of ${\rm supp}\pi$ in form of $(x_{1,+}-x_{1,-})\cdot(x_{0,+}-x_{0,-})$ $\ge 0$ this would imply $(x_{0,+},x_{1,+})=(x_{0,-},x_{1,-})$ so $X_+=X_-$;
however, this contradicts the assumption $t_+(X_+)=t_-(X_-)$ since $Q$ is open and hence $t_+(X)>t_-(X)$ holds for an arbitrary $X\in{\mathcal T}_{+,R}\cap{\mathcal T}_{-,R}$.

Using the flux $f_R$ defined above, we are now able to rigorously formulate the local optimality of $(\rho,j)$ playing a crucial role in our proof.

\begin{lemma}[Local optimality]\label{lem:localoptimality}
	For every $R\in(1,2)$ the measure $f_R$ defined by \eqref{ao37} coincides with the inner trace of $j$ on $(0,1)\times Q_R$ in the sense that for every $\zeta\in C^1_c([0,1]\times\mathbb{R}^d)$,
	\begin{equation}\label{ao100}
		\int_{(0,1)\times Q_R}\partial_t\zeta d\rho + \nabla\zeta\cdot dj = \int_{Q_R}(\zeta(1,\cdot)\chi_{\Omega_1} - \zeta(0,\cdot)\chi_{\Omega_0})dx + \int_{(0,1)\times\partial Q_R}\zeta df_R.
	\end{equation}
	In particular, if a pair of measures $(\tilde{\rho},\tilde{j})$ satisfies that for every $\zeta\in C^1_c([0,1]\times\mathbb{R}^d)$,
	\begin{equation}\label{ao101}
		\int_{(0,1)\times\mathbb{R}^d}\partial_t\zeta d\tilde{\rho} + \nabla\zeta\cdot d\tilde{j} = \int_{Q_R}(\zeta(1,\cdot)\chi_{\Omega_1} - \zeta(0,\cdot)\chi_{\Omega_0})dx + \int_{(0,1)\times\partial Q_R}\zeta df_R,
	\end{equation}
	then the following local optimality holds:
	\begin{equation}\label{ao102}
		\int_{(0,1)\times Q_R}\frac{1}{\rho}|j|^2 \leq \int_{(0,1)\times\mathbb{R}^d}\frac{1}{\tilde{\rho}}|\tilde{j}|^2,
	\end{equation}
	where the right-hand side is defined through the dual formula \eqref{ao96}, while the left-hand side is understood pointwise (by using the minimality of $(\rho,j)$).
\end{lemma}

We proceed with the preparation of our construction.
The trajectories we keep are those that cross $(0,1)\times\partial Q$ in $\{x_1<\delta\}$
or that exit before $\tau$, or enter after $1-\tau$:
\begin{align}
{\mathcal T}'_R &:= {\mathcal T}_{+,R}'\cup{\mathcal T}_{-,R}',\label{ao66}\\
&\mbox{where} \quad
\begin{cases}
	{\mathcal T}_{+,R}':=
	\{X\in{\mathcal T}_{+,R} \mid X_1(t_{+,R})<\delta \ \mbox{or} \ t_{+,R}<\tau\},\\
	{\mathcal T}_{-,R}':=
	\{X\in{\mathcal T}_{-,R} \mid X_1(t_{-,R})<\delta \ \mbox{or} \ t_{-,R}>1-\tau\}.
\end{cases}\nonumber
\end{align}
If we define measures $\rho'_R$ and $j'_R$ on $(0,1)\times\mathbb{R}^d$ analogously to (\ref{ao39}) (used later for defining ``kept trajectories'', cf.\ Figure \ref{fig:kept}), namely for $\zeta\in C(\mathbb{R}^d)$ and $\xi\in C(\mathbb{R}^d;\mathbb{R}^d)$,
\begin{equation}\label{ao114}
	\begin{split}
		\int\zeta(x)(\rho'_R)_t(dx) & :=  \int_{{\mathcal T}'_R}\zeta(X(t))\mathbb{P}(dX),\\
		\int\xi(x)\cdot (j'_R)_t(dx) & :=  \int_{{\mathcal T}'_R}\xi(X(t))\cdot\dot X(t)\mathbb{P}(dX),
	\end{split}
\end{equation}
then (analogously to \eqref{ao100}) $(\rho'_R,j'_R)=(dt(\rho'_R)_t(dx),dt(j'_R)_t(dx))$ distributionally satisfies the continuity equation in $(0,1)\times Q_R$ with initial condition $\rho_{0,R}$ at $t=0$, terminal condition $\rho_{1,R}$ at $t=1$, and normal flux $f'_R$ across $(0,1)\times\partial Q_R$ (as in (\ref{ao100})), where $\rho_{0,R}$ and $\rho_{1,R}$ are defined in line with (\ref{ao41}), that is, for $\zeta\in C_c(\mathbb{R}^d)$,
\begin{equation}\label{ao45}
	\begin{split}
		\int\zeta\rho_{0,R}&:=\int_{{\mathcal T}_{+,R}'}\zeta(X(0))\mathbb{P}(dX),\\
		\int\zeta\rho_{1,R}&:=\int_{{\mathcal T}_{-,R}'}\zeta(X(1))\mathbb{P}(dX),
	\end{split}
\end{equation}
and $f'_R$ is defined analogously to (\ref{ao37}), that is, for $\zeta\in C([0,1]\times\partial Q_R)$,
\begin{align}\label{ao55}
\lefteqn{\int_{(0,1)\times\partial Q_R}\zeta df_R'}\nonumber\\
&:=\int_{{\mathcal T}_{+,R}'}\zeta(t_{+,R},X(t_{+,R}))\mathbb{P}(dX)
-\int_{{\mathcal T}_{-,R}'}\zeta(t_{-,R},X(t_{-,R}))\mathbb{P}(dX).
\end{align}
Hence the remaining construction has to connect the initial condition $\chi_{\Omega_0}-\rho_{0,R}$ to the terminal condition
$\chi_{\Omega_0}-\rho_{0,R}$, with the normal flux $f_R-f'_R$.
What we have gained by discarding a (small) portion of the trajectories %(in conjunction with (\ref{ao16}))
is:
\begin{equation}\label{ao51}
	\begin{split}
		f_R-f'_R = 0\quad \mbox{on}\ \{x_1<\delta\}, \qquad f_R-f'_R
		\begin{cases}
			\le 0 & \mbox{on}\ \{t<  \tau\},\\
			\ge 0 & \mbox{on}\ \{t>1-\tau\}.
		\end{cases}
	\end{split}
\end{equation}
As mentioned above, we split this remaining construction into a construction in $(\tau,1-\tau)\times Q_R$,
into an ``initial construction'' in $(0,\tau)\times Q_R$,
and into a ``terminal construction'' in $(1-\tau,1)\times Q_R$.
The initial and terminal constructions require more precise information on $\rho_{0,R}$
and $\rho_{1,R}$.
To this purpose, we consider the subsets of those trajectories that exit early
or enter late:
\begin{equation}\label{ao70}
	\begin{split}
		{\mathcal T}_{+,R}'' &:= \{X\in{\mathcal T}_{+,R} \mid t_{+,R}<\tau\}\subset{\mathcal T}_{+,R}',\\
		{\mathcal T}_{-,R}'' &:= \{X\in{\mathcal T}_{-,R} \mid t_{-,R}>1-\tau\}\subset{\mathcal T}_{-,R}',
	\end{split}
\end{equation}
and the corresponding densities: for $\zeta\in C_c(\mathbb{R}^d)$,
\begin{equation}\label{ao64}
	\begin{split}
		\int\zeta\rho_{0,R}'&:=\int_{{\mathcal T}_{+,R}''}\zeta(X(0))\mathbb{P}(dX),\\
		\int\zeta\rho_{1,R}'&:=\int_{{\mathcal T}_{-,R}''}\zeta(X(1))\mathbb{P}(dX).
	\end{split}
\end{equation}
We finally note that the measures $f_R$ and $f_R'$ have Hausdorff densities on $(0,1)\times\partial Q_R$ for a.e.\ $R\in(1,2)$.
It follows from Lemma \ref{lem:localoptimality} that for every $R\in(1,2)$, the inner normal trace of $j$ (as a measure) exists and coincides with $f_R$.
Since in addition $j$ has a (square integrable) Lebesgue density, which we denote again by $j$, the Hausdorff density of this inner trace coincides with $\nu\cdot j$ for a.e.\ $R\in(1,2)$; in fact, this is true for every Lebesgue point $R$ of $j$, seen as an element of $L^2((1,2);L^2((0,1)\times\partial Q_R))$.
As a consequence, for a.e.\ $R\in(1,2)$, $\nu\cdot j$ is the Hausdorff density of $f_R$.
In addition, the square integrability (of the density) of $f_R$ obviously transmits to $f_R^{\pm}$. %Because they are the sign-decomposition (orthogonal).
We then learn from \eqref{ao38} that this transmits to the two r.h.s.\ expressions.
Since these clearly dominate the respective r.h.s.\ expressions in \eqref{ao45}, also $f_R'$ is square integrable.
%The sign decomposition is thus important to compare $f$ and $f'$, since otherwise some sign-cancelation in $f$ that do not occur in $f'$ may cause them non-comparable.

We are now in a position to choose a good half-side length $R$ (slice) such that
the quantities introduced above are well estimated, that is, behave like on average (with respect to $R$).
The proof is given in Section \ref{subsect:goodslices}.
\begin{lemma}[Good slices]\label{Legood}
For any fixed $\tau\in(0,1)$, there exists $R\in(1,2)$ such that
\begin{align}
\int_{(0,1)\times\partial Q_R}f_R^2&\lesssim E,\label{ao43}\\
\int_{(0,1)\times\mathbb{R}^d}\frac{1}{\rho'_R}|j'_R|^2&\lesssim ME,\label{ao42}\\
\rho_{0,R},\rho_{1,R}\le 1,\quad
\int_{\mathbb{R}^d}{\rm dist}(\cdot,\partial Q_R)(\rho_{0,R}+\rho_{1,R})&\lesssim E\label{ao44},\\
\rho_{0,R}',\rho_{1,R}'\le 1,\quad
\int_{\mathbb{R}^d}{\rm dist}(\cdot,\partial Q_R)(\rho_{0,R}'+\rho_{1,R}')&\lesssim \tau^2 E\label{ao46}.
\end{align}
\end{lemma}

We remark that in order to get the coefficient $\tau^2$ in (\ref{ao46}) we will use the fact that the initial points (resp.\ the target points) of the trajectories exiting early (resp.\ entering late) are ``close'' to the boundary $\partial Q_R$.
This might not be true for the trajectories exiting late and entering early; this is why we deal with those trajectories separately.

From now on, we drop the index $R$. In fact, for (notational) simplicity, we treat $R$
as being unity $R=1$.

%%%%%%%%%%%%%%%%%%%%%%%%%%%%%%%%%%%%%%%%%%%%%%%%%%%%%%%%%%%%%%%%%%%%%%%%%%%%%%%%%%%%

We now turn to our main estimates.
The Poisson equation we use for approximating the velocity $j/\rho$ is:
\begin{align}\label{ao05}
	\begin{cases}
		\Delta \phi=c & \mbox{in}\ Q\cap\{x_1>0\},\\
		\nu\cdot\nabla\phi = \bar f & \mbox{on}\ \partial Q\cap\{x_1>0\},\\
		\nu\cdot\nabla\phi=0 & \mbox{on}\ Q\cap\{x_1=0\},
	\end{cases}
\end{align}
where we have introduced the following abbreviation of the boundary flux, cf.~(\ref{ao37}) and (\ref{ao55}):
\begin{equation}\label{ao47}
	\bar f := \int_0^1(f-f')dt \quad \mbox{on}\ \partial Q,
\end{equation}
and $c$ is the constant that makes the problem solvable; throughout this section, $\nu$ denotes the outer normal of a domain under consideration.
Because $\partial_1\phi$ vanishes for $x_1=0$, we may extend $\nabla\phi$ harmonically onto $Q$ by reflection, cf.\ \eqref{ao80}.

The next lemma shows that in a certain sense,
$j-\rho\nabla\phi$ and $\nabla\phi$ are almost orthogonal, cf.~(\ref{ao11}).
The proof is given in Section \ref{subsect:orthogonality}.

\begin{lemma}[Approximate orthogonality]\label{Leorth}
Given $\tau\in(0,1/4)$, we have
\begin{equation}\label{ao11}
	\begin{split}
		\int_{(0,1)\times Q}\frac{1}{\rho}|j-\rho\nabla\phi|^2
		&-\left(\int_{(0,1)\times Q}\frac{1}{\rho}|j|^2
		-\int_{Q\cap\{x_1>-\delta\}}|\nabla\phi|^2\right)\\
		&\lesssim E^\frac{1}{2}\delta+(\tau+\delta)^\frac{1}{2}E,
	\end{split}
\end{equation}
and also
\begin{align}
\int_{Q}|\nabla\phi|^2  \lesssim E.\label{ao76}
\end{align}
\end{lemma}

%%%%%%%%%%%%%%%%%%%%%%%%%%%%%%%%%%%%%%%%%%%%%%%%%%%%%%%%%%%%%%%%%%%%%%%%%%%%%%%%%%%%

We then turn to the construction of a competitor, which is based on the solution to the following modified equation:
\begin{align}\label{ao10}
	\begin{cases}
		\Delta\tilde\phi=\tilde c &\mbox{in}\ Q\cap\{x_1>\delta\},\\
		\nu\cdot\nabla\tilde\phi=\bar f &\mbox{on}\ \partial Q\cap\{x_1>\delta\},\\
		\nu\cdot\nabla\tilde\phi=\bar g &\mbox{on}\ Q\cap\{x_1=\delta\},
	\end{cases}
\end{align}
where next to \eqref{ao47} we have also introduced the abbreviation, cf.\ (\ref{ao17}):
\begin{align}\label{ao84}
\bar g := g_0-g_1\quad\mbox{on}\;Q\cap\{x_1=\delta\}\cong Q':=(-1,1)^{d-1},
\end{align}
and $\tilde c$ is the constant that makes the problem solvable.

The final lemma states that there exists a competitor of transport cost
close to the Dirichlet integral of $\nabla\tilde\phi$, cf.~(\ref{ao78}), where it matters that
the r.h.s.\ of (\ref{ao78}) is super-linear in $E+\delta^2$ (after optimization in $\tau$), and also states that $\nabla\phi$ and $\nabla\tilde\phi$ have comparable Dirichlet energies,
cf.~(\ref{ao12}).
The proof is given in Section \ref{subsect:competitor}.

\begin{lemma}[Construction]\label{Lecomp}
Given $\tau\in(0,1/4)$, there exists an admissible pair $(\tilde\rho,\tilde j)$ (i.e., satisfying \eqref{ao101}) concentrated on $(0,1)\times\overline{Q}$ such that
\begin{align}\label{ao78}
\lefteqn{\int_{(0,1)\times\mathbb{R}^d}\frac{1}{\tilde\rho}|\tilde j|^2
-\int_{Q\cap\{x_1>\delta\}}|\nabla\tilde\phi|^2}\nonumber\\
&\lesssim \tau(E+\delta^2)+(E+\delta^2)^\frac{d+2}{d+1}+\delta^3+\frac{(M+\delta)^\frac{1}{d-1}}{\tau}E.
\end{align}
In case of $d=2$, the exponent $\frac{1}{d-1}$, which would be equal to $1$,
has to be replaced by an (in fact, any) exponent $\beta<1$.
In addition,
\begin{equation}\label{ao12}
	\int_{Q\cap\{x_1>\delta\}}|\nabla\tilde\phi|^2 -\int_{Q\cap\{x_1>-\delta\}}|\nabla\phi|^2
	\lesssim E^\frac{1}{2}\delta+\delta^2.
\end{equation}
\end{lemma}

%%%%%%%%%%%%%%%%%%%%%%%%%%%%%%%%%%%%%%%%%%%%%%%%%%%%%%%%%%%%%%%%%%%%%%%%%%%%
%%%%%%%%%%%%%%%%%%%%%%%%%%%%%%%%%%%%%%%%%%%%%%%%%%%%%%%%%%%%%%%%%%%%%%%%%%%%%%%%%%%%

We conclude this subsection by demonstrating that the above lemmas indeed imply Proposition \ref{Prharm}.

\begin{proof}[Proof of Proposition \ref{Prharm}]
	Within this proof we restore the index $R\in(1,2)$ taken in Lemma \ref{Legood} for clarity.

	We take $\phi\in C^{\infty}(Q_R)$ as defined in (\ref{ao05}) and by reflection, and prove that this $\phi$ (restricted to $\overline{Q_1}$) is the desired harmonic gradient.
	%(restricted from the $Q_R$ with $R\ge 1$ constructed in Lemma \ref{Legood} to $Q_1$).
	Since $\nabla\phi$ is harmonic and satisfies (\ref{ao80}) by definition, and (\ref{ao81}) follows from (\ref{ao76}), it only remains to establish (\ref{ao88}).
	We start from (\ref{ao11}) and (\ref{ao12}) in Lemma \ref{Leorth} which combine to
	\begin{align*}
	& \int_{(0,1)\times Q_R}\frac{1}{\rho}|j-\rho\nabla\phi|^2
	- \left(\int_{(0,1)\times Q_R}\frac{1}{\rho}|j|^2
	-\int_{Q_R\cap\{x_1>\delta\}}|\nabla\tilde\phi|^2\right) \nonumber\\
	&\lesssim E^\frac{1}{2}\delta+\delta^2+(\tau+\delta)^\frac{1}{2}E.
	\end{align*}
	We then turn to Lemma \ref{Lecomp} and appeal to the local optimality in Lemma \ref{lem:localoptimality}, namely
	$$\int_{(0,1)\times Q_R}\frac{1}{\rho}|j|^2\le\int_{(0,1)\times\mathbb{R}^d}\frac{1}{\tilde\rho}|\tilde j|^2,$$
	and to (\ref{ao78}) to obtain
	\begin{align}
	\lefteqn{\int_{(0,1)\times Q_R}\frac{1}{\rho}|j-\rho\nabla\phi|^2
	\lesssim E^\frac{1}{2}\delta+\delta^2}\nonumber\\
	&+(\tau+\delta)^\frac{1}{2}(E+\delta^2)+(E+\delta^2)^\frac{d+2}{d+1}
	+\frac{(M+\delta)^\frac{1}{d-1}}{\tau}E.\label{ao86}
	\end{align}
	Given $\varepsilon\in(0,1]$,
	for the very first r.h.s.~term $E^\frac{1}{2}\delta$ we apply Young's inequality
	to control it by $\varepsilon E+\frac{1}{\varepsilon}\delta^2\lesssim \varepsilon E+\frac{1}{\varepsilon}D$ as desired, cf.\ (\ref{ao87}).
	The remaining r.h.s.\ terms can all be made $\lesssim\varepsilon E+\delta^2$ so that $\lesssim\varepsilon E+D$; namely we choose $\tau$ so small that $(\tau+\delta)^\frac{1}{2}\le\varepsilon$, cf.\ (\ref{ao02}),
	and since $E+D\ll_\varepsilon1$ so that both
	$$(E+\delta^2)^\frac{1}{d+1} \stackrel{\eqref{ao53}}{\lesssim} (E+D)^\frac{1}{d+1} \le \varepsilon, \quad \frac{(M+\delta)^\frac{1}{d-1}}{\tau} \stackrel{\eqref{ao87},\eqref{ao53}}{\lesssim}\frac{(E^\frac{1}{d+2}+D^\frac{1}{2})^\frac{1}{d-1}}{\tau} \le \varepsilon.$$
	The arbitrariness of $\varepsilon$ implies the assertion.
\end{proof}

In the remainder of this section we prove all the above lemmas.

%%%%%%%%%%%%%%%%%%%%%%%%%%%%%%%%%%%%%%%%%%%%%%%%%%%%%%%%%%%%%%%%%%%%%%%%%%%%%%%%%%%%%
%%%%%%%%%%%%%%%%%%%%%%%%%%%%%%%%%%%%%%%%%%%%%%%%%%%%%%%%%%%%%%%%%%%%%%%%%%%%%%%%%%%%%

\subsection{Local optimality and good slices lemma}\label{subsect:goodslices}

In this subsection we prove Lemmas \ref{lem:localoptimality} and \ref{Legood} in line with \cite{GoldmanHuesmannOtto}.

\begin{proof}[Proof of Lemma \ref{lem:localoptimality}]
	By an approximation argument we may use the (discontinuous) test functions $\chi_{Q_R}\partial_t\zeta$ and $\chi_{Q_R}\nabla\zeta$ in the definition of $(\rho,j)$, cf.\ \eqref{ao39}, for an arbitrary $\zeta\in C^1_c([0,1]\times\mathbb{R}^d)$.
	Hence we have
	\begin{align*}
		&\int_{(0,1)\times Q_R}\partial_t\zeta d\rho + \nabla\zeta\cdot dj \\
		&= \int_0^1\int_{\{X(t)\in Q_R\}} \left(\partial_t\zeta(t,X(t))+\nabla\zeta(t,X(t))\cdot\dot{X}(t)\right)\mathbb{P}(dX)dt\\
		&= \int_{\mathcal{T}_R} \left(\zeta(t_+,X(t_+))-\zeta(t_-,X(t_-))\right) \mathbb{P}(dX),
	\end{align*}
	where $\mathcal{T}_R:=\{X \mid X(t)\in Q_R\ \mbox{for some }t\}$, and the exiting and entering times $t_{\pm,R}=t_{\pm,R}(X)$ are understood as in \eqref{ao71} even if $X$ is totally contained in $Q_R$.
	Since $\mathcal{T}_R$ is decomposed into the disjoint sets $\{X(1)\in Q_R\}$ and $\mathcal{T}_{+,R}=\mathcal{T}_R\cap\{X(1)\not\in Q_R\}$, cf.\ \eqref{ao67}, we deduce from the marginal condition \eqref{ao41} and \eqref{ao38} that
	\begin{align*}
		\int_{\mathcal{T}_R} \zeta(t_{+,R},X(t_{+,R})) \mathbb{P}(dX)
		%&= \int_{\{X(1)\in Q_R\}} \zeta(1,X(1)) \mathbb{P}(dX) + \int_{\mathcal{T}_{+,R}} \zeta(t_{+,R},X(t_{+,R})) \mathbb{P}(dX)\\
		= \int_{Q_R}\zeta(1,\cdot) \chi_{\Omega_1}dx + \int_{(0,1)\times\partial Q_R}\zeta df_{R,+}.
	\end{align*}
	Similarly we have
	\begin{align*}
		\int_{\mathcal{T}_R} \zeta(t_{-,R},X(t_{-,R})) \mathbb{P}(dX) = \int_{Q_R}\zeta(0,\cdot) \chi_{\Omega_0}dx + \int_{(0,1)\times\partial Q_R}\zeta df_{R,-},
	\end{align*}
	and hence by the definition of $f_R$, cf.\ \eqref{ao37}, we obtain \eqref{ao100}.

	We finally prove local optimality \eqref{ao102}.
	Since by \eqref{ao100} and \eqref{ao101} the pair $(\hat{\rho},\hat{j}):=(\tilde{\rho},\tilde{j})+(\rho,j)|_{(0,1)\times(\mathbb{R}^d\setminus\overline{Q_R})}$ is a competitor in \eqref{ao36}, and since by definition \eqref{ao96} the transportation cost is sub-additive, we have
	$$\int_{(0,1)\times\mathbb{R}^d}\frac{1}{\rho}|j|^2 \leq \int_{(0,1)\times\mathbb{R}^d}\frac{1}{\hat\rho}|\hat{j}|^2 \leq \int_{(0,1)\times\mathbb{R}^d}\frac{1}{\tilde\rho}|\tilde{j}|^2 + \int_{(0,1)\times(\mathbb{R}^d\setminus\overline{Q_R})}\frac{1}{\rho}|j|^2.$$
	The additivity of the classical integral $\int\frac{1}{\rho}|j|^2$ completes the proof.
\end{proof}

\begin{proof}[Proof of Lemma \ref{Legood}]
	We start with (\ref{ao43}) by arguing that
	\begin{align*}
	\int_1^2\int_{(0,1)\times\partial Q_R}f_R^2 dR\leq E.
	\end{align*}
	Indeed, since we have $f_R=\nu\cdot j$ a.e.~on $\partial Q_R$ for a.e.~$R\in(1,2)$, cf. \eqref{ao100},
	we obtain from the co-area formula $\int_{Q_2\setminus Q_1}=\int_1^2\int_{\partial Q_R} dR$ that
	\begin{align*}
	\int_1^2\int_{(0,1)\times\partial Q_R}f_R^2 dR \leq \int_{(0,1)\times Q_2}(\nu\cdot j)^2
	\stackrel{\eqref{ao54}}{\le}\int_{(0,1)\times Q_2}\frac{1}{\rho}|j|^2
	\leq E,
	\end{align*}
	where the last estimate follows since the $L^\infty$-bounds (\ref{ao53}) and (\ref{ao72}) yield that there are no trajectories going through $Q_2$ such that $x_0\not\in Q_{8}$ and $x_1\not\in Q_{8}$.

	Turning to (\ref{ao42}) we show
	\begin{align*}
	\int_1^2\int_{(0,1)\times\mathbb{R}^d}\frac{1}{\rho_R'}|j_R'|^2dR\le ME.
	\end{align*}
	Indeed, it follows from the definition of $(\rho_R',j_R')$, cf.\ \eqref{ao114},
	that
	\begin{align*}
	\lefteqn{\int_{(0,1)\times\mathbb{R}^d}\frac{1}{\rho_R'}|j_R'|^2
	\leq\int_{{\mathcal T}_R'}|X(1)-X(0)|^2\mathbb{P}(dX)}\nonumber\\
	&\stackrel{\eqref{ao66}}{\le}\int_{{\mathcal T}_{+,R}\cup{\mathcal T}_{-,R}}
	|X(1)-X(0)|^2\mathbb{P}(dX)\stackrel{\eqref{ao67}}{=}
	\int I_{X,R} |X(1)-X(0)|^2\mathbb{P}(dX),\nonumber
	\end{align*}
	where $I_{X,R}$ denotes the characteristic function of the set
	$$\left\{ X \left|\ |X(1)|_\infty\ge R>\min_t|X(t)|_\infty\; {\rm or}\;
	|X(0)|_\infty\ge R>\min_t|X(t)|_\infty\right\}\right.,$$
	and $|\cdot|_\infty$ denotes the norm in which the cube is the ball.
	We now integrate in $R$: Using that ($I_{X,R}=1$ $\Longrightarrow$ $\max_t|X(t)|_\infty $ $\ge R$ $>\min_t|X(t)|_\infty$), appealing to
	$\max_t|X(t)|_\infty-\min_t|X(t)|_\infty\le\max_{t',t}|X(t')-X(t)|$, inferring from (\ref{ao53}) that
	\begin{align}\label{ao68}
	\max_{t',t\in(0,1)}|X(t')-X(t)|\le M\quad\mbox{provided}\;X(0)\in Q_{4}\;\mbox{or}\;X(1)\in Q_{4},
	\end{align}
	and noting that by (\ref{ao72}) the set of (straight) trajectories $X$ with $X(t)\in Q_2$ for some $t$
	is contained in the set of trajectories with $X(1)\in Q_{4}$ or $X(0)\in Q_{4}$, we have
	\begin{align*}
	\lefteqn{\int_1^2\int_{(0,1)\times\mathbb{R}^d}\frac{1}{\rho_R'}|j_R'|^2 dR
	}\nonumber\\
	&\le M
	\int_{\{X(1)\in Q_{4}\}\cup\{X(0)\in Q_{4}\}}|X(1)-X(0)|^2\mathbb{P}(dX)
	\lesssim M E,
	\end{align*}
	where for the last estimate we directly inferred from (\ref{ao73}) that
	\begin{align}\label{ao69}
	\int_{\{X(0)\in Q_{4}\}\cup\{X(1)\in Q_{4}\}}|X(1)-X(0)|^2\mathbb{P}(dX)\lesssim E.
	\end{align}

	We now turn to the proof of (\ref{ao44}), and (\ref{ao46}).
	Appealing to symmetry, we restrict ourselves to $\rho_0$ and $\rho_0'$.
	The pointwise bounds in form of $0\le\rho_0'\le\rho_0\le 1$
	follow via definitions (\ref{ao45}) and (\ref{ao64}) from (\ref{ao41}).
	For the integral bound we will restrict ourselves to (\ref{ao46}),
	since (\ref{ao44}) follows from the latter for $\tau=\frac{1}{2}$, and show
	that
	\begin{align*}
	\int_1^2\int_{\mathbb{R}^d}{\rm dist}(\cdot,\partial Q_R)\rho_{0,R}' dR\le \tau^2E.
	\end{align*}
	By definition (\ref{ao64}), we have
	\begin{align*}
	\int_{\mathbb{R}^d}{\rm dist}(\cdot,\partial Q_R)\rho_{0,R}'
	=\int_{{\mathcal T}_{+,R}''}|R-|X(0)|_\infty|\mathbb{P}(dX).
	\end{align*}
	By definition (\ref{ao70}) of ${\mathcal T}_{+,R}''$ and (\ref{ao71}) of $t_{+,R}$
	we have for $X\in {\mathcal T}_{+,R}''$ that $\sup_{t<\tau}|X|_\infty$ $\ge R$
	$>\inf_{t<\tau}|X|_\infty$ and thus in particular
	$$|R-|X(0)|_\infty| \le \sup_{t<\tau}|X|_\infty - \inf_{t<\tau}|X|_\infty \le \tau|X(1)-X(0)|,$$
	where the last estimate follows since trajectories are straight.
	Hence we obtain from integrating in $R$ (and noting the implicitly-multiplied characteristic function on the set $\{(X,R)\mid\sup_{t<\tau}|X|_\infty\ge R>\inf_{t<\tau}|X|_\infty\}$ as above),
	\begin{align*}
	\int_1^2\int_{\mathbb{R}^d}{\rm dist}(\cdot,\partial Q_R)\rho_{0,R}'dR \le \tau^2\int_{\{X\mid \exists t,\ X(t)\in Q_2 \}}|X(1)-X(0)|^2\mathbb{P}(dX).
	\end{align*}
	By (\ref{ao72}) again, the r.h.s.~integral is estimated by $E$, cf.~(\ref{ao69}).
\end{proof}

%By the above estimates for the averages in $R$, we can choose a good $R$, because a (measurable) subset $S$ of $(1,2)$ such that $f(R) > 100 \int_1^2 f dR$ holds for $R\in S$ must have measure less than $1/100$.

%%%%%%%%%%%%%%%%%%%%%%%%%%%%%%%%%%%%%%%%%%%%%%%%%%%%%%%%%%%%%%%%%%%%%%%%%%%%%%%%%%%%
%%%%%%%%%%%%%%%%%%%%%%%%%%%%%%%%%%%%%%%%%%%%%%%%%%%%%%%%%%%%%%%%%%%%%%%%%%%%%%%%%%%%

\subsection{Approximate orthogonality}\label{subsect:orthogonality}

\begin{proof}[Proof of Lemma \ref{Leorth}]
	By the definition of the integrand $\frac{1}{\rho}|j|^2$, the support condition (\ref{ao16}) implies that also
	$\frac{1}{\rho}|j-\rho\nabla\phi|^2$ and $\frac{1}{\rho}|j|^2$ are supported in $\{x_1>-\delta\}$.
	Hence we obtain from expanding the square:
	\begin{align}\label{ao95}
	\int_{(0,1)\times Q}\frac{1}{\rho}|j-\rho\nabla\phi|^2
	&=\int_{(0,1)\times Q}\frac{1}{\rho}|j|^2
	-\int_{(0,1)\times(Q\cap\{x_1>-\delta\})}(2-\rho)|\nabla\phi|^2\nonumber\\
	&+2\int_{(0,1)\times(Q\cap\{x_1>-\delta\})}(\nabla\phi-j)\cdot\nabla\phi.
	\end{align}
	Therefore by $\rho\le 1$, cf.~(\ref{ao54}), for (\ref{ao11}) it
	remains to bound the last term.

	In what follows we first prepare some trace estimates for later use, then rephrase the above last term so that the estimates in the previous step are applicable, and finally complete the proof by estimating the rephrased one; estimate (\ref{ao76}) is shown in the middle of Step 1.

	{\em Step 1: Trace estimates and control of boundary fluxes.}
	We first note %that estimate (\ref{ao43}) and definition (\ref{ao47}) combine to
	\begin{align}\label{ao48}
	\int_{(0,1)\times\partial Q}(f-f')^2+\int_{\partial Q}\bar f^2\lesssim E.
	\end{align}
	Indeed, since we have a splitting $f'=f'_+-f'_-$ with
	$0\le f_{\pm}'\le f_{\pm}$, cf.~(\ref{ao38}) and (\ref{ao55}),
	and $f_\pm$ being the positive and negative parts of $f$ as was shown right after (\ref{ao38}),
	we have
	\begin{align}\label{ao65}
	|f-f'|\le|f|,\quad |f'|\le |f|,
	\end{align}
	so that (\ref{ao48}) follows from definition (\ref{ao47}) and estimate (\ref{ao43}).

	By standard $L^2$-based maximal regularity theory for (\ref{ao05}) in terms
	of the Neumann data (see Remark \ref{rem:L2maximalregularity} below for details), we have
	\begin{align}\label{ao06}
	\sup_{x_1}\int_{Q'}|\nabla\phi(x_1,\cdot)|^2\lesssim\int_{\partial Q}\bar f^2
	\stackrel{\eqref{ao48}}{\lesssim} E,
	\end{align}
	where $Q'=\{x'\in(-1,1)^{d-1}\}$ denotes the $(d-1)$-dimensional cube in tangential
	direction; we call (\ref{ao06}) a maximal regularity estimate since the left and
	right hand side have the same scaling.
	Estimate (\ref{ao76}) now follows directly from (\ref{ao06}).

	We further post-process (\ref{ao06}): For a test function $\zeta$ on $Q\cap\{x_1>a\}$ (for some $a\in(-1/2,1/2)$), combining the Poincar\'e-trace estimate in $(d-1)$ dimensions
	\begin{align*}
	\int_{\partial Q'}\zeta^2+\int_{Q'}\zeta^2
	\lesssim\int_{Q'}|\nabla'\zeta|^2+\big|\int_{Q'}\zeta\big|^2,
	\end{align*}
	with the obvious
	\begin{align*}
	\sup_{x_1}\big|\int_{Q'}\zeta\big|^2
	\lesssim\sup_{x_1}\int_{Q'}|\partial_1\zeta|^2
	+\big|\int_{Q\cap\{x_1>a\}}\zeta\big|^2,
	\end{align*}
	we obtain
	\begin{align}\label{ao75}
	\sup_{x_1}\big(\int_{\partial Q'}\zeta^2+\int_{Q'}\zeta^2\big)
	\lesssim\sup_{x_1}\int_{Q'}|\nabla\zeta|^2
	+\big|\int_{Q\cap\{x_1>a\}}\zeta\big|^2.
	\end{align}
	Applying this to $\zeta=\phi$ on $Q\cap\{x_1>-\delta\}$, where we normalize such that
	\begin{align}\label{ao09}
	\int_{Q\cap\{x_1>-\delta\}}\phi=0,
	\end{align}
	we obtain by (\ref{ao06}) the (co-dimension two) estimate
	\begin{align}
	\sup_{x_1}\left(\int_{\partial Q'}\phi^2+\int_{Q'}\phi^2\right)\lesssim E.\label{ao07}
	\end{align}

	{\em Step 2: Rephrasing in terms of $\phi$ and boundary fluxes.}
	We now rephrase the last term in (\ref{ao95}) in terms of $\phi$ and normal fluxes on the boundary.
	Integrating in $t$ the continuity equation and recalling the normal flux boundary data $f$ for $j$, and the initial data $\chi_{\Omega_0}$ and terminal data $\chi_{\Omega_1}$ for $\rho$ (namely taking constant-in-$t$ test functions $\zeta$ in (\ref{ao100})), we see that $\bar j:=\int_0^ 1jdt$ distributionally satisfies
	\begin{align}\label{ao115}
		\begin{cases}
			\nabla\cdot\bar j=-g & \mbox{in}\ Q\cap\{x_1>-\delta\},\\
			\nu\cdot\bar j=\int_0^1fdt & \mbox{on}\ \partial Q\cap\{x_1>-\delta\},\\
			\nu\cdot\bar j=0 & \mbox{on}\ Q\cap\{x_1=-\delta\},
		\end{cases}
	\end{align}
	where we have introduced the characteristic function $g:=\chi_{\Omega_1}-\chi_{\Omega_0}$, that is,
	\begin{align}\label{ao56}
		g(x):=
		\begin{cases}
			1 & \mbox{for}\ g_1(x')<x_1<g_0(x'),\\
			-1 & \mbox{for}\ g_0(x')<x_1<g_1(x'),\\
			0 & \mbox{else},
		\end{cases}
	\end{align}
	in the first line, cf.~(\ref{ao17}), and used (\ref{ao16}) for the last item.
	Therefore, appealing to (\ref{ao05}) (in its reflected form), to (\ref{ao09}), and to (\ref{ao115}) tested by $\phi$, we obtain
	for the last term in (\ref{ao95})
	\begin{align*}
	\lefteqn{\int_{(0,1)\times(Q\cap\{x_1>-\delta\})}(\nabla\phi-j)\cdot\nabla\phi
	=\int_{Q\cap\{x_1>-\delta\}}(\nabla\phi-\bar j)\cdot\nabla\phi}\nonumber\\
	&\stackrel{\eqref{ao09},\eqref{ao115}}{=}-\int_{Q\cap\{x_1>-\delta\}}g\phi
	+\int_{\partial Q\cap\{x_1>-\delta\}}\phi\left(\bar f-\int_0^1fdt\right)
	-\int_{Q\cap\{x_1=-\delta\}}\phi\partial_1\phi.
	\end{align*}
	We claim that we may rewrite the last term as
	\begin{align}\label{ao89}
	\lefteqn{-\int_{Q\cap\{x_1=-\delta\}}\phi\partial_1\phi}\nonumber\\
	&=\int_{-\delta}^0\int_{Q'}\nabla'\phi(-\delta,x')\cdot\nabla'\phi(x_1,x')dx'dx_1
	+\delta c\int_{Q'}\phi(-\delta,x')dx'.
	\end{align}
	Indeed, we obtain \eqref{ao89} by using the homogeneous flux boundary condition in (\ref{ao05}) in form of
	$$-\partial_1\phi(-\delta,x')=\int_{-\delta}^0\partial_1^2\phi(x_1,x')dx_1,$$
	combining it with the Poisson equation in form of
	$$\int_{-\delta}^0\partial_1^2\phi(x_1,x')dx_1=\delta c-\int_{-\delta}^0\Delta'\phi(x_1,x')dx_1,$$
	and appealing to
	integration by parts in $x'$ in form of
	$$-\int_{Q'}\phi(-\delta,x')\int_{-\delta}^0\Delta'\phi(x_1,x')dx_1 dx'=\int_{-\delta}^0\int_{Q'}\nabla'\phi(-\delta,x')\cdot\nabla'\phi(x_1,x')dx'dx_1,$$
	where there are no boundary terms here since $\bar f$ vanishes for $x_1<\delta$ by (\ref{ao51}) and (\ref{ao47}).
	Hence we rearrange the terms as follows
	\begin{align*}
	&\int_{(0,1)\times(Q\cap\{x_1>-\delta\})}(\nabla\phi-j)\cdot\nabla\phi \nonumber\\
	&=-\int_{Q\cap\{|x_1|<\delta\}}g\phi-\int_{\partial Q\cap\{|x_1|<\delta\}}\phi\int_0^1fdt
	+\int_{\partial Q\cap\{x_1>\delta\}}\phi\left(\bar f-\int_0^1fdt\right) \nonumber\\
	& \quad +\int_{Q\cap\{-\delta<x_1<0\}}(\nabla'\phi|_{x_1=-\delta})\cdot\nabla'\phi
	+\delta c\int_{Q\cap\{x_1=-\delta\}}\phi\nonumber\\
	& =:T_1+T_2+T_3+T_4+T_5,
  \end{align*}
	where we also used ${\rm supp}g\subset\{|x_1|\le\delta\}$, cf.~(\ref{ao56}) and (\ref{ao02}).

	{\em Step 3: Main estimates.}
	We finally estimate the above $T_1,\dots,T_5$.
	Using $|g|\le 1$ and (\ref{ao07}), we see that $T_1$ is responsible
	for the leading-order term in (\ref{ao11}):
	\begin{align*}
	|T_1|\lesssim E^\frac{1}{2}\delta.
	\end{align*}
	For $T_2$ we note that by the Cauchy-Schwarz inequality
	\begin{align*}
	|T_2|\le\left(\Big(\sup_{x_1}\int_{\partial Q'}\phi^2\Big)2\delta\int_{(0,1)\times\partial Q}f^2\right)^{1\over 2}
	\stackrel{\eqref{ao07},\eqref{ao43}}{\lesssim} \delta^\frac{1}{2}E,
	\end{align*}
	which is contained in the r.h.s.~of (\ref{ao11}).
	We turn to $T_3$ and note that by definition (\ref{ao47}) of
	$\bar f$ we have $\int_0^1f dt-\bar f=\int_0^1f'dt$; by definition
	(\ref{ao55}) of $f'$ we have, provided $x_1>\delta$, that $f'$ vanishes unless
	$t\in(0,\tau)\cup(1-\tau,1)$, so that
	$$\int_{\partial Q\cap\{x_1>\delta\}}\left(\int_0^1f'dt\right)^2 \le 2\tau\int_{(0,1)\times \partial Q}f'^2.$$
	Finally by (\ref{ao65}) we may apply (\ref{ao43}), to the effect of
	$\int_{(0,1)\times \partial Q}f'^2\lesssim E$. In conclusion we have
	\begin{align*}
	\int_{\partial Q\cap\{x_1>\delta\}}\left(\bar f-\int_0^1f dt\right)^2\lesssim\tau E,
	\end{align*}
	so that we obtain with help of the Cauchy-Schwarz inequality and (\ref{ao07})
	\begin{align*}
	|T_3|\lesssim\tau^\frac{1}{2} E,
	\end{align*}
	in line with (\ref{ao11}). We directly obtain from the Cauchy-Schwarz inequality
	and (\ref{ao06}):
	\begin{align*}
	|T_4|\lesssim\delta E,
	\end{align*}
	in agreement with (\ref{ao11}) (since $\delta\ll1$).
	Finally, for $T_5$ we note that from (\ref{ao05}) we have
	$c|Q\cap\{x_1>0\}|$ $=\int_{\partial Q\cap\{x_1>0\}}\bar f$, so that
	by (\ref{ao48}) we get
	\begin{align}\label{ao90}
	|c|\lesssim E^\frac{1}{2},
	\end{align}
	hence in conjunction with (\ref{ao07}) we likewise obtain
	\begin{align*}
	|T_5|\lesssim\delta E,
	\end{align*}
	completing the proof.
\end{proof}

\begin{remark}[$L^2$-maximal regularity]\label{rem:L2maximalregularity}
	For the reader's convenience, we sketch the argument for (\ref{ao06}).
	By the triangle inequality, it is enough to consider Neumann data that are supported
	on one of the $2d$ faces of the box $Q$, which we take without loss of generality to
	be the cube $(0,2)^d$ and $\{x_1=2\}$ to be that face.
	By even reflection, which preserves the equation, we may replace the cube by a slab
	$S:=(-2,2)\times[-2,2)^{d-1}$
	of thickness 4 and (lateral) periodicity with period 4. By treating constant Neumann
	data explicitly, we may restrict to the case of Neumann data of vanishing spatial average.
	By the boundedness of the Neumann-to-Dirichlet map (which easily may be seen for instance on the
	Fourier side) we may assume that we have a harmonic function $\phi$ on the periodic slab
	$S$ %(we label the normal coordinate again by $x_1$)
	of which we control $\int_{\partial S}|\nabla\phi|^2$
	on the boundary $\partial S:=\{-2,2\}\times[-2,2)^{d-1}$ and where we seek control on faces
	$\int_{S\cap\{x_i=\mathrm{const.}\}}|\nabla\phi|^2$. Hence we may replace
	(any component of) the gradient $\nabla\phi$ by a harmonic function $\phi$
	%of vanishing spatial average
	in both $L^2$-norms.
	Control of tangential faces $\int_{S\cap\{x_1=\mathrm{const.}\}}\phi^2$
	can easily be seen, for instance on the Fourier side.
	For a perpendicular face, say
	$\int_{S\cap\{x_2=0\}}\phi^2$, the argument goes as follows:
	By the triangle inequality, we may assume that $\phi$ vanishes on the lower boundary $\{x_1=-2\}$.
	Let $\psi$ be such that $\partial_1\psi=\phi$ and that it vanishes on
	the upper boundary $\{x_1=2\}$.
	We easily see, for instance by considering the Fourier transform in the tangential
	directions $x'$ to obtain a representation of $\phi$ in terms of its boundary data,
	that the control of the boundary data in terms of
	$\int_{\partial S\cap\{x_1=2\}}\phi^2$ yields control of the harmonic extension
	$\phi$ (and its anti-derivative $\psi$)
	in the weighted $L^2_{x'}$-based norm
	$$\int_{S}\Big(\frac{1}{2-x_1}(\partial_2\psi)^2+(2-x_1)(\partial_1\phi)^2+\phi^2\Big),$$
	where the power of the distance to the boundary $\{x_1=2\}$ is optimal (for the
	first two terms) and dictated by scaling.
	By Young's inequality, this yields control of
	$\int_{(-2,2)^d}(|\partial_2\psi\partial_1\phi|+\phi^2).$
	By an integration by
	parts in $x_1$ using $\phi|_{\{x_1=-2\}}=0$ and $\psi|_{\{x_1=2\}}=0$, this yields control
	of
	$$\int_{-2}^2\Big( \big|\int_{(-2,2)^{d-1}}\phi\partial_2\phi dx'' \big|+\int_{(-2,2)^{d-1}}\phi^2dx'' \Big)dx_2,$$
	where we have set $x'':=(x_1,x_3,\cdots,x_d)$.
	Because of
	$$\frac{d}{dx_2}\int_{(-2,2)^{d-1}}\phi^2dx''=2\int_{(-2,2)^{d-1}}\phi\partial_2\phi dx'',$$
	this yields the desired control of $\sup_{x_2}\int_{(-2,2)^{d-1}}\phi^2 dx''$.
\end{remark}

%%%%%%%%%%%%%%%%%%%%%%%%%%%%%%%%%%%%%%%%%%%%%%%%%%%%%%%%%%%%%%%%%%%%%%%%%%%%%%%%%%%%
%%%%%%%%%%%%%%%%%%%%%%%%%%%%%%%%%%%%%%%%%%%%%%%%%%%%%%%%%%%%%%%%%%%%%%%%%%%%%%%%%%%%

\subsection{Construction of a competitor}\label{subsect:competitor}

\begin{proof}[Proof of Lemma \ref{Lecomp}]
	We first independently address estimate (\ref{ao12}), and then turn to the construction of an admissible competitor $(\tilde{\rho},\tilde{j})$.
	%Throughout the proof we argue within the local region $(0,1)\times\overline{Q}$.
	As is described in Section \ref{subsec:outline}, we will construct $(\tilde{\rho},\tilde{j})$ of the form \eqref{ao103}, where all the r.h.s.\ terms are concentrated on $(0,1)\times Q$ except for the singular construction $(\rho^\mathrm{sing},j^\mathrm{sing})$ concentrated on $(0,1)\times\partial Q$, and then appeal to sub-additivity of the cost functional.
	The construction is divided into a number of steps.

	{\em Step 1: Comparability of the energies of $\phi$ and $\tilde{\phi}$.}
	For (\ref{ao12}), obviously, it is sufficient to estimate
	\begin{align}\label{ao15}
	T:=\int_{Q\cap\{x_1>\delta\}}|\nabla\tilde\phi|^2
	-\int_{Q\cap\{x_1>\delta\}}|\nabla\phi|^2=\int_{Q\cap\{x_1>\delta\}}\nabla(\tilde\phi+\phi)\cdot\nabla(\tilde\phi-\phi).%\lesssim\delta^\frac{1}{2}(E+\delta^2).
	\end{align}
	By an integration by parts and (\ref{ao05}) and (\ref{ao10}), $T$ can be rewritten as
	\begin{align*}
	T=-\int_{Q\cap\{x_1>\delta\}}(\tilde\phi+\phi)(\tilde c-c)
	+\int_{Q\cap\{x_1=\delta\}}(\tilde\phi+\phi)(\bar g+\partial_1\phi).
	\end{align*}
	As for (\ref{ao89}), we rewrite the term
	$\int_{Q\cap\{x_1=\delta\}}(\tilde\phi+\phi)\partial_1\phi$, leading to
	\begin{align*}
	T&=-(\tilde c-c)\int_{Q\cap\{x_1>\delta\}}(\tilde\phi+\phi)
	+\int_{Q\cap\{x_1=\delta\}}(\tilde\phi+\phi)(\bar g+2\delta c)\nonumber\\
	&+\int_{-\delta}^\delta
	\int_{Q'}\nabla'(\tilde\phi+\phi)(\delta,x')\cdot\nabla'\phi(x_1,x')dx'dx_1.
	%\\&+\int_{-\delta}^\delta
	%\int_{\partial Q'}(\tilde\phi+\phi)(\delta,x')\bar f(x_1,x')dx'dx_1.
	\end{align*}
	We apply (\ref{ao75})
	to $\zeta:=\tilde\phi+\phi$ on $Q\cap\{x_1>\delta\}$, where we normalize $\tilde\phi$ such
	that $\int_{Q\cap\{x_1>\delta\}}\zeta$ $=0$, so that in particular
	the first of the above r.h.s.~terms vanishes, to the effect of
	\begin{align}\label{ao91}
	|T|\lesssim\left(\sup_{x_1}\int_{Q'}|\nabla\zeta|^2\Big(\int_{Q'}\bar g^2+\delta^2c^2+
	\delta^2\sup_{x_1}\int_{Q'}|\nabla\phi|^2
	%+\delta\int_{\partial Q}\bar f^2\big)
	\Big)\right)^\frac{1}{2}.
	\end{align}
	We now appeal to the a priori estimate (\ref{ao06})
	not just in case of $\nabla\phi$, but also in case of $\nabla\tilde\phi$,
	where in view of (\ref{ao10}), the (squared) $L^2$-norm of the Neumann data
	is estimated by $\int_{\partial Q}\bar f^2$ $+\int_{Q'}\bar g^2$, where
	the second term contributes
	\begin{align}\label{ao92}
	\int_{Q'}\bar g^2\lesssim\delta^2,
	\end{align}
	cf.~(\ref{ao84}) and (\ref{ao02}), to obtain
	\begin{align}\label{ao93}
	\sup_{x_1}\int_{Q'}|\nabla\zeta|^2
	\lesssim\sup_{x_1}\left(\int_{Q'}|\nabla\tilde\phi|^2+\int_{Q'}|\nabla\phi|^2\right)
	\lesssim \delta^2+E.
	\end{align}
	Inserting (\ref{ao06}), (\ref{ao90}), (\ref{ao92}), and (\ref{ao93}) into (\ref{ao91}) (and using $E\ll1$)
	yields
	\begin{align*}
	|T|\lesssim E^{1\over 2}\delta+\delta^2,
	\end{align*}
	which amounts to the r.h.s.~of (\ref{ao12}).

	{\em Step 2: Necessary mass balance.}
	From now on we argue for the construction.
	In this step we first obtain some necessary mass balance conditions for solvability of the continuity equation, which are used for determining (the initial and terminal density in) the main construction.
	After this step we give the precise definition and estimate of each term of $(\tilde{\rho},\tilde{j})$, cf.\ \eqref{ao103}.

	We start with a remark on global mass preservation in our various constructions:
	As could be derived from (\ref{ao41}) with $\zeta=\chi_{Q}$ and (\ref{ao37}) with $\zeta=1$, we have
	\begin{align}\label{ao57}
	|\Omega_0\cap Q|=|\Omega_1\cap Q|+\int_{(0,1)\times\partial Q}f.
	\end{align}
	Likewise, we have from the definition (\ref{ao55}) of $f'$ with $\zeta=1$ and the definitions
	(\ref{ao45}) of $\rho_0,\rho_1$ with $\zeta=\chi_Q$ (not continuous but allowable by approximation) that
	\begin{align}\label{ao58}
	\int_Q\rho_0=\int_Q\rho_1+\int_{(0,1)\times\partial Q}f'.
	\end{align}

	The initial construction $(\rho^\mathrm{ini},j^\mathrm{ini})$ in the layer $(0,\tau)\times Q$, which connects
	the density $\chi_{\Omega_0}-\rho_0$ at $t=0$ to the piecewise constant density
	$\chi_{\Omega_0\cap\{x_1<\delta\}}+c_0\chi_{\{x_1>\delta\}}$
	at $t=\tau$ with no flux across $(0,\tau)\times\partial Q$, requires $c_0$ to be defined
	such that
	\begin{align}\label{ao59}
	|Q\cap\{x_1>\delta\}|-\int_Q\rho_0=c_0|Q\cap\{x_1>\delta\}|.
	\end{align}

	Similarly, the terminal construction $(\rho^\mathrm{term},j^\mathrm{term})$ lives in the layer $(1-\tau,1)\times Q$ and
	connects $\chi_{\Omega_1\cap\{x_1<\delta\}}+c_1\chi_{\{x_1>\delta\}}$
	at $t=1-\tau$ to $\chi_{\Omega_1}-\rho_1$ with no flux across $(1-\tau,1)\times\partial Q$,
	which requires $c_1$ to be defined such that
	\begin{align}\label{ao60}
	|Q\cap\{x_1>\delta\}|-\int_Q\rho_1=c_1|Q\cap\{x_1>\delta\}|.
	\end{align}

	The boundary construction $(\rho^\mathrm{bdry},j^\mathrm{bdry})$ is in the left half-space
	$(\tau,1-\tau)\times(Q\cap\{x_1<\delta\})$; it connects $\chi_{\Omega_0}$ at $t=\tau$
	to $\chi_{\Omega_1}$ at $t=1-\tau$, with the constant-in-$t$ flux $-\frac{1}{1-2\tau}\bar g$,
	cf.~(\ref{ao84}),
	across $(\tau,1-\tau)\times(Q\cap\{x_1=\delta\})$ and no flux through the remaining boundary
	portion $(\tau,1-\tau)\times(\partial Q\cap\{x_1>\delta\})$ (in line with the support condition in (\ref{ao51})). In particular, as could also be seen
	from (\ref{ao17}) and (\ref{ao84}), we have
	\begin{align}\label{ao77}
	|\Omega_0\cap Q\cap \{x_1<\delta\}|=|\Omega_1\cap Q\cap \{x_1<\delta\}|-\int_{Q\cap\{x_1=\delta\}}\bar g.
	\end{align}

	It then follows from (\ref{ao57}), (\ref{ao58}), (\ref{ao59}), (\ref{ao60}), (\ref{ao77}),
	and the support condition on $f-f'$ in (\ref{ao51}) that we have the necessary
	mass balance for the main construction $(\rho^\mathrm{main},j^\mathrm{main})$ in $(\tau,1-\tau)\times(Q\cap\{x_1>\delta\})$, namely
	\begin{equation*}
		c_0|Q\cap\{x_1>\delta\}|=c_1|Q\cap\{x_1>\delta\}|+\int_{(0,1)\times(\partial Q\cap\{x_1>\delta\})}(f-f')
		+\int_{Q\cap\{x_1=\delta\}}\bar g,
	\end{equation*}
	from which we infer that the constant $\tilde{c}$ in \eqref{ao10} satisfies
	\begin{equation}\label{ao113}
		\tilde{c}=c_0-c_1.
	\end{equation}
	We also note for later purpose that we learn from (\ref{ao44}) that for $r\in(0,1)$,
	$$\int_Q(\rho_0+\rho_1) \leq 2|Q\setminus (1-r)Q| + \frac{1}{r}\int_{(1-r)Q}\mathrm{dist}(\cdot,\partial Q)(\rho_0+\rho_1) \lesssim r +\frac{1}{r}E,$$
	and by optimizing in $r$ that $\int_Q(\rho_0+\rho_1)\lesssim E^{\frac{1}{2}}$, so that by (\ref{ao59}) and (\ref{ao60}) we must have in our regime of $E\ll_\varepsilon 1$, cf.~(\ref{ao73}),
	\begin{align}\label{ao62}
	|c_0-1|,\ |c_1-1| \lesssim E^{\frac{1}{2}} \ll_\varepsilon 1.
	\end{align}

	{\em Step 3: Kept trajectories.}
	We now turn to the estimate of a competitor, starting with the kept trajectories $(\rho^\mathrm{kept},j^\mathrm{kept}):=(\rho',j')|_{(0,1)\times Q}$, cf.\ \eqref{ao114} and Figure \ref{fig:kept}, the cost of which is already estimated, cf.~(\ref{ao42}):
	\begin{align}\label{ao107}
		\int_{(0,1)\times\mathbb{R}^d}\frac{1}{\rho^\mathrm{kept}}|j^\mathrm{kept}|^2 \lesssim ME,
	\end{align}
	which is of higher order with respect to the r.h.s.~of (\ref{ao78}).

	\begin{figure}[htbp]
		\begin{center}
			\includegraphics[width=125mm]{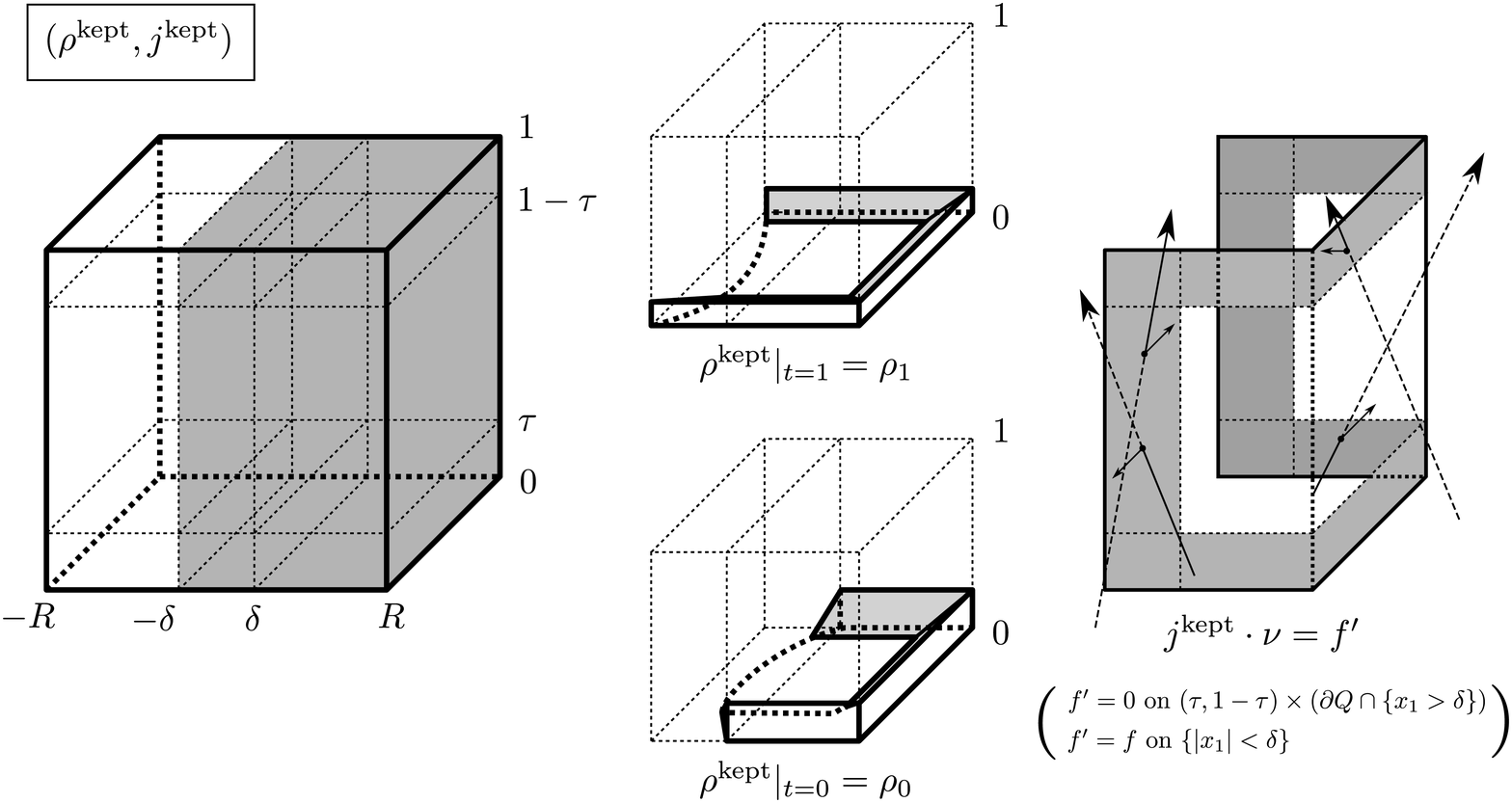}
			\caption{Kept trajectories $(\rho^\mathrm{kept},j^\mathrm{kept})$.}
	    \label{fig:kept}
		\end{center}
	\end{figure}

	{\em Step 4: Main construction.}
	We now go into individual constructions, starting with the main construction.
	The main construction $(\rho^\mathrm{main},j^\mathrm{main})$ is defined by the sum of two measures $(\tilde{s},\tilde{q})$ and $(s,q)$ both living in $(\tau,1-\tau)\times(Q\cap\{x_1>\delta\})$.
	The former is defined by $(\tilde{s},\tilde{q}):=(\frac{t-\tau}{1-2\tau}c_1+\frac{1-\tau-t}{1-2\tau}c_0,\frac{1}{1-2\tau}\nabla\tilde{\phi})$,	cf.\ \eqref{ao10}, which satisfies the continuity equation by \eqref{ao113}, and connects the constant density $c_0$
	at $t=\tau$ to the constant density $c_1$ at $t=1-\tau$, with constant-in-$t$ normal fluxes $\frac{1}{1-2\tau}\bar{g}$ across $(\tau,1-\tau)\times (Q\cap\{x_1=\delta\})$, cf.\ (\ref{ao84}), and $\frac{1}{1-2\tau}\bar{f}$ across the remaining boundary portion $(\tau,1-\tau)\times(\partial Q\cap\{x_1>\delta\})$, cf.\ \eqref{ao47}; its cost is directly computed as
	\begin{align*}
		\int_{(\tau,1-\tau)\times(Q\cap\{x_1>\delta\})}\frac{1}{\tilde{s}}|\tilde{q}|^2 = \frac{1}{1-2\tau}\int_0^1\frac{dt}{tc_1+(1-t)c_0}\int_{Q\cap\{x_1>\delta\}}|\nabla\tilde\phi|^2,
	\end{align*}
	and hence by \eqref{ao62} and by $\int_{Q\cap\{x_1>\delta\}}|\nabla\tilde{\phi}|^2\lesssim E+\delta^2$ (e.g.\ from (\ref{ao06})) we have
	\begin{align}\label{ao111}
		\int_{(\tau,1-\tau)\times(Q\cap\{x_1>\delta\})}\frac{1}{\tilde{s}}|\tilde{q}|^2 - \int_{Q\cap\{x_1>\delta\}}|\nabla\tilde\phi|^2 \lesssim (\tau+E^{\frac{1}{2}})(E+\delta^2).
	\end{align}
	The latter $(s,q)$ is based on the boundary layer construction from \cite[Lemma 2.4]{GoldmanOtto}\footnote{
	The original statement deals only with a ball domain, but it is not difficult to similarly argue for the present rectangular domain $Q\cap\{x_1>\delta\}$ with small $\delta>0$.
	}; it satisfies the continuity equation in $(\tau,1-\tau)\times(Q\cap\{x_1>\delta\})$ and $|s|\leq1/2$,
	connects the zero-densities at $t=\tau$ and $t=1-\tau$, with no flux across $(\tau,1-\tau)\times(Q\cap\{x_1=\delta\})$ and flux $f^\mathrm{main}-\frac{1}{1-2\tau}\bar{f}$ across $(\tau,1-\tau)\times(\partial Q\cap\{x_1>\delta\})$, where $f^\mathrm{main}$ is defined (on $(0,1)\times\partial Q$ for later use) by
	\begin{align}\label{ao63}
		f^\mathrm{main} :=
		\begin{cases}
			f-f'+\frac{1}{1-2\tau}\int_{(0,\tau)\cup(1-\tau,1)}(f-f')dt & \mbox{if}\ \tau<t<1-\tau,\\
			0 & \mbox{else},
		\end{cases}
	\end{align}
	so that since $\int_\tau^{1-\tau}(f^\mathrm{main}-\frac{1}{1-2\tau}\bar{f})dt=0$ we may apply \cite[Lemma 2.4]{GoldmanOtto}; in addition, noting that $\int_\tau^{1-\tau}(f^\mathrm{main}-\frac{1}{1-2\tau}\bar{f})^2dt \lesssim E$ by (\ref{ao47}) and (\ref{ao48}), we may choose $(s,q)$ to be concentrated on $(\tau,1-\tau)\times A_r$ for an $r \lesssim E^\frac{1}{d+1}\ll1$,
	where $A_r:=Q\cap\{\mathrm{dist}(\cdot,\partial(Q\cap\{x_1>\delta\}))<r\}$, and also satisfy the key estimate
	\begin{equation}\label{ao112}
		\int_{(\tau,1-\tau)\times Q}|q|^2 \lesssim \Big(\int_{(\tau,1-\tau)\times\partial Q}\big(f^\mathrm{main}-\frac{1}{1-2\tau}\bar{f}\big)^2\Big)^\frac{d+2}{d+1} \lesssim E^\frac{d+2}{d+1}.
	\end{equation}
	%and vanishes outside $(\tau,1-\tau)\times A_r$ for an arbitrary $r\gg(\int_{(\tau,1-\tau)\times\partial Q}F^2)^{\frac{1}{d+1}}$ (we will fix shortly), where $A_r:=Q\cap\{\mathrm{dist}(\cdot,\partial(Q\cap\{x_1>\delta\}))<r\}$;
	%\begin{equation}\label{ao112}
	%	\int_{(\tau,1-\tau)\times A_r}\frac{1}{2}|q|^2 \lesssim r\int_{(\tau,1-\tau)\times\partial Q}F^2 \stackrel{(\ref{ao47}),(\ref{ao48})}{\lesssim} rE \sim E^\frac{d+2}{d+1},
	%\end{equation}
	%where in the last procedure we have fixed $r$ being a large but order-one multiple of $E^\frac{1}{d+1}$.
	In summary, $(\rho^\mathrm{main},j^\mathrm{main})$ connects the constant density $c_0$
	at $t=\tau$ to the constant density $c_1$ at $t=1-\tau$, with fluxes $\frac{1}{1-2\tau}\bar{g}$ across $(\tau,1-\tau)\times(Q\cap\{x_1=\delta\})$ and $f^\mathrm{main}$ across $(\tau,1-\tau)\times(\partial Q\cap\{x_1>\delta\})$ (see Figure \ref{fig:main}).
	By definition of $(\rho^\mathrm{main},j^\mathrm{main})$,
	\begin{align*}
	\int_{(0,1)\times\mathbb{R}^d}\frac{1}{\rho^\mathrm{main}}|j^\mathrm{main}|^2 &= \int_{(\tau,1-\tau)\times(Q\cap\{x_1>\delta\}\setminus A_r)}\frac{1}{\tilde{s}}|\tilde{q}|^2 + \int_{(\tau,1-\tau)\times A_r}\frac{1}{\tilde{s}+s}|\tilde{q}+q|^2,
	\end{align*}
	and hence by \eqref{ao111}, by $|s|\leq1/2$ and by $|\tilde{s}-1|\lesssim E^\frac{1}{2} \ll 1$, cf.\ \eqref{ao62},
	\begin{align*}
		\int_{(0,1)\times\mathbb{R}^d}\frac{1}{\rho^\mathrm{main}}|j^\mathrm{main}|^2 - \int_{(\tau,1-\tau)\times(Q\cap\{x_1>\delta\})}\frac{1}{\tilde{s}}|\tilde{q}|^2 &\lesssim \int_{(\tau,1-\tau)\times A_r}(|\tilde{q}|^2+|q|^2)\\
		& \lesssim \int_{A_r}|\nabla\tilde\phi|^2 + E^\frac{d+2}{d+1}.
	\end{align*}
	Combining this with (\ref{ao111}) and $\int_{A_r}|\nabla\tilde\phi|^2 \lesssim r(E+\delta^2)$ (obtained by estimates such as (\ref{ao06}) in all directions), recalling that $r\lesssim E^\frac{1}{d+1}$, and absorbing all the higher-order terms, we reach the desired
	\begin{align}\label{ao108}
	\int_{(0,1)\times\mathbb{R}^d}\frac{1}{\rho^\mathrm{main}}|j^\mathrm{main}|^2 -\int_{Q\cap\{x_1>\delta\}}|\nabla\tilde\phi|^2 \lesssim \tau (E+\delta^2)+(E+\delta^2)^\frac{d+2}{d+1}.
	\end{align}

	\begin{figure}[htbp]
		\begin{center}
			\includegraphics[width=125mm]{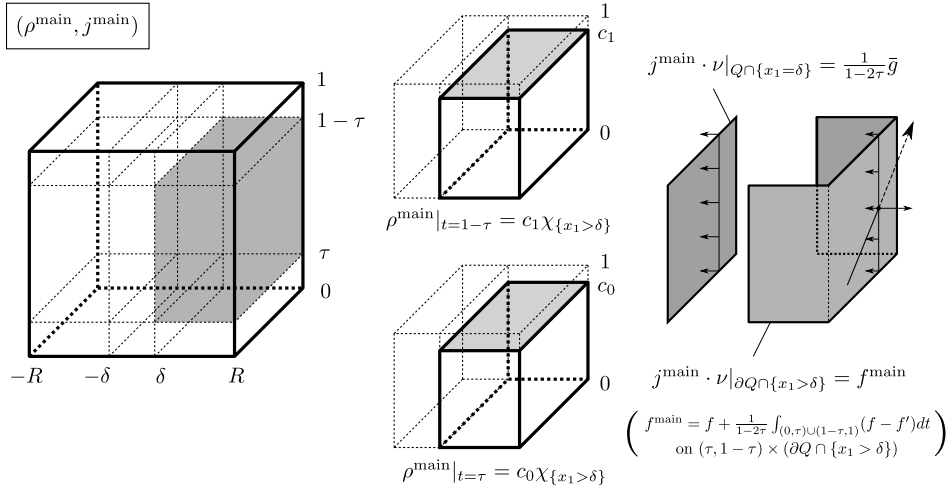}
			\caption{Main construction.}
	    \label{fig:main}
		\end{center}
	\end{figure}

	{\em Step 5: Singular construction.}
	The singular construction $(\rho^\mathrm{sing},j^\mathrm{sing})$ is then taken for accommodating the flux $f^\mathrm{main}$ through $(0,1)\times\partial Q$, cf.\ \eqref{ao63}, to the desired one $f-f'$.
	More precisely, we define the measure $(\rho^\mathrm{sing},j^\mathrm{sing})$ concentrated on $(0,1)\times\partial Q$ (in fact on $(0,1)\times(\partial Q\cap\{x_1\geq\delta\})$) through the density
	\begin{align*}
	\rho^\mathrm{sing}:=
	\begin{cases}
		-\int_0^t(f-f') & \mbox{for}\ t\in(0,\tau),\\
		\frac{t+\tau-1}{1-2\tau}\int_0^\tau(f-f')+\frac{t-\tau}{1-2\tau}\int_{1-\tau}^1(f-f') &
		\mbox{for}\ t\in(\tau,1-\tau),\\
		\int_t^1(f-f') & \mbox{for}\ t\in(1-\tau,1),
	\end{cases}
	\end{align*}
	which is made so that $\rho^\mathrm{sing}|_{t=0,1}=0$, $\rho^\mathrm{sing}\ge 0$, cf.\ \eqref{ao51},
	and $\partial_t\rho^\mathrm{sing}=f^\mathrm{main}-(f-f')$ on $(0,1)\times\partial Q$ (see Figure \ref{fig:singular}).
	Then it distributionally solves the continuity equation with the everywhere-vanishing flux $j^\mathrm{sing}:=0$ in $(0,1)\times Q$, with flux $(f-f')-f^\mathrm{main}$ across $(0,1)\times\partial Q$, so that this extra construction comes at no cost, cf.\ \eqref{ao96}:
	\begin{equation}\label{ao109}
		\int_{(0,1)\times\mathbb{R}^d} \frac{1}{\rho^\mathrm{sing}}|j^\mathrm{sing}|^2=0.
	\end{equation}

	\begin{figure}[htbp]
		\begin{center}
			\includegraphics[width=125mm]{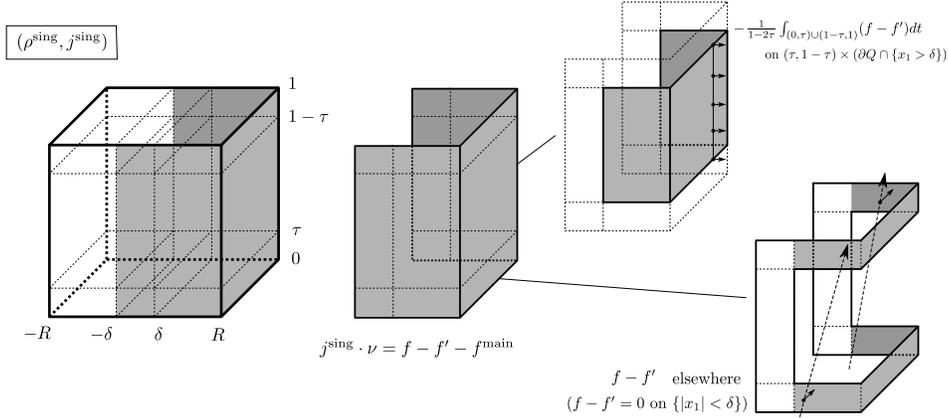}
			\caption{Singular construction.}
	    \label{fig:singular}
		\end{center}
	\end{figure}

	{\em Step 6: Boundary construction.}
	We now turn to the boundary construction in
	$(\tau,1-\tau)\times(Q\cap\{x_1<\delta\})$; we recall that it connects
	the density $\chi_{\Omega_0\cap\{x_1<\delta\}}$ at $t=\tau$ to
	the density $\chi_{\Omega_1\cap\{x_1<\delta\}}$ at $t=1-\tau$,
	has constant-in-$t$ normal flux $-\frac{1}{1-2\tau}\bar g$
	across $(\tau,1-\tau)\times(Q\cap\{x_1=\delta\})$,
	and no normal flux across $(\tau,1-\tau)\times(\partial Q\cap\{x_1<\delta\})$.
	The boundary construction is explicitly given by shearing in the normal direction $x_1$:
	\begin{align*}
	(\rho^\mathrm{bdry},j^\mathrm{bdry})(t,(x_1,x')) :=
	\begin{cases}
		(1,-\frac{\bar g(x')}{1-2\tau} e_1) & \mbox{for}\ x_1>\frac{t-\tau}{1-2\tau}g_1(x')+\frac{1-\tau-t}{1-2\tau}g_0(x'),\\
		(0,0) &\mbox{else},
	\end{cases}
	\end{align*}
	which in view of the definition of $\bar g$, cf.~(\ref{ao84}), distributionally satisfies the continuity equation %cf.~(\ref{ao97}),
	and obviously the desired flux boundary condition (see Figure \ref{fig:boundary}).
	It satisfies the desired initial and terminal conditions by (\ref{ao17}),
	and we have
	\begin{align}\label{ao110}
	\int_{(0,1)\times\mathbb{R}^d}\frac{1}{\rho^\mathrm{bdry}}|j^\mathrm{bdry}|^2 = \frac{1}{1-2\tau}\int_{Q'}\big(\frac{g_0+g_1}{2}-\delta\big)\bar g^2
	\stackrel{\eqref{ao02}}{\lesssim}\delta^3.
	\end{align}

	\begin{figure}[htbp]
		\begin{center}
			\includegraphics[width=125mm]{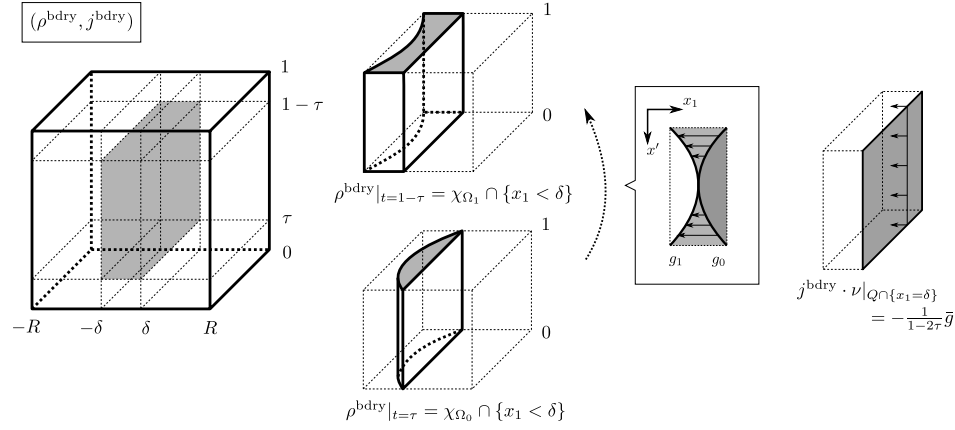}
			\caption{Boundary construction.}
	    \label{fig:boundary}
		\end{center}
	\end{figure}

	{\em Step 7: Initial and terminal construction.}
	The remainder of the proof is devoted to the initial and terminal construction;
	by symmetry, we restrict to the initial construction.
	The initial construction $(\rho^\mathrm{ini},j^\mathrm{ini})$ lives in $(0,\tau)\times Q$ and is defined by the (Eulerian) optimal transport from $(\chi_{\Omega_0}-\rho_0)|_Q$ to $(\chi_{\Omega_0\cap\{x_1<\delta\}}+c_0\chi_{\{x_1>\delta\}})|_Q$ rescaled-in-$t$ from $(0,1)$ to $(0,\tau)$ (see Figure \ref{fig:initial_terminal}).
	The no-flux condition follows since $Q$ is convex.
	In what follows we will verify that
	\begin{align}\label{ao106}
	\int_{(0,1)\times\mathbb{R}^d}\frac{1}{\rho^\mathrm{ini}}|j^\mathrm{ini}|^2 \lesssim \big(\tau+\frac{(M+\delta)^\frac{1}{d-1}}{\tau}\big)E,
	\end{align}
	dividing the proof into Steps 7-1, 7-2, and 7-3.

	\begin{figure}[htbp]
		\begin{center}
			\includegraphics[width=125mm]{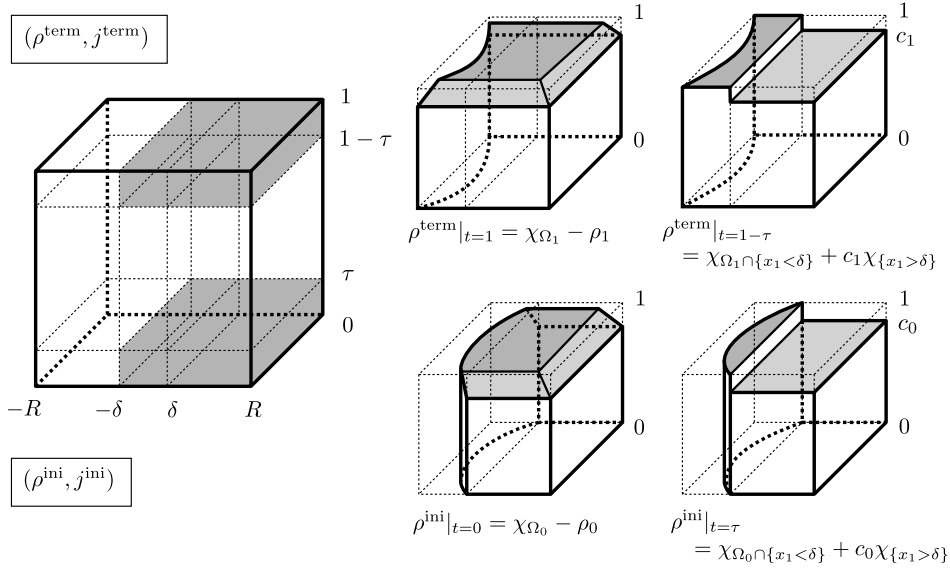}
			\caption{Initial and terminal construction.}
	    \label{fig:initial_terminal}
		\end{center}
	\end{figure}

	{\em Step 7-1: Estimate of the cost by using a Poisson equation.}
	Estimate \eqref{ao106} is based on the following observation:
	\begin{align}\label{ao18}
	\int_{(0,1)\times\mathbb{R}^d}\frac{1}{\rho^\mathrm{ini}}|j^\mathrm{ini}|^2 \lesssim\frac{1}{\tau}\int_{\Omega_0}|\nabla\phi_0|^2,
	\end{align}
	where $\nabla\phi_0$ solves the Neumann problem (solvable by (\ref{ao59})):
	\begin{align}\label{ao25}
		\begin{cases}
			\Delta\phi_0 = (1-c_0)\chi_{\{x_1>\delta\}}-\rho_0 & \mbox{in}\ \Omega_0\cap Q,\\
			\nu\cdot\nabla\phi_0 = 0 & \mbox{on}\ \partial(\Omega_0\cap Q).
		\end{cases}
	\end{align}
	Indeed, if $W_Q$ denotes the Wasserstein distance on the (convex) set $Q$, i.e., if for densities $\mu_0$ and $\mu_1$ of same mass on $Q$ we let $W_Q^2(\mu_0,\mu_1)$ denote the cost of the optimal transport map between $\mu_0|_Q$ and $\mu_1|_Q$, cf.\ \eqref{ao36}, then by definition of the initial construction we have
	\begin{align}\label{ao22}
	\int_{(0,1)\times\mathbb{R}^d}\frac{1}{\rho^\mathrm{ini}}|j^\mathrm{ini}|^2 = \frac{1}{\tau}W_Q^2\big(\chi_{\Omega_0}-\rho_0\,,
	\,\chi_{\Omega_0\cap\{x_1<\delta\}}+c_0\chi_{\{x_1>\delta\}}\big).
	\end{align}
	This entails (\ref{ao18}) by two general observations: The first observation is that
	for arbitrary densities $\mu_0$, $\mu_1$ on $Q$ of the same mass we have
	\begin{align}\label{ao20}
	W_Q(\mu_0,\mu_1)\le\frac{1}{\sqrt{2}-1} W_Q(\mu_0+\mu_1,2\mu_1).
	\end{align}
	Inequality (\ref{ao20}) follows from the scaling of $W_Q$ in the mass,
	its triangle inequality, and its sub-additivity:
	\begin{align*}
	W_Q(\mu_0,\mu_1) &= \frac{1}{\sqrt{2}}W_Q(2\mu_0,2\mu_1)\nonumber\\
	&\le \frac{1}{\sqrt{2}}\big(W_Q(2\mu_0,\mu_0+\mu_1)+W_Q(\mu_0+\mu_1,2\mu_1)\big)\nonumber\\
	&\le \frac{1}{\sqrt{2}}\big(W_Q(\mu_0,\mu_1)+W_Q(\mu_0+\mu_1,2\mu_1)\big).
	\end{align*}
	The second observation is that for any smooth open set $\Omega_0$
	and same generic densities $\mu_{0}$, $\mu_{1}$ (uniformly positive on $\Omega_0\cap Q$),
	\begin{align}\label{ao19}
	W_Q^2(\mu_{0},\mu_{1})\le\frac{1}{\inf_{\Omega_0\cap Q}\min\{\mu_{0},\mu_{1}\}}
	\int_{\Omega_0\cap Q}|\nabla\bar\phi|^2,
	\end{align}
	where $\nabla\bar\phi$ solves the Neumann problem:
	\begin{equation}\label{ao21}
		\begin{cases}
			\Delta\bar\phi=\mu_{0}-\mu_{1} & \mbox{in}\ \Omega_0\cap Q,\\
			\nu\cdot\nabla\bar\phi=0 & \mbox{on}\ \partial(\Omega_0\cap Q).
		\end{cases}
	\end{equation}
	This follows by appealing to the Eulerian formulation of the Wasserstein distance and
	choosing
	\begin{align*}
	(\bar{\rho},\bar{j})=
	\begin{cases}
		(t\mu_{1}+(1-t)\mu_{0},\nabla\bar\phi) & \mbox{on}\ \Omega_0\cap Q,\\
		(0,0) & \mbox{else},
	\end{cases}
	\end{align*}
	so that the (distributional) continuity equation subject to $\bar{\rho}(0,\cdot)=\mu_{0}$ and $\bar{\rho}(1,\cdot)=\mu_{1}$ is a direct consequence of (\ref{ao21}), and hence
	$$W_Q^2(\mu_{0},\mu_{1})\leq\int_{(0,1)\times\mathbb{R}^d}\frac{1}{\bar{\rho}}|\bar{j}|^2.$$
	Starting from (\ref{ao20}) and applying (\ref{ao19})
	with $(\mu_{0},\mu_{1})$ replaced by $(\mu_0+\mu_1,2\mu_1)$
	yield
	\begin{align*}
	W_Q^2(\mu_0,\mu_1) \leq \big(\frac{1}{\sqrt{2}-1}\big)^2\frac{1}{\inf_{\Omega_0\cap Q}\mu_1}
	\int_{\Omega_0\cap Q}|\nabla\bar{\phi}|^2,
	\end{align*}
	with $\bar{\phi}$ still defined as in (\ref{ao21}), as long as $\mu_0$ and $\mu_1$ are supported on $\Omega_0\cap Q$ and of same mass there.
	Applying this to $\mu_0:=\chi_{\Omega_0}-\rho_0$ and $\mu_1:=\chi_{\Omega_0\cap\{x_1<\delta\}}+c_0\chi_{\{x_1>\delta\}}$ both restricted to $Q$, and hence $\phi_0$ playing the role of $\bar{\phi}$, we see that (\ref{ao22}) turns into (\ref{ao18}), where we also use (\ref{ao62}).

	%%%%%%%%%%%%%%%%%%%%%%%%%%%%%%%%%%%%%%%%%%%%%%%%%%%%%%%%%%%%%%%%%%%%%%%%%%%%%%%%%%%

	{\em Step 7-2: Elliptic estimates.}
	In order to estimate the r.h.s.~of (\ref{ao18}), we need an elementary
	but perhaps somewhat un-usual elliptic estimate on (\ref{ao25}), which capitalizes
	on the concentration of $\rho_0$ (which we extend trivially on $\mathbb{R}^d$)
	near $\partial Q$, namely
	\begin{align}\label{ao28}
	\int_{Q\cap\Omega_0}|\nabla\phi_0|^2\lesssim
	\big(\int_{\partial Q}|\bar\rho_0|^\frac{2(d-1)}{d}\big)^\frac{d}{d-1}
	+\int_{Q}{\rm dist}^2(\cdot,\partial Q)\rho_0^2,
	\end{align}
	where for any $x\in\partial Q$, $\bar\rho_0(x)$ is the integral of $|\rho_0|$ along
	the line perpendicular to $\partial Q$ in $x$ (non-negativity of $\rho_0$ is not
	needed for this linear estimate), more precisely, along the segment $I_x:=\{y\in Q \mid |y-x|={\rm dist}(y,\partial Q)\}$, i.e., consisting of those
	points that have $x$ as their (orthogonal) projection onto $\partial Q$.
	In case of $d=2$, the exponent $\frac{2(d-1)}{d}$, which would be $=1$, has to be replaced
	by an (in fact, any) exponent $q>1$. For later
	reference we retain the elementary inequality
	\begin{align}\label{ao29}
	\frac{1}{2}\bar\rho_0(x)^2\le\sup_{I_x}|\rho_0|\int_{I_x}{\rm dist}(\cdot,\partial Q)|\rho_0|.
	\end{align}
	For the convenience of the reader, we give the argument for (\ref{ao29}) (for general $\rho_0$):
	Writing $\rho_0=|\rho_0|$, we may assume $\rho_0\ge 0$; by homogeneity, we may assume $\sup_{I_x}\rho_0=1$.
	By an extremality argument we may then assume $\rho_0\in\{0,1\}$, so that we
	are dealing with a characteristic function of a subset of $I_x$. It is clear that
	the configuration that minimizes $\int_{I_x}{\rm dist}(\cdot,\partial Q)\rho_0$
	comes from an interval adjacent to the boundary of $I_x$, so that (\ref{ao29})
	follows from an explicit calculation.

	We now turn to the argument for (\ref{ao28}):
	Since we may without loss of generality assume that
	$\int_{Q\cap\{x_1>\delta\}}\phi_0$ $=0$, we obtain from testing (\ref{ao25}):
	\begin{align*}
	\int_{Q\cap\Omega_0}|\nabla\phi_0|^2=\int_{Q\cap\Omega_0}\phi_0\rho_0.
	\end{align*}
	It is convenient to have an extension $\zeta$ of $\phi_0$ on all of $\mathbb{R}^d$, but
	supported in $2Q$, at hand
	with $\int_{\mathbb{R}^d}|\nabla\zeta|^2\lesssim\int_{Q\cap\Omega_0}|\nabla\phi_0|^2$, which exists
	since $Q\cap\partial\Omega_0$ is in particular locally a Lipschitz graph with constant $\le 1$. Hence it
	suffices to establish
	\begin{align*}
	\big|\int_{\mathbb{R}^d}\zeta\rho_0\big|^2\lesssim\Big(\big(\int_{\partial Q}
	|\bar\rho_0|^\frac{2(d-1)}{d}\big)^\frac{d}{d-1}
	+\int_{Q}{\rm dist}^2(\cdot,\partial Q)\rho_0^2\Big)
	\int_{\mathbb{R}^d}|\nabla\zeta|^2,
	\end{align*}
	which amounts to an estimate of $\rho_0$ in $H^{-1}$.
	%where we think of $\rho_0$ as being trivially extended from $Q\cap\Omega_0$ to $\mathbb{R}^d$.
	For this, it is convenient to split the cube $Q$ into $2d$ pyramids consisting
	of those points closest
	to one of the faces; we split $\rho_0$ into $2d$ pieces $\rho$ accordingly.
	Hence it is enough to replace $Q$ by the half space $\{(x_1,x') \mid x_1>0\}$:
	\begin{align*}
	\big|\int_{\{x_1>0\}}\zeta\rho\big|^2\lesssim
	\Big(\big(\int_{\mathbb{R}^{d-1}}|\bar\rho|^\frac{2(d-1)}{d}\big)^\frac{d}{d-1}
	+\int_{\{x_1>0\}}x_1^2\rho^2\Big)
	\int_{\mathbb{R}^d}|\nabla\zeta|^2,
	\end{align*}
	where the bar still denotes the integral in normal direction, which now means $\bar\rho(x')=\int_0^\infty\rho(x_1,x') dx_1$.
	This is best seen by splitting the l.h.s.~as
	\begin{align*}
	\int_{\{x_1>0\}}\zeta\rho=\int_{\mathbb{R}^{d-1}}\zeta(0,\cdot)\bar\rho
	+\int_{\{x_1>0\}}(\zeta-\zeta(0,\cdot))\rho.
	\end{align*}
	The first r.h.s.~term is estimated as desired by H\"older's inequality and by the Sobolev trace inequality on $\zeta$ in form of
	$\big(\int_{\mathbb{R}^{d-1}}|\zeta(0,\cdot)|^\frac{2(d-1)}{d-2})^\frac{d-2}{2(d-1)}$
	$\lesssim(\int_{\mathbb{R}^d}|\nabla\zeta|^2)^\frac{1}{2}$; in the case of $d=2$ we use that for any $p<\infty$,
	$\big(\int_{\mathbb{R}^{d-1}}|\zeta(0,\cdot)|^p)^\frac{1}{p}$
	$\lesssim(\int_{\mathbb{R}^d}|\nabla\zeta|^2)^\frac{1}{2}$, making use of the fact that $\zeta$ is supported in $2Q$.
	The second r.h.s.~term can be rewritten
	as
	$$\int_{\{x_1>0\}}(\zeta-\zeta(0,\cdot))\rho = \int_{\{x_1>0\}}\rho\int_0^{x_1}\partial_1\zeta = \int_{\{x_1>0\}}\partial_1\zeta\int_{x_1}^\infty\rho,$$
	so that the desired estimate follows from the Cauchy-Schwarz inequality on $\{x_1>0\}$
	and Hardy's inequality in form of
	\begin{align*}
	\int_0^\infty\big(\int_{x_1}^\infty\rho\big)^2dx_1
	\le 4\int_0^\infty x_1^2\rho^2dx_1.
	\end{align*}

	%%%%%%%%%%%%%%%%%%%%%%%%%%%%%%%%%%%%%%%%%%%%%%%%%%%%%%%%%%%%%%%%%%%%%%%%%%%%%%%%%

	{\em Step 7-3: Completion of the estimate for the initial construction.}
	Equipped with (\ref{ao28}), we are now in a position to prove (\ref{ao106}).
	The second contribution to (\ref{ao28}) is easily estimated: By the $L^\infty$-bound
	(\ref{ao68}) in conjunction with definition (\ref{ao45}) and ${\mathcal T}_+'$
	$\subset{\mathcal T}_+$ we have
	\begin{align*}
	{\rm supp}\rho_0\subset\{{\rm dist}(\cdot,\partial Q)\le M\}.
	\end{align*}
	Combining this with (\ref{ao44}), we have
	\begin{align*}
	\int_Q{\rm dist}^2(\cdot,\partial Q)\rho_0^2\lesssim M E.
	\end{align*}
	The first contribution in (\ref{ao28}) requires more care; we split
	$\rho_0$ into $\rho_0'$ and $\rho_0-\rho_0'$, cf.~(\ref{ao64}). On the contribution
	from $\rho_0'$, or rather its integral $\bar\rho_0'$ in the normal direction to $\partial Q$ (defined as the above $\bar\rho_0$),
	we use H\"older's inequality:
	\begin{align*}
	\big(\int_{\partial Q}(\bar\rho_0')^\frac{2(d-1)}{d}\big)^\frac{d}{d-1}
	\lesssim\int_{\partial Q}(\bar\rho_0')^2
	\stackrel{\eqref{ao29}}{\lesssim}\sup|\rho_0'|\int_Q{\rm dist}(\cdot,\partial Q)|\rho_0'|
	\stackrel{\eqref{ao46}}{\lesssim}\tau^2 E.
	\end{align*}
	We now turn to the contribution from $\rho_0-\rho_0'$ and note that
	by definitions (\ref{ao45}) and (\ref{ao64}), we have
	\begin{equation*}
		\begin{split}
			\int\zeta(\rho_0-\rho_0')=\int_{{\mathcal T}_+'\setminus{\mathcal T}_+''}\zeta(X(0))\mathbb{P}(dX),\\
			\mbox{with}\quad{\mathcal T}_+'\setminus{\mathcal T}_+''
			\subset\{X\in{\mathcal T}_+ \mid |X_1(t_+)|<\delta\},
		\end{split}
	\end{equation*}
	so that by the $L^\infty$-bound (\ref{ao68}) we obtain
	\begin{align*}
	\begin{array}{c}
	{\rm supp}(\rho_0-\rho_0')\subset\{|x_1| \le M+\delta\}\cap\{{\rm dist}(\cdot,\partial Q)\le M\} \\[1ex]
	\mbox{and thus}\quad
	{\rm supp}(\overline{\rho_0-\rho_0'})\subset\{|x_1|\le M+\delta\},
	\end{array}
	\end{align*}
	where $\overline{\rho_0-\rho_0'}$ denotes the integral of $\rho_0-\rho_0'$ in the normal direction (as above).
	This allows us to use H\"older's inequality in the following form
	\begin{align*}
	\lefteqn{\big(\int_{\partial Q}\overline{\rho_0-\rho_0'}^\frac{2(d-1)}{d}\big)^\frac{d}{d-1}
	\lesssim(M+\delta)^\frac{1}{d-1}\int_{\partial Q}\overline{\rho_0-\rho_0'}^2}\\
	&\le(M+\delta)^\frac{1}{d-1}\int_{\partial Q}\bar\rho_0^2
	\stackrel{\eqref{ao29},\eqref{ao44}}{\lesssim}
	(M+\delta)^\frac{1}{d-1}E.
	\end{align*}
	Combining all three contributions we obtain
	\begin{align*}
	\int_{Q\cap\Omega_0}|\nabla\phi_0|^2\lesssim(\tau^2+(M+\delta)^\frac{1}{d-1})E,
	\end{align*}
	which in turn we insert in (\ref{ao18}) to end up with \eqref{ao106}.

	Finally, we notice that the local competitor $(\tilde{\rho},\tilde{j})$ of the form \eqref{ao103} is admissible in the sense of \eqref{ao101}, and then deduce \eqref{ao78} from \eqref{ao107}, \eqref{ao108}, \eqref{ao109}, \eqref{ao110} and \eqref{ao106} combined with sub-additivity, completing the proof.
\end{proof}

\section{From harmonic approximation to $C^{1,\alpha}$-regularity}\label{sect:lagrangian}

In this final section we complete the proof of Theorem \ref{thm:epsilonregularity}, demonstrating that the harmonic approximation of Proposition \ref{Prharm} implies the boundary $C^{1,\alpha}$-regularity via the steps in Section \ref{sect:preliminaries}.
Throughout this section we use the same notations as in Section \ref{sect:preliminaries}.
%In addition, without loss of generality, we may assume that $R=1$ by the rescaling that $\widetilde{\Omega}_i:=R^{-1}\Omega_i$ ($i=0,1$) and $\widetilde{T}(\widetilde{x}):=R^{-1}T(R\widetilde{x})$, and also assume that $\nu_0(0)=-e_1$ $(=\nu_1(0))$ by rotation.

\subsection{One-step improvement result}

\begin{proof}[Proof of Proposition \ref{prop:onestep}]
	Fix any $\Omega_0$, $\Omega_1$, $T$, and $R$ satisfying the conditions.
	Without loss of generality we may assume that $\nu_0(0)=-e_1$ $(=\nu_1(0))$ by rotation.
	In what follows we fix $\beta\in(0,1)$ and always implicitly assume that $E+D\ll_\beta1$ (and hence $E+D\ll1$ in particular).
	In addition, we use the same notation $C>0$ for all universal constants depending only on $d$ and $\alpha$.

	{\it Step 1: Definition of $b$ and $B$.}
	We define $\bar{b}\in\mathbb{R}^d$ and $\bar{A},\bar{B}\in\mathbb{R}^{d\times d}$ by
	\begin{align}\label{eqn:bAB}
		\bar{b}:=\nabla\phi(0), \quad \bar{A}:=\nabla^2\phi(0), \quad \bar{B}:=e^{-\bar{A}/2},
	\end{align}
	where $\phi$ is as in Proposition \ref{prop:harmoniclagrangian}.
	Notice that $\bar{B}$ is symmetric, i.e., $\bar{B}=\bar{B}^*$ since so is $\bar{A}$.
	By the mean value property of harmonic $\nabla\phi$,
	\begin{equation}\label{ao015}
		|\bar{b}|^2+|\bar{A}|^2=|\nabla\phi(0)|^2+|\nabla^2\phi(0)|^2\lesssim \int_{B_r}|\nabla\phi|^2 \stackrel{\eqref{eqn:harmonicDirichlet}}\lesssim E,
	\end{equation}
	and thus in particular $|\bar{A}|\ll1$ so that $|\bar{B}-Id|^2\lesssim|\bar{A}|^2\lesssim E$; hence,
	\begin{align}\label{ao014}
		|\bar{b}|^2+|\bar{B}-Id|^2\lesssim E\ll1.
	\end{align}
	%so in particular $|\bar{B}|$ $\bar{B}$ is regular.
	In addition, $\bar{b}\cdot e_1=0$, i.e., $\bar{b}$ is perpendicular to $e_1$, and moreover $\bar{B}$ has the block structure of $\bar{B}_{i1}=\bar{B}_{1j}=0$ for $i,j\geq2$, where $\bar{B}=(\bar{B}_{ij})$.
	Indeed, $\partial_1\phi(0,x')= 0$ holds for all $(0,x')\in B_r$ by the reflection symmetry (\ref{eqn:harmonicsymmetry}).
	Hence, $\partial_1\phi(0)=0$ so that $\bar{b}\cdot e_1=0$, cf.\ (\ref{eqn:bAB}), and also $\nabla'\partial_1\phi(0)=0$ so that $\bar{A}$ has the desired block structure, from which $\bar{B}$ inherits the same block structure.

	We will see later (in Step 3) that an affine transformation defined by $\bar{b}$ and $\bar{B}$ plays a key role in proving the main estimate (\ref{eqn:onestep2}) as in the interior regularity theory.
	However this transformation generically destroys the well-preparedness of the boundaries.
	In order to recover the well-preparedness we need to modify $\bar{b}$ and $\bar{B}$ to $b$ and $B$ so that not only (\ref{eqn:onestep1}) still holds but also, for later use, the deviation $|B-\bar{B}|^2+|b-\bar{b}|^2$ is super-linearly bounded by $E+D$.

	We first define $b$ by
	\begin{align}\label{ao018}
		b:=\bar{b}+\widetilde{b},
	\end{align}
	taking a vector $\widetilde{b}$ of the form $\ell e_1$ so that $0\in\partial\Omega_0\cap\partial(\Omega_1-b)$; namely, we take $\ell:=g_1(\bar{b}')$, where $g_1$ denotes the graph representation of $\partial\Omega_1$ near the origin, and $\bar{b}'\in\R^{d-1}$ denotes the last $(d-1)$-components of $\bar{b}$.
	Then we have the super-linear estimate\footnote{Rigorously speaking, $[\nabla'g_1]_{\alpha,B_1}$ is interpreted as the supremum of $\frac{|\nabla g_1(x')-\nabla g_1(y')|}{|x'-y'|^\alpha}$ over all $x'\neq y' \in \mathbb{R}^{d-1}$ contained in the projection of $\mathrm{Graph}(g_1)\cap B_1$ to the plane $\{x\cdot e_1=0\}\simeq\mathbb{R}^{d-1}$.}
	\begin{align}\label{eqn:onestep4}
		|\widetilde{b}|^2=|g_1(\bar{b}')|^2\leq([\nabla'g_1]_{\alpha,B_1}|\bar{b}|^{1+\alpha})^2 \stackrel{\eqref{ao014}}{\lesssim} DE^{1+\alpha}\leq(E+D)^{2+\alpha},
	\end{align}
	and, since $|b|^2=|\bar{b}|^2+|\widetilde{b}|^2$, we get in particular
	\begin{align}\label{eqn:onestep10}
		|b|^2\lesssim E+D.
	\end{align}

	We then define $B$ by
	\begin{align}\label{ao017}
		B:=\widetilde{B}\bar{B},
	\end{align}
	seeking a certain matrix $\widetilde{B}$ to make the normals of $\hat{\Omega}_0$ ($=B^{-*}\Omega_0$) and $\hat{\Omega}_1$ ($=B(\Omega_1-b)$) at the origin parallel in the following way.
	Since the normals are transformed by the cofactor matrices under affine changes of variables, the normals of $\bar{B}(\Omega_1-b)$ and $\bar{B}^{-*}\Omega_0$ are parallel to $\bar{B}^{-*}\nu_1(b)$ and $\bar{B}\nu_0(0)$, respectively, and thus in general not parallel to each other.
	Note however that because of the above-mentioned block structure of $\bar{B}$ and the well-preparedness, $\bar{B}^{-*}\nu_1(0)=-\bar{B}^{-*}e_1$ and $\bar{B}\nu_0(0)=-\bar{B}e_1$ are parallel to $-e_1$, so that
	\begin{align}
		&\Big|\frac{\bar{B}^{-*}\nu_1(b)}{|\bar{B}^{-*}\nu_1(b)|}-\frac{\bar{B}\nu_0(0)}{|\bar{B}\nu_0(0)|}\Big|^2=\Big|\frac{\bar{B}^{-*}\nu_1(b)}{|\bar{B}^{-*}\nu_1(b)|}-\frac{\bar{B}^{-*}\nu_1(0)}{|\bar{B}^{-*}\nu_1(0)|}\Big|^2 \nonumber\\
		&\stackrel{\eqref{ao014}}{\lesssim} |\nu_1(b)-\nu_1(0)|^2 \lesssim ([\nabla'g_1]_{\alpha,1}|b|^{\alpha})^2\stackrel{\eqref{eqn:onestep10}}{\lesssim} D(E+D)^\alpha. \label{eqn:onestep9}
	\end{align}
	Hence it suffices to find some matrix $\widetilde{B}$ such that $\widetilde{B}^{-*}\bar{B}^{-*}\nu_1(b)$ ($=B^{-*}\nu_1(b)$) and $\widetilde{B}\bar{B}\nu_0(0)$ (=$B\nu_0(0)$) are parallel, and that is close to the identity in the super-linear sense of
	\begin{align}\label{eqn:onestep5}
		|\widetilde{B}-Id|^2\lesssim (E+D)^{1+\alpha}.
	\end{align}
	We restrict ourselves to constructing a matrix $\widetilde{B}$, the square $\widetilde{A}:=\widetilde{B}^2$ of which is symmetric and to satisfy $\widetilde{A}\frac{\bar{B}\nu_0(0)}{|\bar{B}\nu_0(0)|}=\frac{\bar{B}^{-*}\nu_1(b)}{|\bar{B}^{-*}\nu_1(b)|}$.
	Since within the space of symmetric matrices, in a small neighborhood of the identity matrix, the square root is well defined and a Lipschitz operation, for (\ref{eqn:onestep5}) it is enough to construct such a symmetric matrix with $|\widetilde{A}-Id|^2\lesssim(E+D)^{1+\alpha}$.
	This is easily done: We think of $\mathbb{R}^d$ as orthogonal sum of the space spanned by $\frac{\bar{B}\nu_0(0)}{|\bar{B}\nu_0(0)|}=-e_1$ and its complement, and of $\widetilde{A}$ as a corresponding symmetric block matrix: On the one-dimensional space, $\widetilde{A}$ is defined as required, which by symmetry determines $\widetilde{A}$ up to the $(d-1)$-dimensional diagonal block, where $\widetilde{A}$ is set to be the identity: Since $|\widetilde{A}e_1-e_1|^2\lesssim D(E+D)^\alpha$ by (\ref{eqn:onestep9}), this closeness to the identity translates to the entire matrix $\widetilde{A}$.
	We deduce from (\ref{ao014}) and (\ref{eqn:onestep5}) that $|B-Id|^2\lesssim E+D$, which together with (\ref{eqn:onestep10}) implies (\ref{eqn:onestep1}).

	{\it Step 2: Well-preparedness of $\hat{T}$.}
	The optimality of $\hat{T}$ follows from the general affine invariance.
	We also have $\hat{\lambda}=\lambda|\det B|^{-2}\in[1/4,4]$ since $|\lambda-1|^2\lesssim E\ll1$ by Lemma \ref{lem:value} and since $||\det B|-1|^2\lesssim|B-Id|^2\lesssim E+D\ll1$ by (\ref{eqn:onestep1}).
	In addition, by definition of $B$ and $b$ in Step 1, the open sets $\hat{\Omega}_0$ and $\hat{\Omega}_1$ satisfy the tangency condition (\ref{eqn:tangency}) at the origin, and moreover for any fixed $\theta\in(0,1/2)$ (which we will fix later) the topological condition (\ref{eqn:topological}) holds in $B_{\theta}$ by the $L^\infty$-bounds in Proposition \ref{prop:Linfty}, provided that $E+D\ll_{\theta}1$; this smallness will be satisfied since $E+D\ll_\beta1$ and $\theta$ will only depend on $\beta$ (next to $d$, $\alpha$).

	{\it Step 3: Estimate for $\hat{E}$.}
	We now prove the main estimate (\ref{eqn:onestep2}), provided that $E+D\ll_{\theta}1$ for a given $\theta\in(0,r/4)$ which we fix later, where $r$ denotes the radius in Proposition \ref{prop:harmoniclagrangian}.
	By $|B-Id|\ll_\theta1$, cf.\ (\ref{eqn:onestep1}), we in particular have $|B|\lesssim 1$, $|\det B^{-*}|\lesssim 1$ and $B^*(B_\theta)\subset B_{2\theta}$; hence,
	\begin{align*}
		\hat{E} &=\frac{1}{\theta^2|B_\theta|}\int_{B_\theta}|\hat{T}-\hat{x}|^2\chi_{\hat{\Omega}_0} \\
		& \lesssim \theta^{-(d+2)}\int_{B^*(B_\theta\cap\hat{\Omega}_0)}|B(T-b)-B^{-*}x|^2|\det B^{-*}|\\
		& \lesssim \theta^{-(d+2)}\int_{B_{2\theta}\cap\Omega_0}|T-b-(B^*B)^{-1}x|^2.
	\end{align*}
	In view of the triangle inequality this is bounded above by
	\begin{align}
		&\theta^{-(d+2)}\int_{B_{2\theta}\cap\Omega_0}(|T-x-\nabla\phi|^2+|\nabla\phi-\bar{b}-\bar{A}x|^2 \nonumber\\
		& \qquad\qquad +|x+\bar{A}x-\bar{B}^{-2}x|^2+|\bar{B}^{-2}x-(B^*B)^{-1}x|^2+|\bar{b}-b|^2). \label{eqn:fiveterms}
	\end{align}
	We now estimate these five terms.
	By Proposition \ref{prop:harmoniclagrangian}, for any $\varepsilon\in(0,1)$, if $E+D\ll_{\varepsilon}1$, the first term is bounded as
	$$\theta^{-(d+2)}\int_{B_{2\theta}\cap\Omega_0}|T-x-\nabla\phi|^2 \lesssim \theta^{-(d+2)}(\varepsilon E+\frac{1}{\varepsilon}D),$$
	since $B_{2\theta}\subset B_{r}$.
	Next, in view of the definition \eqref{eqn:bAB}, by Taylor's estimate the second term is bounded as
	\begin{align*}
		\theta^{-(d+2)}\int_{B_{2\theta}\cap\Omega_0}|\nabla\phi-\bar{b}-\bar{A}x|^2
		\lesssim \theta^{-(d+2)}\sup_{B_{2\theta}}|\nabla^3\phi|^2\int_{B_{2\theta}\cap\Omega_0}|x|^{4}
		\lesssim \theta^{2}E,
	\end{align*}
	where in the last estimate, noting that $B_{2\theta}\subset B_{r/2}$, we again used the mean-value property and (\ref{eqn:harmonicDirichlet}) for obtaining $\sup_{B_{2\theta}}|\nabla^3\phi|^2\lesssim\int_{B_r}|\nabla\phi|^2\lesssim E$.
	The third term is bounded as
	\begin{align*}
		\theta^{-(d+2)}\int_{B_{2\theta}\cap\Omega_0}|x+\bar{A}x-\bar{B}^{-2}x|^2
		\lesssim \theta^{-(d+2)}\int_{B_{2\theta}\cap\Omega_0}|\bar{A}|^4|x|^2
		\lesssim  E^2,
	\end{align*}
	because $|\bar{A}|^2\lesssim E\ll1$ by (\ref{ao015}) and hence $|Id+\bar{A}-\bar{B}^{-2}|=|e^{\bar{A}}-Id-\bar{A}|\lesssim |\bar{A}|^2$.
	Concerning the fourth term, noting that all the matrices are regular and their norms are comparable to $1$, cf.\ (\ref{eqn:onestep1}), (\ref{ao014}), and (\ref{eqn:onestep5}), we obtain the bound that
	\begin{align*}
		\theta^{-(d+2)}\int_{B_{2\theta}\cap\Omega_0}|\bar{B}^{-2}x-(B^*B)^{-1}x|^2
		&\lesssim |\bar{B}^{-2}-(B^*B)^{-1}|^2
		%&=|\bar{B}^{-2}(B^*B)(B^*B)^{-1}-\bar{B}^{-2}\bar{B}^2(B^*B)^{-1}|^2\\
		\lesssim |B^*B-\bar{B}^2|^2\\ %\stackrel{\eqref{ao017},$B=B^*$}{=}|\bar{B}\widetilde{B}^*\widetilde{B}\bar{B}-\bar{B}^2|\\
		%&\lesssim |\widetilde{B}^*\widetilde{B}-Id|^2
		&\stackrel{\eqref{ao017}}{\lesssim} |\widetilde{B}-Id|^2 \stackrel{\eqref{eqn:onestep5}}{\lesssim} (E+D)^{1+\alpha}.
	\end{align*}
	Finally, the fifth term is bounded as
	\begin{align*}
		\theta^{-(d+2)}\int_{B_{2\theta}\cap\Omega_0}|\bar{b}-b|^2
		\stackrel{\eqref{ao018}}{\lesssim} \theta^{-2}|\tilde{b}|^2 \stackrel{\eqref{eqn:onestep4}}{\lesssim} \theta^{-2}(E+D)^{2+\alpha}.
	\end{align*}
	In summary, keeping the linear terms with respect to $E$ and $D$ and absorbing all super-linear terms into the one with the smallest exponent $(E+D)^{1+\alpha}$, we find that for any $\theta\in(0,r/4)$, if $E+D\ll_{\theta,\varepsilon}1$, then there is $\bar{C}=\bar{C}(d,\alpha)>0$ such that
	\begin{align*}
		\hat{E}=\theta^{-2}\fint_{B_\theta}|\hat{T}-\hat{x}|^2\chi_{\hat{\Omega}_0} \leq \bar{C}\big((\varepsilon\theta^{-(d+2)}+\theta^{2})E+\varepsilon^{-1}\theta^{-(d+2)}D+ (E+D)^{1+\alpha}\big).
	\end{align*}
	Now, we fix $\theta\ll_{\beta}1$ so small that $\bar{C}\theta^{2}\leq\frac{1}{3}\theta^{2\beta}$; this is possible since $\beta<1$.
	Next, we fix $\varepsilon\ll_{\beta}1$ so small that $\bar{C}\varepsilon\theta^{-(d+2)}\leq\frac{1}{3}\theta^{2\beta}$.
	Finally, thanks to $E+D\ll_{\beta}1$, we have
	$$\bar{C}(E+D)^{1+\alpha}\leq\frac{1}{3}\theta^{2\beta}(E+D),$$
	and thus we conclude that if $E+D\ll_{\beta}1$, then
	\begin{align*}
		\hat{E}=\theta^{-2}\fint_{B_\theta}|\hat{T}-\hat{x}|^2\chi_{\hat{\Omega}_0} &\leq \theta^{2\beta}E+C_\beta D,
	\end{align*}
	which implies (\ref{eqn:onestep2}) since $\theta$ and $\varepsilon$ only depend on $\beta$ (and $d$, $\alpha$).

	{\it Step 4: Estimate for $\hat{D}$.}
	Finally we prove (\ref{eqn:onestep8}).
	By definition, this amounts to show that for $i=0,1$,
	\begin{align}\label{eqn:onestep6}
		[\hat{\nu}_i]_{\alpha,B_\theta}\leq (1+C\sqrt{E+D})[\nu_i]_{\alpha,B_1},
	\end{align}
	where $\nu_i$ and $\hat{\nu}_i$ respectively denote the outer unit normal vectors of $\partial\Omega_i$ and $\partial\hat{\Omega}_i$.
	%Indeed, if (\ref{eqn:onestep6}) holds, then the smallness $E+D\ll_{\beta}1$ implies that $1+C\sqrt{E+D}\leq %\theta^{\beta-\alpha}$ (since $\beta<\alpha$) so that
	%$$[\hat{\nu}_i]_{\alpha,B_\theta}\leq \theta^{\beta-\alpha}[\nu_i]_{\alpha,B_1},$$
	%from which we deduce
	%$$\hat{D}=\theta^{2\alpha}([\hat{\nu}_0]_{\alpha,B_\theta}^2+[\hat{\nu}_1]_{\alpha,B_\theta}^2)\leq \theta^{2\beta}([\nu_0]_{\alpha,B_1}^2+[\nu_1]_{\alpha,B_1}^2)=\theta^{2\beta}D.$$

	We prove (\ref{eqn:onestep6}) only for $i=0$; the case $i=1$ is similar since the translation by $b$ does not change the H\"{o}lder semi-norm of the boundary (up to which part of the boundary is monitored).
	For notational simplicity, let $\hat{B}:=B^{-*}$.
	Since $|\hat{B}-Id|$ and $\theta$ are small, for any $\hat{x},\hat{y}\in\partial\hat{\Omega}_0\cap B_\theta$ there are unique points $x,y\in\partial\Omega_0\cap B_1$, respectively, such that $\hat{x}=\hat{B}x$ and $\hat{y}=\hat{B}y$.
	For such points we have
	$$|\hat{x}-\hat{y}|^{-\alpha}\leq (1+C\sqrt{E+D})|x-y|^{-\alpha},$$
	since
	$$|x-y|\leq |B^*||\hat{x}-\hat{y}|\leq (1+C\sqrt{E+D})|\hat{x}-\hat{y}|.$$
	Thus (\ref{eqn:onestep6}) is reduced to showing
	\begin{align*}
	\Big|\frac{\hat B\nu_0(x)}{|\hat B\nu_0(x)|}-\frac{\hat B\nu_0(y)}{|\hat B\nu_0(y)|}\Big|
	\le(1+C|\hat B-Id|)|\nu_0(x)-\nu_0(y)|,
	\end{align*}
	which by the triangle inequality follows from
	\begin{align*}
	\Big|\big(\frac{\hat B\nu_0(x)}{|\hat B\nu_0(x)|}-\nu_0(x)\big)
	-\big(\frac{\hat B\nu_0(y)}{|\hat B\nu_0(y)|}-\nu_0(y)\big)\Big|
	\lesssim|\hat B-Id||\nu_0(x)-\nu_0(y)|.
	\end{align*}
	This amounts to the statement that the mapping
	\begin{align*}
	\hat B\mapsto\big(\partial B\ni\nu\mapsto\frac{\hat B\nu}{|\hat B\nu|}\in\partial B\big)
	\end{align*}
	is Lipschitz continuous from a neighborhood of $Id$ with values in the Lipschitz transformations
	of the sphere $\partial B$. The latter is a direct consequence of the (local) smoothness of
	$(\hat B,\nu)\mapsto \frac{\hat B\nu}{|\hat B\nu|}$ and the compactness of $\partial B$.
\end{proof}

\begin{remark}\label{rem:nonunitvalue}
	We now briefly explain why we have to allow the values of the initial and target densities to be different.
	The main reason is to obtain a super-linear type estimate in Step 3 of the above proof.
	In fact, if we had only allowed $\lambda=\hat{\lambda}=1$ in (\ref{eqn:onestep0}), then in order to get the marginal condition $\hat{T}\sharp\chi_{\hat{\Omega}_0}=\chi_{\hat{\Omega}_1}$, we need to take $\hat{T}(\hat{x})=\hat{c}B(T(B^*\hat{x})-b)$ with $\hat{c}:=|\det B|^{-2/d}$ (and $\hat{\Omega}_1=\hat{c}B(\Omega_1-b)$).
	In this case we need to replace the fourth term $|\bar{B}^{-2}x-(B^*B)^{-1}x|^2$ in (\ref{eqn:fiveterms}) by $|\bar{B}^{-2}x-(\hat{c}B^*B)^{-1}x|^2$, and thus the super-linear bound of the form $\lesssim(E+D)^{1+\alpha}$ deteriorates into a linear bound $\lesssim E+D$, since we only have $|\hat{c}-1|^2\lesssim|\det B-1|^2\lesssim|B-Id|^2\lesssim E$.
	The linear bound is not sufficient for our purpose.
\end{remark}

\begin{remark}\label{rem:tangency}
	Here is also a good position to observe that the (qualitative) tangency condition \eqref{eqn:tangency} is not restrictive; more precisely, we need not assume $p\in\partial\Omega_0\cap\partial\Omega_1$ and $\nu_0(p)=\nu_1(p)$ if we instead assume a natural quantitative counterpart.
	To observe this fact, for notational simplicity, we may suppose that $\Omega_0$ and $\Omega_1$ are represented by the epigraphs of $g_0,g_1\in C^{1,\alpha}$ in the $e_1$-direction locally in the unit ball $B_1$ (i.e., $p=0$ and $R=1$), respectively, and also $g_0(0)=0$ and $\nabla'g_0(0)=0$ (but $g_1$ is not qualitatively fixed).
	Now, we assume that in addition to $E\ll1$,
	$$\widetilde{D}:=|g_1(0)|^2+|\nabla' g_1(0)|^2+[\nabla'g_0]_\alpha^2+[\nabla'g_1]_\alpha^2\ll1,$$
	where the first two terms yield the natural ``quantitative'' tangency condition (while the last two correspond to the original $D$).
	Then in particular $|\nu_1(0)-\nu_0(0)|^2\lesssim|\nabla' g_1(0)|^2\ll1$, so that there is a symmetric positive definite matrix $A$ such that $A\nu_0(0)=\nu_1(0)$ and $|A-Id|^2\lesssim|\nabla' g_1(0)|^2\ll1$ (see the last part of Step 1 above).
	Then the transformed map $\hat{T}$ defined in \eqref{eqn:onestep0} with $B:=A^\frac{1}{2}$ and $b=g_1(0)e_1$ satisfies the assumption of Theorem \ref{thm:epsilonregularity} (including the tangency condition) at least in a (slightly) smaller ball, so that the desired assertion holds in a smaller ball.
	This can be translated back to the original map in a similar way, where all the terms in $\widetilde{D}$ also appear in the r.h.s.\ of the last linear estimate in Theorem \ref{thm:epsilonregularity}.
\end{remark}

\subsection{Iteration}

For convenience of the readers we give a complete proof, which is however almost parallel to a part of the proof of \cite[Proposition 3.7]{GoldmanOtto},

\begin{proof}[Proof of Proposition \ref{prop:iteration}]
	The assertions (\ref{eqn:iteration1}) and (\ref{eqn:iteration9}) follow if we prove the following discrete version: For any nonnegative integer $k$ there are $\bar{A}_k$ and $\bar{a}_k$ such that
	\begin{align}\label{eqn:iteration2}
		\frac{1}{(\frac{1}{2}\theta^k)^{d+2+2\alpha}}\int_{B_{\frac{1}{2}\theta^k}}|T-(\bar{A}_kx+\bar{a}_k)|^2\chi_{\Omega_0}\lesssim E+D,
	\end{align}
	and
	\begin{align}\label{eqn:iteration8}
		|\bar{A}_k-Id|^2 + \frac{1}{\theta^{2k}}|\bar{a}_k|^2\lesssim E+D,
	\end{align}
	where $\theta\in(0,1)$ is the constant in Proposition \ref{prop:onestep}.
	In what follows we prove this discrete version by using Proposition \ref{prop:onestep} iteratively.

	{\em Step 1: Inductive argument for iteration via the one-step improvement.}
	Set $\Omega_{0,0}:=\Omega_0$, $\Omega_{1,0}:=\Omega_1$, $T_0:=T$, $\lambda_0=\lambda$.
	We demonstrate that we can inductively define $\Omega_{0,k}$, $\Omega_{1,k}$, $T_k$, $\lambda_k$, $B_k$, $b_k$ by applying Proposition \ref{prop:onestep} to $\Omega_{0,{k-1}}$, $\Omega_{1,{k-1}}$, $\lambda_{k-1}$, $T_{k-1}$ with the exponent $\beta:=(\alpha+1)/2>\alpha$; now all constants depending on $\beta$ are universal (i.e., only depending on $d$ and $\alpha$).
	Notice carefully that for this inductive definition we need to inductively verify the smallness hypothesis \eqref{onestep11} for all $k$.

	We now verify by induction that for all $k$ we have not only the hypothesis \eqref{onestep11} but also the stronger key estimate
	\begin{align}\label{eqn:iteration6}
		E_k+D_k \leq C' \theta^{2\alpha k}(E+D) \ (\ll1),
	\end{align}
	where $E_k:=E(\Omega_{0,k},\Omega_{0,k},T_k,\theta^k)$, $D_k:=D(\Omega_{0,k},\Omega_{0,k},\theta^k)$, and $C'=C'(d,\alpha)\in[1,\infty)$ is defined by
	$$C':= C_1'+C_2' \quad \textrm{with} \quad C_1':=\bar{C}\theta^{-2\beta}\sup_{k\geq1}\big(k\theta^{2(\beta-\alpha)(k-1)}\big), \quad C_2':=\prod_{k=1}^{\infty}(1+\theta^{\alpha k}),$$
	and $\bar{C}:=\max\{C_\beta,C\}$ for $C_\beta$ in \eqref{eqn:onestep2} and $C$ in \eqref{eqn:onestep8}.
	We prove (\ref{eqn:iteration6}) by induction, so suppose that it holds for $k\leq K-1$.
	Then we deduce from (\ref{eqn:onestep2}) and (\ref{eqn:onestep8}) that for $k=1,\dots,K$,
	\begin{align}
		E_k &\leq \theta^{2\beta} E_{k-1} + \bar{C}D_{k-1}, \nonumber\\
		D_k &\leq\theta^{2\alpha}(1+\bar{C}\sqrt{E_{k-1}+D_{k-1}})D_{k-1} \stackrel{\eqref{eqn:iteration6}}{\leq} D_{k-1}, \label{eqn:iteration11}
	\end{align}
	which imply
	\begin{align}
		E_k\leq \theta^{2\beta k}E + \bar{C}k\theta^{2\beta(k-1)}D, \quad D_k \leq D. \label{eqn:iteration5}
	\end{align}
	Hence, in particular, from the first item in \eqref{eqn:iteration5} we obtain
	\begin{equation}\label{eqn:iteration12}
		\theta^{-2\alpha K}E_K\leq \theta^{2(\beta-\alpha)K}E + \bar{C}K\theta^{2(\beta-\alpha)(K-1)}\theta^{-2\beta}D \leq E+ C_1'D.
	\end{equation}
	On the other hand, combining the induction hypothesis, that is, (\ref{eqn:iteration6}) for $k=0,\dots,K-1$, with (\ref{eqn:iteration11}) for $k=1,\dots,K$, and noting that $\bar{C}\sqrt{C'(E+D)}\leq1$ since $E+D\ll1$, we have for $k=1,\dots,K$,
	$$\theta^{-2\alpha k}D_k \leq \theta^{-2\alpha(k-1)}(1+\theta^{\alpha(k-1)})D_{k-1},$$
	and therefore
	\begin{equation}\label{eqn:iteration13}
		\theta^{-2\alpha K}D_K \leq \prod_{k=1}^{K-1}(1+\theta^{\alpha k})D \leq C_2' D.
	\end{equation}
	Summing \eqref{eqn:iteration12} and \eqref{eqn:iteration13}, we obtain (\ref{eqn:iteration6}) for $k=K$.

	{\em Step 2: Iteration argument for the Campanato-type estimate.}
	We finally complete the construction of $\bar{A}_k$ and $\bar{a}_k$ satisfying (\ref{eqn:iteration2}) and (\ref{eqn:iteration8}) by iteration.
	We first notice that by (\ref{eqn:onestep1}) and (\ref{eqn:iteration6}),
	\begin{align}
		|B_k-Id|^2 + \frac{1}{\theta^{2k}}|b_k|^2 \lesssim \theta^{2k\alpha}(E+D), \label{eqn:iteration7}
	\end{align}
	and if we define $A_k:=B_k\cdots B_1$ and $a_k:=\sum_{i=1}^{k}B_k\cdots B_ib_i$, then $\Omega_{0,k}=A_k^{-*}\Omega_0$ and $T_k(x)=A_kT(A_k^*x)-a_k$. %$\Omega_{1,k}=\gamma_kA_k\Omega_1-a_k$
	Now the geometric estimate (\ref{eqn:iteration7}) implies that
	\begin{align}
		|A_k-Id|^2 + |a_k|^2\lesssim E+D. \label{eqn:iteration10}
	\end{align}
	In particular, since $E+D\ll 1$, $B_{\frac{1}{2}\theta^k}\subset A_k^*B_{\theta^k}$, $|A_k^{-1}|\sim1$ and $|\det A_k^*|\sim1$.
	Therefore, if we set $\bar{A}_k:=A_k^{-1}A_k^{-*}$ and $\bar{a}_k:=A_k^{-1}a_k$ so that $T-(\bar{A}_kx+\bar{a}_k)=A_k^{-1}(T_k(A_k^{-*}x)-A_k^{-*}x)$, we immediately find that (\ref{eqn:iteration10}) translates into (\ref{eqn:iteration8}), and also have
	\begin{align*}
		\frac{1}{(\frac{1}{2}\theta^k)^{d+2}}\int_{B_{\frac{1}{2}\theta^k}}|T-(\bar{A}_kx+\bar{a}_k)|^2\chi_{\Omega_0} &\lesssim \frac{1}{\theta^{k(d+2)}}\int_{A_k^*B_{\theta^k}}|T-(\bar{A}_kx+\bar{a}_k)|^2\chi_{\Omega_0}\\
		&= \frac{1}{\theta^{k(d+2)}}\int_{B_{\theta^k}}|A_k^{-1}(T_k-x)|^2|\det A_k^*|\chi_{\Omega_{0,k}}\\
		&\lesssim  \frac{1}{\theta^{k(d+2)}}\int_{B_{\theta^k}}|T_k-x|^2\chi_{\Omega_{0,k}} \lesssim E_k,
	\end{align*}
	which, combined with (\ref{eqn:iteration6}), implies (\ref{eqn:iteration2}).
\end{proof}

\subsection{$C^{1,\alpha}$-regularity}

\begin{proof}[Proof of Theorem \ref{thm:epsilonregularity}]
	Without loss of generality we may assume $R=1$ and $p=0$. By Campanato's characterization
	of H\"older spaces, see for instance \cite[Definition 1.5]{Giaquinta},
	it is enough to establish for any $p_0\in B_{1/16}\cap\Omega_0$
	\begin{align}\label{wg07}
	\frac{1}{r^{d+2+2\alpha}}\inf_{A,a}\int_{B_r(p_0)\cap\Omega_0}|T-(Ax+a)|^2\lesssim E+D
	\quad\mbox{for}\;0<r\le\frac{1}{8},
	\end{align}
	where the infimum runs over all matrices $A$ and vectors $a$.
	Setting $R:={\rm dist}(p_0,\partial\Omega_0)$, and noting that $R\le\frac{1}{16}$
	because of $0\in\partial\Omega_0$, we first address the (non-empty) range
	of $R\le r\le\frac{1}{8}$ and then the range $0<r\le R$.

	{\em Step 1: Treatment of the range $R\le r\le\frac{1}{8}$ by boundary regularity.}
	Let $q_0\in\partial\Omega_0$ be such that $|p_0-q_0|=R$;
	noting that $q_0\in B_{1/8}$, we have by definition of $D$ that there exists
	$q_1\in\partial\Omega_1$ such that
	\begin{align}\label{wg01}
	|q_1-q_0|^2+|\nu_1(q_1)-\nu_0(q_0)|^2\lesssim D,
	\end{align}
	and thus a matrix $B$ with
	\begin{align}\label{wg02}
	\det B=1,\quad\frac{B^{-*}\nu_1(p_1)}{|B^{-*}\nu_1(p_1)|}=\frac{B\nu_0(q_0)}{|B\nu_0(q_0)|}
	\quad\mbox{and}\quad|B-Id|^2\lesssim D,
	\end{align}
	so that the two domains
	\begin{align*}
	\hat\Omega_0:=B^{-*}(\Omega_0-q_0)\quad\mbox{and}\quad\hat\Omega_1:=B(\Omega_1-q_1)
	\end{align*}
	satisfy the tangency condition (\ref{eqn:tangency}) with respect to $B_{1/2}$. From the smallness
	conditions in (\ref{wg01}) and (\ref{wg02}) we retain
	\begin{align}\label{wg03}
	|q_1-q_0|^2+|B-Id|^2\lesssim D.
	\end{align}
	If $\hat D$ is defined as $D$, cf.~(\ref{wg12}), with $(\Omega_0,\Omega_1,1)$ replaced
	by $(\hat\Omega_0,\hat{\Omega}_1,\frac{1}{2})$ we clearly have
	\begin{align}\label{wg04}
	\hat D\lesssim D.
	\end{align}
	Because of the structure of the affine transformation, the map
	\begin{align*}
	\hat T(\hat x)=B(T(B^*\hat x+q_0)-q_1)
	\end{align*}
	is optimal for $\chi_{\hat\Omega_0}$ and $\chi_{\hat\Omega_1}$.
	Moreover, if $\hat E$ is defined as
	$E$, cf.~(\ref{wg12}), with $(\Omega_0,\Omega_1,T,1)$ replaced
	by $(\hat\Omega_0,\hat\Omega_1,\hat T,\frac{1}{2})$, we obtain from (\ref{wg03})
	\begin{align}\label{wg05}
	\hat E\lesssim E+D.
	\end{align}
	By (\ref{wg03}), the closeness (\ref{eqn:Linfty1})
	of $T$ to the identity, cf.~Proposition \ref{prop:Linfty}, transfers to $\hat T$,
	again at the expense of a (dyadic) loss in the radius:
	\begin{align*}
	\sup_{\hat\Omega_0\cap B_{1/4}}|\hat T-Id|
	+\sup_{\hat\Omega_1\cap B_{1/4}}|\hat T^{-1}-Id|\ll 1.
	\end{align*}
	Hence also the topological condition (\ref{eqn:topological}) is satisfied with respect to $B_{1/2}$.

	We may thus apply Proposition \ref{prop:iteration} to the effect of
	\begin{align*}
	\min_{\hat A,\hat a}\Big(|\hat A-Id|^2+|\hat a|^2
	+\frac{1}{r^{2+2\alpha}}\fint_{B_{4r}\cap\hat\Omega_0}
	|\hat T-(\hat A\hat x+\hat a)|^2\Big)\lesssim \hat E+\hat D,
	\end{align*}
	which, also appealing to (\ref{wg04}) and (\ref{wg05}),
	and with $(\hat A,\hat a)$ $=\big(BAB^*,B(a+Aq_0-q_1)\big)$ translates back to
	\begin{align*}
	\min_{A,a}\Big(|A-Id|^2+|a|^2+\frac{1}{r^{2+2\alpha}}\fint_{B_{2r}(q_0)\cap\Omega_0}
	|T-(Ax+a)|^2\Big)\lesssim E+D.
	\end{align*}
	Because of $2|p_0-q_0|\le r$, this yields the desired
	\begin{align}\label{wg08}
	\lefteqn{\min_{A,a}\Big(|A-Id|^2+|a|^2}\nonumber\\
	&+\frac{1}{r^{2+2\alpha}}\fint_{B_{r}(p_0)\cap\Omega_0}
	|T-(Ax+a)|^2\Big)\lesssim E+D.
	\end{align}

	{\em Step 2: The symmetry of $A$.}
	In preparation of treating the range $0<r\le R$,
	we argue that in (\ref{wg08}) we may assume that for $r=R$, the infimum is taken over
	symmetric $A$.
	Note that by definition of $R$, the integral extends over $B_R(p_0)$.
	If $A^\mathrm{anti}$ denotes the antisymmetric part of $A$ it suffices to show
	\begin{align}\label{wg11}
	|A^\mathrm{anti}|^2\lesssim\frac{1}{R^2}\fint_{B_R(p_0)}|T-(Ax+a)|^2,
	\end{align}
	which will rely on $T$ being a gradient. Since this is the only property of $T$
	we use, we may without loss of generality assume that $R=1$ and $p_0=0$.
	Fixing $\eta\in C_c^\infty(B_1)$ with $\int\eta=1$ and $i,j=1,\cdots,d$,
	we consider the (curl-like) vector field
	$\xi=-(A_{ji}-A_{ij})(\partial_i\eta e_j-\partial_j\eta e_i)$,
	so that on the one hand, $\int\xi\cdot Ax=(A_{ji}-A_{ij})^2$,
	and on the other hand, $\int\xi\cdot(T-a)=0$. This yields (\ref{wg11}) by the Cauchy-Schwarz inequality.

	{\em Step 3: Treatment of the range $0<r\le R$ by interior regularity.}
	We appeal to (\ref{wg08}) for $r=R$. Let us denote by $(A,a)$ a (near) optimizer;
	by Step 2 we may assume that $A$ is symmetric (and positive definite by its
	closeness to $Id$).
	Hence there exists a (symmetric) matrix $B$ and a vector $p_1$ such that
	\begin{align}
	(Id,0)=(BAB^*,a+Ap_0-p_1)\quad\mbox{and}\nonumber\\
	|B-Id|^2+|p_1-p_0|^2\lesssim E+D.\label{wg09}
	\end{align}
	Then the map
	\begin{align*}
	\hat T(\hat x)=B(T(B^*\hat x+p_0)-p_1)
	\end{align*}
	is optimal for $\chi_{\hat\Omega_0}$ and $(\det B)^{-2}\chi_{\hat\Omega_1}$ where
	\begin{align*}
	\hat\Omega_0:=B^{-*}(\Omega_0-p_0)\quad\mbox{and}\quad\hat\Omega_1:=B(\Omega_1-p_1).
	\end{align*}
	Since by definition of $R$, $B_R(p_0)\subset\Omega_0$ and
	thus $B_{R/2}\subset\hat\Omega_0$ in view of (\ref{wg09}), the minimality (\ref{wg08}) implies
	\begin{align}\label{wg08bis}
	\frac{1}{R^{2+2\alpha}}\fint_{B_{R/2}}
	|\hat T-\hat x|^2\lesssim E+D.
	\end{align}
	We now argue that
	\begin{align}\label{wg10}
	B_{R/4}\subset\hat\Omega_1.
	\end{align}
	Indeed, by $(\det B)^{-2}\chi_{\hat\Omega_1}=\hat{T}\sharp\chi_{\hat\Omega_0}$, by using a cut-off function $\eta$ such that $0\leq\eta\leq1$, $\eta\equiv0$ outside $B_{R/2}$, $\eta\equiv1$ on $B_{R/2-s}$, and $|\nabla\eta|\lesssim 1/s$,
	where $0<s\ll R$ to be optimized later, and by (\ref{wg08bis}) we have
	\begin{align*}
		|B_{R/2}|-(\det B)^{-2}|\hat\Omega_1\cap B_{R/2}| &\leq \big(|B_{R/2}|-\int\eta \big) + \big( \int\eta-\int_{B_{R/2}}\eta\circ\hat{T} \big)\\
		& \lesssim \frac{s}{R}|B_{R/2}| + \frac{1}{s}|B_{R/2}|R(E+D)^{1/2} \ll |B_{R/2}|,
	\end{align*}
	where $s$ is chosen so that $(E+D)^{1/2}\ll s/R \ll 1$,	and hence by (\ref{wg09}) we have $|B_{R/2}\setminus\hat\Omega_1|\ll|B_{R/2}|$. Since in
	view of (\ref{wg09}) we also deduce that $\hat\Omega_1\cap B_{1/2}$ is a connected set
	with Lipschitz boundary of small Lipschitz constant, this implies (\ref{wg10}).
	%By a (much) simpler version of Proposition \ref{prop:Linfty}, this also yields $\hat T^{-1}(B_{R/8})\subset B_{R/4}$ and thus $\int_{B_{R/8}}|\hat T^{-1}-\hat y|^2$ $\le\int_{B_{R/4}}|\hat T-\hat x|^2$.
	From $(\det B)^{-2}\chi_{\hat\Omega_1}$ $=\hat{T}\sharp\chi_{\hat\Omega_0}$ and (\ref{wg08bis})
	we also obtain $|(\det B)^{-2}-1|^2$ $\lesssim R^{2\alpha}(E+D)$ (as Lemma \ref{lem:value}); hence
	by a second (implicit) change of variables in the target space
	(in form of a dilation, as in Step 2 for Proposition \ref{prop:harmoniclagrangian}) we may assume that $\chi_{\hat\Omega_1}=\hat{T}\sharp\chi_{\hat\Omega_0}$
	without affecting (\ref{wg08bis}).

	As a consequence of these observations, we may apply
	the interior regularity theory \cite[Proposition 3.7]{GoldmanOtto}
	(with both densities $\equiv1$)
	in form of $[\nabla\hat T]_{\alpha,B_{R/16}}^2\lesssim E+D$. This
	trivially implies
	\begin{align*}
	\frac{1}{r^{2+2\alpha}}\min_{\hat a,\hat A}\fint_{B_{r}}
	|\hat T-(\hat A\hat x+\hat a)|^2\lesssim E+D\quad\mbox{for}\;r\le R/16,
	\end{align*}
	which translates back into the desired
	\begin{align*}
	\frac{1}{r^{2+2\alpha}}\min_{a,A}\fint_{B_{r}}
	|T-(Ax+a)|^2\lesssim E+D\quad\mbox{for}\;r\le R/32.
	\end{align*}
	The remaining intermediate range $R/32\le r\le R$ follows directly from (\ref{wg08})
	for $r=R$.
\end{proof}

\end{document}